\documentclass[preprint,3p,times]{elsarticle}

\usepackage{amssymb}
\usepackage{amsmath}
\usepackage{amsthm}
\usepackage{mathrsfs}
\usepackage{multirow}
\DeclareMathOperator{\rank}{rank}
\DeclareMathOperator{\St}{St}
\DeclareMathOperator{\GL}{GL}
\DeclareMathOperator{\tr}{tr}
\DeclareMathOperator{\diag}{diag}
\DeclareMathOperator{\vecop}{vec}
\DeclareMathOperator{\ten}{ten}
\newcommand{\matL}{\operatorname{L}}
\newcommand{\R}{\operatorname{R}}
\DeclareMathOperator{\sym}{sym}
\DeclareMathOperator{\ptr}{Tr}
\DeclareMathOperator{\grad}{grad}
\DeclareMathOperator{\Proj}{Proj}
\DeclareMathOperator{\Retr}{Retr}
\DeclareMathOperator*{\argmin}{arg\,min}
\newcommand{\TT}{\mathrm{TT}}
\newcommand{\rankTT}{\operatorname{rank}_{\TT}}
\newtheorem{lemma}{Lemma}
\newdefinition{definition}{Definition}
\newtheorem{corollary}{Corollary}

\newtheorem{proposition}{Proposition}

\usepackage{xurl}
\usepackage[hidelinks]{hyperref}
\usepackage{cleveref}

\AtBeginEnvironment{thebibliography}{\renewcommand{\url}[1]{}\renewcommand{\newline}{}}
\usepackage{graphicx}
\usepackage{epstopdf}
\usepackage[ruled,vlined,linesnumbered]{algorithm2e}
\crefname{algocf}{algorithm}{Algorithms}
\usepackage{circuitikz}
\usepackage{booktabs}
\usepackage{float}
\usepackage{subcaption}

\usepackage{lineno}

\makeatletter
\let\ps@pprintTitle\ps@empty
\makeatother

\begin{document}

\begin{frontmatter}

\title{High-dimensional extreme eigenvalue problems: low-rank tensor parametrization and optimization}

\author[amss,ucas]{Pengfei Hao}
\ead{haopengfei@amss.ac.cn}

\author[cityu]{Renfeng Peng}
\ead{renfpeng@cityu.edu.hk}

\author[gdut]{Yang Kuang}
\ead{ykuang@gdut.edu.cn}

\author[amss]{Bin Gao\corref{cor1}}
\ead{gaobin@lsec.cc.ac.cn}

\cortext[cor1]{Corresponding author.}

\affiliation[amss]{
  organization={State Key Laboratory of Mathematical Sciences,
    Academy of Mathematics and Systems Science, Chinese Academy of Sciences},
  city={Beijing},
  postcode={100190},
  country={China}
}

\affiliation[ucas]{
  organization={University of Chinese Academy of Sciences},
  city={Beijing},
  postcode={100049},
  country={China}
}

\affiliation[cityu]{
  organization={Department of Data Science, City University of Hong Kong},
  city={Hong Kong SAR},
  postcode={999077},
  country={China}
}

\affiliation[gdut]{
  organization={School of Mathematics and Statistics,
    Guangdong University of Technology},
  city={Guangzhou},
  postcode={510006},
  state={Guangdong},
  country={China}
}

\begin{abstract}
High-dimensional extreme eigenvalue problems often arise from molecular vibrational models, electronic structure calculations, and quantum mechanics. Directly solving these problems suffers from the curse of dimensionality. Instead of tackling the problem in high-dimensional ambient space, we reformulate the problem through low-rank tensor formats, which can significantly reduce the computational cost and storage. Specifically, we consider the Rayleigh--Ritz problem on bounded-rank tensors in the tensor train format, which enables a more flexible choice of rank parameters. Moreover, we consider a smooth parametrization for bounded-rank tensors by introducing slack variables, leading to a smooth manifold structure. The original Rayleigh--Ritz problem is therefore transferred to an optimization problem on the manifold. We develop the Riemannian geometry and propose optimization methods for solving extreme eigenvalue problems. In practice, the Kronecker-product structure in Hamiltonian and PDE operators is employed to simplify the computation. 
Numerical experiments on the harmonic oscillator, the Laplace operator, the Schr{\"o}dinger equation, the layered cluster problem, and the Bose--Einstein model demonstrate that the proposed method achieves accuracy comparable to full-space eigensolvers with reduced computation time and avoids storing full vectors. The results also show numerical rank reduction during iteration and robustness when the rank parameter is over-estimated.
In addition, the fixed-point iterations illustrate that the proposed methods are able to serve as an inner eigensolver for solving nonlinear eigenvalue problems.
\end{abstract}

\begin{keyword}
extreme eigenvalue problem \sep Rayleigh--Ritz problem \sep low-rank tensor \sep Riemannian optimization \sep high-dimensional PDE
\end{keyword}

\end{frontmatter}

\section{Introduction}
High-dimensional extreme eigenvalue problems are central tasks in scientific computing, which has wide applications in molecular vibrational models~\cite{RakhubaOseledets2016}, electronic structure calculations~\cite{HackbuschKhoromskijSauterTyrtyshnikov2012}, beam buckling~\cite{johnson1974beam}, quantum mechanics~\cite{KERNER19891} and PDE stability~\cite{Lehoucq2001}. Specifically, the problem aims to find an eigenvector $\mathbf{x}\in\mathbb{R}^{N}$ for a symmetric matrix $\mathbf{H}\in\mathbb{R}^{N\times N}$ with $N:=\prod_{k=1}^d n_k$, where the high-dimensional space $\mathbb{R}^{n_1n_2\cdots n_d}$ naturally arises from, e.g., the discretization of elliptic operators or quantum Hamiltonians on the product space $\mathbb{R}^{n_1}\otimes\mathbb{R}^{n_2}\otimes\cdots\otimes\mathbb{R}^{n_d}$. 
An extreme eigenpair
$(\lambda,\mathbf{x})$ satisfies
\begin{equation}\label{eq:eigen_pairs}
  \mathbf{H}\mathbf{x}=\lambda\mathbf{x},
  \quad \lambda\in\mathbb{R},\quad \mathbf{x}\in\mathbb{R}^{N}\setminus\{\mathbf{0}\},
\end{equation}
and the smallest (largest) eigenvalue can be computed by the Rayleigh--Ritz problem
\begin{equation}\label{eq:directly_EEP}
  \underset{\mathbf{x}\in\mathbb{R}^{N}\setminus\{\mathbf{0}\}}{\min\,(\max)}\ 
  \frac{\mathbf{x}^{\top}\mathbf{H}\mathbf{x}}{\mathbf{x}^{\top}\mathbf{x}}.
\end{equation}
Since computing the largest eigenvalue of $\mathbf H$ is equivalent to computing the smallest eigenvalue of $-\mathbf H$, we only consider the minimization problem throughout this paper.

Existing eigensolvers in the ambient space $\mathbb{R}^{N}$ treat eigenvectors as unstructured vectors and remain limited in exploiting the tensor structure that often appears in real applications. Rather than working in the ambient space, we consider the tensor space $\mathbb{R}^{n_1}\otimes\mathbb{R}^{n_2}\otimes\cdots\otimes\mathbb{R}^{n_d}\simeq\mathbb{R}^{n_1\times n_2\times\cdots\times n_d}$, and reshape an eigenvector $\mathbf{x}\in\mathbb{R}^{N}$ into a tensor $\mathcal{X}\in\mathbb{R}^{n_1\times n_2\times\cdots\times n_d}$. Consequently, the extreme eigenvalue problem~\eqref{eq:eigen_pairs} is equivalent to $\mathscr{H}\mathcal{X}=\lambda\mathcal{X}$, where $\mathscr{H}$ is a self-adjoint linear operator on the tensor space satisfying $\vecop(\mathscr{H}\mathcal{X})=\mathbf{H}\vecop(\mathcal{X})$.

Note that the dimension of tensor space $\mathbb{R}^{n_1\times n_2\times\cdots\times n_d}$, $n_1 n_2 \cdots n_d$, scales exponentially in $d$. Thus, directly solving the Rayleigh--Ritz problem~\eqref{eq:directly_EEP} in tensor space still suffers from the curse of dimensionality. Nevertheless, a key benefit of tensor representations is that low-rank structure can be exploited to represent high-dimensional tensors using substantially fewer parameters. The tensor train~(TT) format~\cite{Oseledets2011} uses low-dimensional factors in $\mathbb{R}^{r_0\times n_1 \times r_1}\times\cdots\times\mathbb{R}^{r_{d-1}\times n_d \times r_d}$ whose storage scales polynomially with dimension to represent a tensor, which significantly reduces the computational cost and storage if the
tensor has low rank $(r_0,r_1,\ldots,r_d)$.
Moreover, low-rank TT tensors can faithfully represent the ground state of local Hamiltonians~\cite{PhysRevB.73.094423}, which suggests that the original full tensor is able to be effectively approximated by low-rank TT tensors. 
Thus, we adopt the low-rank TT representation in the high-dimensional Rayleigh--Ritz problem~\eqref{eq:directly_EEP}. Specifically, we aim to find a low-rank solution:
\begin{equation}\label{eq:final_EEP}
  \underset{\mathcal{X}\neq 0}{\min}\ \rho(\mathcal{X})
  := \frac{\left\langle \mathcal{X},\mathscr{H}\mathcal{X}\right\rangle_{\mathrm{F}}}
  {\left\langle \mathcal{X},\mathcal{X}\right\rangle_{\mathrm{F}}},
  \quad\text{s.\,t.}\quad
  \mathcal{X} \in\mathbb{R}^{n_1\times \cdots\times n_d}_{\leq\mathbf r^{\TT}},
\end{equation}
where the rank parameters $\mathbf r^{\TT}=(r_0,r_1,\ldots,r_d)$ with $r_0=r_d=1$, and
\begin{equation*}
    \mathbb R^{n_1\times\cdots\times n_d}_{\leq\mathbf r^{\TT}}
  :=\{\mathcal X\in\mathbb R^{n_1\times\cdots\times n_d}:\rankTT(\mathcal X)\leq\mathbf r\}
\end{equation*}
is the set of tensors with bounded TT rank; see~\Cref{sec:preliminary} for details. 

It is worth noting that the low-rankness can alternatively be imposed through the set of fixed-rank TT tensors~\cite{steinlechner2016,RakhubaNovikovOseledets2019} $\{\mathcal X\in\mathbb R^{n_1\times\cdots\times n_d}:\rankTT(\mathcal X)=\mathbf r\}$ instead of the set of bounded-rank TT tensors. However, on the one hand, the set of fixed-rank tensors is not closed~\cite{holtz2012}, thus the limit of convergent sequences could lie outside the set. On the other hand, the set of bounded-rank tensors is the union of different fixed-rank sets, which allows a more flexible choice of rank parameters, especially when the rank of target eigenvector is unknown. Therefore, we formulate the Rayleigh--Ritz minimization problem on the set of bounded-rank tensors.

\paragraph{Related work}
The high-dimensional extreme eigenvalue problem can be directly solved by many methods, e.g., Rayleigh--Chebyshev acceleration~\cite{Anderson2010}, locally optimal block preconditioned conjugate gradients (LOBPCG)~\cite{Knyazev2001}, and production-level multi-method eigensolvers such as PRIMME~\cite{StathopoulosMcCombs2010}. These methods are efficient when the operators and eigenvectors are of moderate size. However, as the dimension increases, the computational cost and storage of operators and eigenvectors can grow significantly, thereby rendering the methods computationally intractable.

The tensor train format~\cite{Oseledets2011} provides a compact representation in which the storage scales linearly in dimension and polynomially in mode size and rank. In computational physics, the same representation is known as the matrix product state (MPS)~\cite{Schollwoeck2011}; it is closely related to density-matrix renormalization group (DMRG) methods~\cite{White1992,Schollwoeck2011} and time-dependent variational formulations~\cite{HaegemanLubichOseledetsVandereyckenVerstraete2016}. TT/MPS representations therefore offer a natural way to replace full vectors by structured low-rank tensors in high-dimensional spectral computations.

Existing approaches to high-dimensional eigenvalue problems based on tensor formats include: 1)~classical iterative eigensolvers but using low-rank tensor representations, such as inverse iteration and subspace-based methods; 2)~optimization methods, e.g., gradient methods and alternating updates of the tensor cores. More specifically, a variant of LOBPCG method in the hierarchical Tucker format was proposed in~\cite{KressnerTobler2011}.
Rayleigh quotient minimization was studied in block quantized tensor train format~\cite{Lebedeva2011} based on a conjugate gradient method. Tensor formats and preconditioned low-rank iterations for solving elliptic eigenvalue problems were employed in~\cite{HackbuschKhoromskijSauterTyrtyshnikov2012}. 
Moreover, alternating linear schemes are common techniques for optimization across multiple cores in TT formats~\cite{HoltzRohwedderSchneider2012ALS} and block-TT formats~\cite{DolgovKhoromskijOseledetsSavostyanov2014}. 
Extreme singular values and vectors were estimated also using TT and block-TT formats in~\cite{LeeCichocki2015}. 
Similarly, the DMRG method was an alternating scheme for minimizing the Rayleigh quotient across factors in TT format~\cite{Oseledets2011DMRG}.
Subspace-correction strategies for symmetric eigenvalue problems further combined low-rank tensor representations with eigenspace iterations~\cite{KressnerSteinlechnerUschmajew2014}. 
TT-based inverse iteration and LOBPCG algorithm were used for vibrational spectra~\cite{RakhubaOseledets2016}.
A low-rank Riemannian eigensolver was developed to compute multiple eigenstates of high-dimensional Hamiltonians~\cite{RakhubaNovikovOseledets2019}.
A Riemannian method on the normalized fixed-rank TT manifold was developed for the problems in quantum information theory~\cite{peng2025}.
An inexact subspace represented in low-rank TT tensor formats was used in~\cite{DEKTOR2026286} to solve eigenvalue problems, which was more robust to rank truncation errors.
More recently, block-sparse MPS solvers exploited particle-number conservation and accuracy-adapted ranks for fermionic Schrödinger equations~\cite{BachmayrKramerPfeffer2026}.

\paragraph{Contributions}
In this paper, we consider the set of bounded-rank tensors in the high-dimensional extreme eigenvalue problem, which differs from the existing works. The set of fixed-rank tensors forms a smooth manifold, enabling (Riemannian) optimization methods on manifolds~\cite{AbsilMahonySepulchre+2008,boumal2020introduction}. However, the set of bounded-rank tensors is an algebraic variety containing singular points~\cite{KUTSCHAN2018370}, which poses fundamental geometric difficulty. Consequently, a standard Riemannian algorithm cannot be applied directly to the feasible set of~\eqref{eq:final_EEP}. The smooth bounded-rank parametrization introduced in~\cite{gao2024} constructs a smooth manifold for bounded-rank Tucker and TT tensors, whereas the corresponding algorithm for TT is not concretized. In a similar spirit, we consider a smooth manifold $\mathcal{M}_{\mathbf r}^{\TT}$ and a smooth map $\phi$ for TT tensors such that
\[
  \phi(\mathcal{M}_{\mathbf r}^{\TT})
  =\mathbb{R}^{n_1\times \cdots\times n_d}_{\leq\mathbf r^{\TT}},
\]
which makes it possible to consider the problem in the smooth manifold through this parameterization. Specifically, the Rayleigh--Ritz problem~\eqref{eq:final_EEP} is then solved by minimizing $\rho\circ\phi$ over $\mathcal{M}_{\mathbf r}^{\TT}$. Since the manifold $\mathcal{M}_{\mathbf r}^{\TT}$ is still in a large space, the storage and computation of elements in this space are developed through a TT-cores space $\mathcal{S}_{\mathbf r}^{\TT}$, as summarized in~\Cref{fig:intro}.

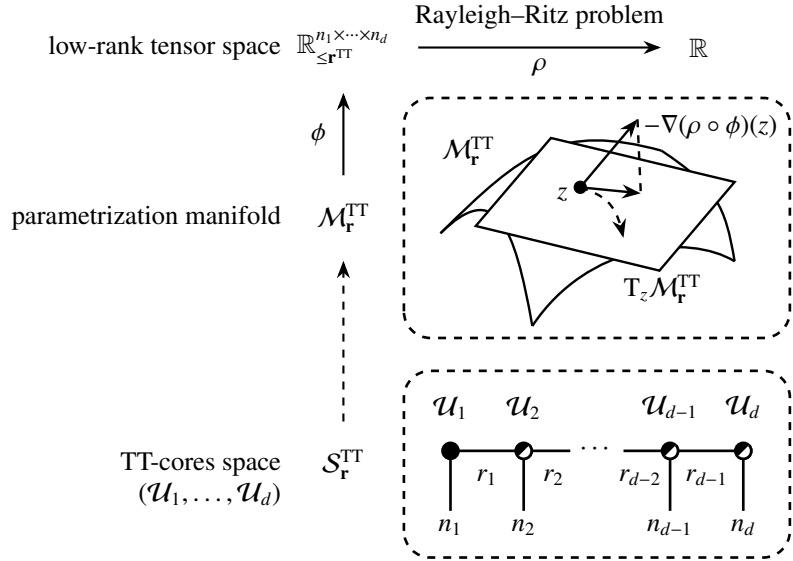
\begin{figure}[htbp]
  \centering
  \begin{circuitikz}
        \tikzset{every node/.style={font=\normalsize}}

                \node [anchor=east] at (6.5,6.5) {low-rank tensor space};
        \node (bounded-space) at (7.2,6.5)
      {$\mathbb{R}_{\leq\mathbf r^{\TT}}^{n_1\times\cdots\times n_d}$};
        \node (real-line) [anchor=west] at (11.65,6.5) {$\mathbb{R}$};
    \draw [line width=0.8pt, ->, >=Stealth]
      ([xshift=0.25cm]bounded-space.east) --
      node [midway, above=0.12cm] {Rayleigh--Ritz problem}
      ([xshift=-0.25cm]real-line.west);
    \node at (9.8,6.25) {$\rho$};
            \node [anchor=east] at (6.5,4.25) {parametrization manifold};
    \node (manifold-space) at (7.2,4.25) {$\mathcal{M}_{\mathbf r}^{\TT}$};

            \draw [line width=0.8pt, ->, >=Stealth]
      ([yshift=0.25cm]manifold-space.north) --
      node [midway, left=0.1cm] {$\phi$}
      ([yshift=-0.25cm]bounded-space.south);

            \node [anchor=east] at (6.5,1.0) {TT-cores space};
    \node [anchor=east] at (6.5,0.6) {$(\mathcal U_1,\ldots,\mathcal U_d)$};
    \node (cores-space) at (7.2,1.0) {$\mathcal{S}_{\mathbf r}^{\TT}$};

            \draw [line width=0.8pt, ->, dashed, >=Stealth]
      ([yshift=0.25cm]cores-space.north) --
      node [midway, left=0.1cm, align=center]
        {}
      ([yshift=-0.25cm]manifold-space.south);
    
                    \begin{scope} [xshift=8cm, yshift=1cm, x=0.88cm, y=0.88cm]
            \draw [line width=1pt, dashed, rounded corners=10] (0,-1.4) rectangle (5.8,1.4);

            \draw [line width=1pt, short] (4.0,0.2) -- (5.1,0.2);
      \draw [line width=1pt, short] (0.7,0.2) -- (0.7,-0.7);
      \draw [line width=1pt, short] (5.1,-0.7) -- (5.1,0.2);
            \node at (0.7,0.85) {$\mathcal{U}_1$};
      \draw [line width=1pt, short] (0.7,0.2) -- (1.8,0.2);
      \draw [line width=1pt, short] (4.0,0.2) -- (4.0,-0.7);
      \node at (1.8,0.85) {$\mathcal{U}_2$};
      \node at (4.0,0.85) {$\mathcal{U}_{d-1}$};
      \node at (5.1,0.85) {$\mathcal{U}_d$};
            \node at (4.0,-0.95) {$n_{d-1}$};
      \node at (5.1,-0.95) {$n_d$};
      \node at (1.25,-0.2) {$r_1$};
      \node at (2.25,-0.2) {$r_{2}$};
      \node at (4.55,-0.2) {$r_{d-1}$};
      \node at (0.7,-0.95) {$n_1$};
      \draw [line width=1pt, short] (1.8,0.2) -- (2.5,0.2);
      \draw [line width=1pt, short] (3.3,0.2) -- (4.0,0.2);
      
            \node at (2.9,0.2) {$\cdots$};
      \draw [line width=1pt, short] (1.8,0.2) -- (1.8,-0.7);
      \node at (1.8,-0.95) {$n_{2}$};
      \node at (3.55,-0.2) {$r_{d-2}$};

            \filldraw [line width=1pt] (0.7,0.2) circle (0.12);
      \filldraw [line width=1pt] (1.8,0.2) circle (0.12);
      \draw [fill=white, line width=1pt] (1.71515,0.11515) arc (-135:45:0.12);
      \filldraw [line width=1pt] (4.0,0.2) circle (0.12);
      \draw [fill=white, line width=1pt] (3.91515,0.11515) arc (-135:45:0.12);
      \filldraw [line width=1pt] (5.1,0.2) circle (0.12);
      \draw [fill=white, line width=1pt] (5.01515,0.11515) arc (-135:45:0.12);
    \end{scope}

                \begin{scope}[xshift=8cm, yshift=4.25cm, x=0.88cm, y=0.88cm]
            \draw [line width=1pt, dashed, rounded corners=10] (0,-1.8) rectangle (5.8,1.8);

            \node at (0.97,1.01) {$\mathcal{M}_{\mathbf r}^{\TT}$};

                  \draw [line width=0.9pt, short] (1.93,-1.62) .. controls (2.76,-0.37) and (4.28,-0.23) .. (4.97,-0.65);
      \draw [line width=0.9pt, short] (1.93,-1.62) .. controls (1.66,-0.10) and (1.38,0.59) .. (0.55,-0.23);
      \draw [line width=0.9pt, short] (0.55,-0.23) .. controls (1.80,1.01) and (2.35,1.42) .. (3.87,1.01);
      \draw [line width=0.9pt, short] (3.87,1.01) .. controls (4.70,0.46) and (4.83,0.18) .. (4.97,-0.65);

            \draw [line width=0.9pt, short, fill=white, fill opacity=1]
        (2.21,1.20) -- (1.08,-0.12) -- (3.84,-0.79) -- (4.97,0.54) -- cycle;
      \node at (3.89,-1.06) {$\mathrm{T}_z \mathcal{M}_{\mathbf r}^{\TT}$};

            \draw [line width=0.9pt, ->, >=Stealth] (2.65,0.46) -- (3.51,1.48);
      \node at (4.61,1.37) {$-\nabla(\rho\circ\phi)(z)$};

            \draw [line width=0.9pt, ->, >=Stealth] (2.65,0.46) -- (3.56,0.37);
      \draw [line width=0.9pt, ->, >=Stealth, dashed] (2.65,0.46) .. controls (3.07,0.32) and (3.20,0.18) .. (3.3,-0.30);
            \draw [line width=0.9pt, dashed] (3.51,1.48) -- (3.56,0.37);
      
            \node at (2.38,0.35) {$z$};
      \filldraw [line width=1pt] (2.65,0.46) circle (0.08);
                        
    \end{scope}

  \end{circuitikz}
  \caption{Overview of low-rank tensor parametrization and optimization. It illustrates how the Rayleigh--Ritz problem on the low-rank tensor space $\mathbb{R}_{\leq\mathbf r^{\TT}}^{n_1\times\cdots\times n_d}$ is reformulated as minimizing $\rho\circ\phi$ on a smooth parametrization manifold $\mathcal{M}_{\mathbf r}^{\TT}$. Riemannian optimization is performed through tangent-space projections and retractions, with points and updates represented in TT-cores space $\mathcal{S}_{\mathbf r}^{\TT}$.}
  \label{fig:intro}
\end{figure}

We formulate the high-dimensional extreme eigenvalue problems as Rayleigh--Ritz optimization over bounded-rank tensor space and propose a smooth parametrization of the bounded-rank TT tensor space.
The geometry of parametrized space is developed, including tangent space, projection, retraction, and vector transport.
The Rayleigh--Ritz problem is then reformulated on the parametrization manifold, of which the Riemannian gradient is derived from the Euclidean gradient.
Moreover, the Kronecker-product structure in the problem is employed to simplify the computation of geometric operations.
Based on the above geometric computations, we propose low-rank TT Riemannian gradient descent (LTT-RGD) and low-rank TT Riemannian conjugate gradient (LTT-RCG) algorithms for extreme eigenvalue problems.
With the Kronecker-product structure, the computational complexity of algorithms is analyzed and shown in~\Cref{tab:complexity}.

We perform numerical experiments of high-dimensional eigenvalue problems on harmonic oscillator, Laplace, Schrödinger, layered cluster, and Bose--Einstein model, and compare with existing methods, LOBPCG and PRIMME.
The numerical results show that the computational cost of the proposed methods scales polynomially with the dimension, rank, and mode size of the problem, which is consistent with the theoretical complexity. 
Additionally, the numerical results validate the effectiveness of LTT-RGD and LTT-RCG in high-dimensional problems compared to other methods while achieving the same level of accuracy.
Moreover, the adaptivity of the numerical ranks in iterations demonstrates robustness to the choice of rank parameters, especially when the rank parameters exceed the target rank.
The experiment on the nonlinear Bose--Einstein model shows that LTT-RCG is able to solve the nonlinear eigenvalue problem with low-rank TT representation.

\paragraph{Organization}
The rest of this paper is organized as follows.
A brief review of tensor notation and the TT decomposition is provided in~\Cref{sec:preliminary}. The parametrization manifold of bounded-rank TT tensors is introduced in~\Cref{sec:desing}, as well as the geometric computations required for Riemannian optimization. In~\Cref{sec:rgd}, we reformulate the Rayleigh--Ritz problem on the parametrization manifold and develop the Riemannian gradient descent and conjugate gradient methods, along with computational complexity. Numerical results for high-dimensional extreme eigenvalue problems are reported in~\Cref{sec:experiment}. Finally, we draw conclusions.

\section{Preliminaries} \label{sec:preliminary}
This section introduces the matrix and tensor notation used throughout the
paper and reviews the tensor train decomposition~\cite{Oseledets2011}. Tensor-network diagrams are used to
illustrate the main operations, following the conventions in~\cite{holtz2012}.
Additional background on low-rank tensor formats is available
from~\cite{Uschmajew2020}.

\subsection{Matrix and tensor notation} \label{sec:matrix}
Let $m,n,r$ be positive integers with $r\leq\min\{m,n\}$. The
Stiefel manifold and the general linear group are denoted by
$\St(r,n):=\{\mathbf U\in\mathbb R^{n\times r}:\mathbf U^\top\mathbf U
=\mathbf I_r\}$ and $\GL(n):=\{\mathbf A\in\mathbb R^{n\times n}:
\rank(\mathbf A)=n\}$, respectively. The space of symmetric matrices is denoted by
$\mathbb S^n:=\{\mathbf A\in\mathbb R^{n\times n}:\mathbf A^\top=\mathbf A\}$.
Moreover, the symmetric part of a square matrix is
$\sym(\mathbf A):=\tfrac12(\mathbf A+\mathbf A^\top)$.

The following elementary properties and blockwise operations will be used to construct smooth curves in the TT parameter space and to derive the geometry of the parametrization manifold. First, for $m>n$, the set of full-column-rank matrices
$\{\mathbf A\in\mathbb R^{m\times n}:\rank(\mathbf A)=n\}$ is
open~\cite[Section~3.1.5]{AbsilMahonySepulchre+2008},
which ensures that sufficiently small perturbations preserve the full-rank condition.
We next introduce two operations on partitioned matrices that simplify
subsequent matrix products and reduce the computational cost. Let
$\mathbf A\in\mathbb R^{(mk)\times (nk)}$ be partitioned as
$\mathbf A=[\mathbf A_{ij}]_{i=1,j=1}^{m,n}$, where
$\mathbf A_{ij}\in\mathbb R^{k\times k}$. For a matrix
$\mathbf C\in\mathbb R^{k\times k}$, the \emph{partial trace}
and \emph{partial inner product} are defined by
\[
  \ptr_k(\mathbf A):=\bigl(\tr(\mathbf A_{ij})\bigr)_{i=1,j=1}^{m,n},
  \qquad
  \langle\mathbf A,\mathbf C\rangle^k
  :=\bigl(\langle\mathbf A_{ij},\mathbf C\rangle_{\mathrm{F}}\bigr)_{i=1,j=1}^{m,n}.
\]
These contractions satisfy
\[
  \langle\mathbf A,\mathbf D\otimes\mathbf I_k\rangle_{\mathrm{F}}
  =\langle\ptr_k(\mathbf A),\mathbf D\rangle_{\mathrm{F}},
  \qquad
  \ptr_k\bigl(\mathbf A(\mathbf B\otimes\mathbf C)\bigr)
  =\langle\mathbf A,\mathbf C^\top\rangle^k\mathbf B,
\]
for $\mathbf D\in\mathbb R^{m\times n}$,
$\mathbf B\in\mathbb R^{n\times n}$, and
$\mathbf C\in\mathbb R^{k\times k}$. The first identity follows by expanding the
Frobenius inner product blockwise:
$\langle\mathbf A,\mathbf D\otimes\mathbf I_k\rangle_{\mathrm{F}}
=\sum_{i,j}d_{ij}\tr(\mathbf A_{ij})$.
For the second identity, the $(i,j)$-th entry of the left-hand side is
$\tr(\sum_{\ell=1}^n b_{\ell j}\mathbf A_{i\ell}\mathbf C)
=\sum_{\ell=1}^n b_{\ell j}
\langle\mathbf A_{i\ell},\mathbf C^\top\rangle_{\mathrm{F}}$, which is the
$(i,j)$-th entry of $\langle\mathbf A,\mathbf C^\top\rangle^k\mathbf B$.

The partial trace also preserves positive semidefiniteness. Indeed, if
$\mathbf A\in\mathbb R^{(mk)\times (mk)}$ is positive semidefinite, we can write
$\mathbf A=\mathbf U\mathbf U^\top$ and partition $\mathbf U$ into block rows
$\mathbf U_i\in\mathbb R^{k\times (mk)}$. One has
$\bigl(\ptr_k(\mathbf A)\bigr)_{ij}
=\tr(\mathbf U_i\mathbf U_j^\top)$; hence
$\ptr_k(\mathbf A)$ is a Gram matrix and is therefore positive semidefinite.

A \emph{tensor} is a multidimensional array
$\mathcal X\in\mathbb R^{n_1\times\cdots\times n_d}$. The integer $d$ is
the \emph{order}, and each dimension is called a \emph{mode}. The entry with
multi-index $(i_1,\ldots,i_d)$ is denoted by
$\mathcal X(i_1,\ldots,i_d)$. For tensors $\mathcal X$ and $\mathcal Y$ of
the same size, the Frobenius inner product and norm are
\[
  \langle\mathcal X,\mathcal Y\rangle_{\mathrm{F}}
  :=\sum_{i_1=1}^{n_1}\cdots\sum_{i_d=1}^{n_d}
  \mathcal X(i_1,\ldots,i_d)\mathcal Y(i_1,\ldots,i_d),
  \qquad
  \|\mathcal X\|_{\mathrm{F}}
  :=\langle\mathcal X,\mathcal X\rangle_{\mathrm{F}}^{1/2}.
\]
Unless stated
otherwise, all matrix and tensor inner products and norms are Frobenius inner
products $\langle\cdot,\cdot\rangle_{\mathrm{F}}$ and norms $\|\cdot\|_{\mathrm{F}}$.
A tensor is represented graphically by a node with one edge per mode.
Contracted modes are indicated by connecting the corresponding edges.
\Cref{fig:tensor-graph} shows a third-order tensor, a matrix--vector product,
and a matrix--tensor--matrix product.
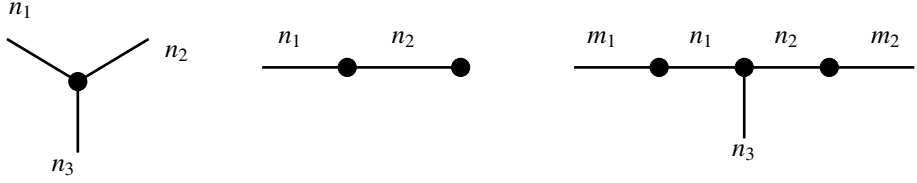
\begin{figure}[htbp]
  \centering
    \begingroup
        \let\Large\normalsize\let\LARGE\normalsize\let\huge\normalsize
    \tikzset{every node/.style={font=\normalsize}}
        \begin{circuitikz}[x=0.75cm,y=0.75cm]
      \tikzset{every node/.style={font=\normalsize}}
      \draw [ fill={rgb,255:red,0; green,0; blue,0} ] (14,10.75) circle[radius=0.12cm];
      \draw [line width=1pt, short] (9.25,10.5) -- (10.5,11.25);
      \draw [line width=1pt, short] (9.25,10.5) -- (8,11.25);
      \draw [line width=1pt, short] (9.25,10.5) -- (9.25,9.25);
      \node [font=\Large] at (8.25,11.75) {$n_1$};
      \node [font=\Large] at (11,11) {$n_2$};
      \node [font=\Large] at (9,9) {$n_3$};
      \draw [ fill={rgb,255:red,0; green,0; blue,0} ] (9.25,10.5) circle[radius=0.12cm];
      \draw [ fill={rgb,255:red,0; green,0; blue,0} ] (16,10.75) circle[radius=0.12cm];
      \draw [line width=1pt, short] (12.5,10.75) -- (14,10.75);
      \draw [line width=1pt, short] (14,10.75) -- (16,10.75);
      \node [font=\Large] at (13,11.25) {$n_1$};
      \node [font=\Large] at (15,11.25) {$n_2$};
      \draw [line width=1pt, short] (21,10.75) -- (22.5,10.75);
      \draw [line width=1pt, short] (21,10.75) -- (19.5,10.75);
      \draw [line width=1pt, short] (21,10.75) -- (21,9.5);
      \node [font=\Large] at (20.25,11.25) {$n_1$};
      \node [font=\Large] at (21.75,11.25) {$n_2$};
      \node [font=\Large] at (21,9.25) {$n_3$};
      \draw [ fill={rgb,255:red,0; green,0; blue,0} ] (21,10.75) circle[radius=0.12cm];
      \draw [ fill={rgb,255:red,0; green,0; blue,0} ] (19.5,10.75) circle[radius=0.12cm];
      \draw [ fill={rgb,255:red,0; green,0; blue,0} ] (22.5,10.75) circle[radius=0.12cm];
      \draw [line width=1pt, short] (19.5,10.75) -- (18,10.75);
      \draw [line width=1pt, short] (24,10.75) -- (22.5,10.75);
      \node [font=\Large] at (18.5,11.25) {$m_1$};
      \node [font=\Large] at (23.5,11.25) {$m_2$};
    \end{circuitikz}
    \endgroup
  \caption{Tensor-network diagrams of a third-order tensor, a matrix--vector product, and a matrix--tensor--matrix product.}
  \label{fig:tensor-graph}
\end{figure}

For $\mathcal X\in\mathbb R^{n_1\times\cdots\times n_d}$, let
$N=\prod_{k=1}^d n_k$. The vectorization
$\vecop(\mathcal X)\in\mathbb R^N$ of $\mathcal{X}$ uses the column-major index
$j:=1+\sum_{k=1}^d(i_k-1)\prod_{\ell=1}^{k-1}n_\ell$, and its inverse, called the tensorization, is denoted by $\ten(\cdot)$. The $k$-th unfolding
$\mathbf X^{\langle k\rangle}\in
\mathbb R^{(n_1\cdots n_k)\times(n_{k+1}\cdots n_d)}$ is defined by
$\mathbf X^{\langle k\rangle}(i_1,\ldots,i_k;i_{k+1},\ldots,i_d)
:=\mathcal X(i_1,\ldots,i_d)$, where the semicolon separates the row and
column multi-indices. Following~\cite{KUTSCHAN2018370}, the \emph{left} and
\emph{right matricizations} are respectively 
\begin{equation*}
    \matL(\mathcal X):=\mathbf X^{\langle d-1\rangle},\quad
\R(\mathcal X):=\mathbf X^{\langle1\rangle}.
\end{equation*}

The \emph{tensor train product}~\cite{KUTSCHAN2018370} of
$\mathcal X\in\mathbb R^{n_1\times\cdots\times n_k\times m}$ and
$\mathcal Y\in\mathbb R^{m\times n_{k+1}\times\cdots\times n_d}$ is
$\mathcal X\mathcal Y:=\ten\bigl(\matL(\mathcal X)\R(\mathcal Y)\bigr)
\in\mathbb R^{n_1\times\cdots\times n_d}$.
For matrices $\mathbf A\in\mathbb R^{m\times n}$ and
$\mathbf B\in\mathbb R^{p\times q}$, the \emph{Kronecker product}
$\mathbf A\otimes\mathbf B\in\mathbb R^{mp\times nq}$ is
\[
  \mathbf A\otimes\mathbf B:=\begin{pmatrix}
    a_{11}\mathbf B&a_{12}\mathbf B&\cdots&a_{1n}\mathbf B\\
    a_{21}\mathbf B&a_{22}\mathbf B&\cdots&a_{2n}\mathbf B\\
    \vdots&\vdots&\ddots&\vdots\\
    a_{m1}\mathbf B&a_{m2}\mathbf B&\cdots&a_{mn}\mathbf B
  \end{pmatrix}.
\]

\subsection{Tensor train decomposition}
The TT decomposition was introduced in~\cite{Oseledets2011} and the same format in computational physics
is commonly called a matrix product state.
\begin{definition}[tensor train decomposition]
  The \emph{tensor train} decomposition of
  $\mathcal X\in\mathbb R^{n_1\times\cdots\times n_d}$ is
  \begin{equation*}
      \mathcal X(i_1,\ldots,i_d):=
      \sum_{\ell_1=1}^{r_1}\cdots\sum_{\ell_{d-1}=1}^{r_{d-1}}
      \mathcal U_1(i_1,\ell_1)
      \mathcal U_2(\ell_1,i_2,\ell_2)\cdots
      \mathcal U_{d-1}(\ell_{d-2},i_{d-1},\ell_{d-1})
      \mathcal U_d(\ell_{d-1},i_d),
  \end{equation*}
  where $\mathcal U_k\in\mathbb R^{r_{k-1}\times n_k\times r_k}$,
  $k=1,\ldots,d$, are the \emph{TT cores}, and the positive integers
  $r_0,\ldots,r_d$ satisfy $r_0=r_d=1$.
\end{definition}
The boundary cores $\mathcal U_1$ and $\mathcal U_d$ are matrices when the
singleton rank modes are omitted; all interior cores $\mathcal U_k,\ k=2,\dots,d-1$ are third-order
tensors.
The nested summation can equivalently be written as a matrix product. Define
the $r_{k-1}\times r_k$ matrix slice
$\mathbf U_k(i_k):=\mathcal U_k(:,i_k,:)$. Then
\[
  \mathcal X(i_1,\ldots,i_d)
  =\mathbf U_1(i_1)\mathbf U_2(i_2)\cdots\mathbf U_d(i_d),
  \qquad
  \mathcal X=\mathcal U_1\mathcal U_2\cdots\mathcal U_d,
\]
where the second equality uses the TT-product notation. This representation
is illustrated in~\Cref{fig:tt}.
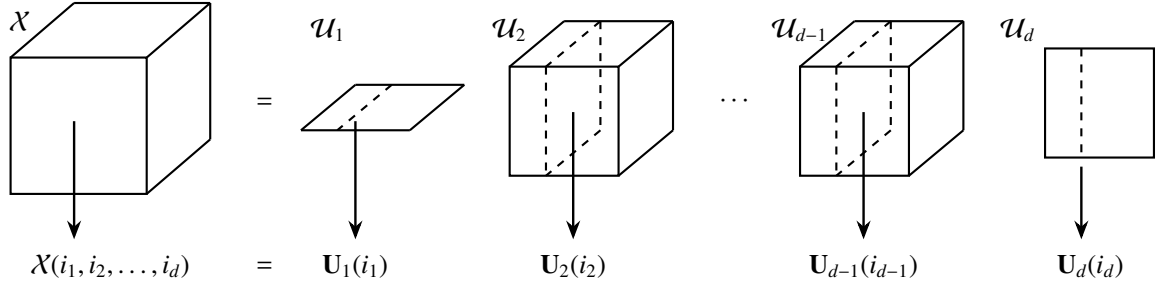
\begin{figure}[htbp]
  \centering
  \begingroup
    \let\Large\normalsize\let\LARGE\normalsize\let\huge\normalsize
  \tikzset{every node/.style={font=\normalsize}}
    \begin{circuitikz}[x=0.48cm,y=0.48cm]
  \tikzset{every node/.style={font=\normalsize}}
  \draw [ line width=0.8pt ] (1.25,2.25) rectangle (5,-1.5);
  \draw [line width=0.8pt, short] (5,2.25) -- (6.75,3.75);
  \node [font=\LARGE] at (2,3) {};
  \node [font=\LARGE] at (2.75,2.5) {};
  \draw [line width=0.8pt, short] (1.25,2.25) -- (3,3.75);
  \draw [line width=0.8pt, short] (5,-1.5) -- (6.75,0);
  \draw [line width=0.8pt, short] (3,3.75) -- (6.75,3.75);
  \draw [line width=0.8pt, short] (6.75,3.75) -- (6.75,0);
  \draw [ line width=0.8pt ] (15,2) rectangle (18,-1);
  \node [font=\LARGE] at (15.75,2.75) {};
  \node [font=\LARGE] at (16.25,2.25) {};
  \draw [line width=0.8pt, short] (15,2) -- (16.5,3.25);
  \draw [line width=0.8pt, short] (18,2) -- (19.5,3.25);
  \draw [line width=0.8pt, short] (18,-1) -- (19.5,0.25);
  \draw [line width=0.8pt, short] (16.5,3.25) -- (19.5,3.25);
  \draw [line width=0.8pt, short] (19.5,3.25) -- (19.5,0.25);
  \node [font=\huge] at (8.25,1) {=};
  \node [font=\LARGE] at (10,1) {};
  \node [font=\LARGE] at (10.5,0.5) {};
  \draw [line width=0.8pt, short] (9.25,0.25) -- (10.75,1.5);
  \draw [line width=0.8pt, short] (12.25,0.25) -- (13.75,1.5);
  \draw [line width=0.8pt, short] (10.75,1.5) -- (13.75,1.5);
  \node [font=\LARGE] at (24,2.75) {};
  \node [font=\LARGE] at (24.5,2.25) {};
  \node [font=\LARGE] at (12.25,0.25) {};
  \node [font=\LARGE] at (12.75,1) {};
  \draw [line width=0.8pt, short] (9.25,0.25) -- (12.25,0.25);
  \node [font=\LARGE] at (23.5,2.75) {};
  \node [font=\LARGE] at (24,2.25) {};
  \node [font=\LARGE] at (10,3) {$\mathcal{U}_1$};
  \node [font=\LARGE] at (15,3) {$\mathcal{U}_2$};
  \node [font=\LARGE] at (23,3) {$\mathcal{U}_{d-1}$};
  \node [font=\LARGE] at (29,3) {$\mathcal{U}_d$};
  \node [font=\LARGE] at (21.25,1) {$\cdots$};
  \node [font=\LARGE] at (1.5,3.25) {$\mathcal{X}$};
  \draw [line width=0.8pt, dashed] (11.75,1.5) -- (10.25,0.25);
  \draw [line width=0.8pt, dashed] (17.5,3.25) -- (16,2);
  \draw [line width=0.8pt, dashed] (16,2) -- (16,-1);
  \draw [line width=0.8pt, dashed] (17.5,3.25) -- (17.5,0.25);
  \node [font=\LARGE] at (17.25,2) {};
  \node [font=\LARGE] at (17.25,2) {};
  \node [font=\LARGE] at (17.25,2) {};
  \node [font=\LARGE] at (17.25,2) {};
  \node [font=\LARGE] at (17.25,2) {};
  \draw [line width=0.8pt, dashed] (17.5,0.25) -- (16,-1);
  \draw [ line width=0.8pt ] (23,2) rectangle (26,-1);
  \node [font=\LARGE] at (23.75,2.75) {};
  \node [font=\LARGE] at (24.25,2.25) {};
  \draw [line width=0.8pt, short] (23,2) -- (24.5,3.25);
  \draw [line width=0.8pt, short] (26,2) -- (27.5,3.25);
  \draw [line width=0.8pt, short] (26,-1) -- (27.5,0.25);
  \draw [line width=0.8pt, short] (24.5,3.25) -- (27.5,3.25);
  \draw [line width=0.8pt, short] (27.5,3.25) -- (27.5,0.25);
  \draw [line width=0.8pt, dashed] (25.5,3.25) -- (24,2);
  \draw [line width=0.8pt, dashed] (24,2) -- (24,-1);
  \draw [line width=0.8pt, dashed] (25.5,3.25) -- (25.5,0.25);
  \node [font=\LARGE] at (25.25,2.25) {};
  \node [font=\LARGE] at (25.25,2.25) {};
  \node [font=\LARGE] at (25.25,2.25) {};
  \node [font=\LARGE] at (25.25,2) {};
  \node [font=\LARGE] at (25.25,2.25) {};
  \draw [line width=0.8pt, dashed] (25.5,0.25) -- (24,-1);
  \draw [ line width=0.8pt ] (29.75,2.5) rectangle (32.75,-0.5);
  \node [font=\LARGE] at (29.75,2.75) {};
  \node [font=\LARGE] at (30.25,2.25) {};
  \draw [line width=0.8pt, dashed] (30.75,2.5) -- (30.75,-0.5);
  \node [font=\LARGE] at (32,2.5) {};
  \node [font=\LARGE] at (32,2.5) {};
  \node [font=\LARGE] at (32,2.5) {};
  \node [font=\LARGE] at (32,2.5) {};
  \node [font=\LARGE] at (32,2.5) {};
  \node [font=\huge] at (4,-3.5) {$\mathcal{X}(i_1,i_2,\dots,i_d)$};
  \draw [line width=0.9pt, ->, >=Stealth] (3,0.5) -- (3,-2.75);
  \node [font=\huge] at (8.25,-3.5) {=};
  \node [font=\huge] at (10.75,-3.5) {$\mathbf U_1(i_1)$};
  \node [font=\huge] at (16.75,-3.5) {$\mathbf U_2(i_2)$};
  \node [font=\huge] at (24.75,-3.5) {$\mathbf U_{d-1}(i_{d-1})$};
  \node [font=\huge] at (31,-3.5) {$\mathbf U_d(i_d)$};
  \draw [line width=0.9pt, ->, >=Stealth] (10.75,0.5) -- (10.75,-2.75);
  \draw [line width=0.9pt, ->, >=Stealth] (16.75,0.75) -- (16.75,-2.75);
  \draw [line width=0.9pt, ->, >=Stealth] (24.75,0.75) -- (24.75,-2.75);
  \draw [line width=0.9pt, ->, >=Stealth] (30.75,-0.75) -- (30.75,-2.75);
  \end{circuitikz}
  \endgroup
  \caption{Tensor train decomposition and the corresponding matrix-slice representation.}
  \label{fig:tt}
\end{figure}

The TT rank of tensor is determined by the ranks of the $k$-th unfoldings.
The \emph{TT rank} of $\mathcal X$ is the vector
  \begin{equation*}
    \rankTT(\mathcal X)
    :=\bigl(1,\rank(\mathbf X^{\langle1\rangle}),\ldots,
    \rank(\mathbf X^{\langle d-1\rangle}),1\bigr)
    =(r_0,r_1,\ldots,r_{d-1},r_d),
  \end{equation*}
with $r_0=r_d=1$.
The full-rank conditions for a minimal TT decomposition
imply the admissibility bounds~\cite[Sections~2.4--2.5 and
Lemma~3.3]{holtz2012}
\begin{equation}\label{eq:tt-rank-admissibility}
  r_k\leq\min\{r_{k-1}n_k,n_{k+1}r_{k+1}\},
  \qquad k=1,\ldots,d-1.
\end{equation}

The \emph{left and right interface matrices}
$\mathbf X_{\leq k}\in\mathbb R^{(n_1\cdots n_k)\times r_k}$ and
$\mathbf X_{\geq k+1}\in
\mathbb R^{(n_{k+1}\cdots n_d)\times r_k}$ collect the products of the TT
cores to the left and right of index $k$, respectively, and are defined by
\[
  \begin{aligned}
    \mathbf X_{\leq k}(i_1,\ldots,i_k;\,\cdot\,)
    &:=\mathbf U_1(i_1)\cdots\mathbf U_k(i_k),\\
    \mathbf X_{\geq k+1}(i_{k+1},\ldots,i_d;\,\cdot\,)
    &:=\bigl(\mathbf U_{k+1}(i_{k+1})\cdots\mathbf U_d(i_d)\bigr)^\top.
  \end{aligned}
\]
It holds that 
\[
  \begin{aligned}
    \mathbf X_{\leq k}&=(\mathbf I_{n_k}\otimes\mathbf X_{\leq k-1})\matL(\mathcal U_k),\\
    \mathbf X_{\geq k+1}&=(\mathbf X_{\geq k+2}\otimes\mathbf I_{n_{k+1}})\R(\mathcal U_{k+1})^\top,
  \end{aligned}
\]
with $\mathbf X_{\leq1}=\matL(\mathcal U_1)$ and
$\mathbf X_{\geq d}=\R(\mathcal U_d)^\top$. Moreover,
$\mathbf X^{\langle k\rangle}=\mathbf X^{}_{\leq k}\mathbf X_{\geq k+1}^\top$.
The corresponding tensor-network representation is shown
in~\Cref{fig:graph-tt}.
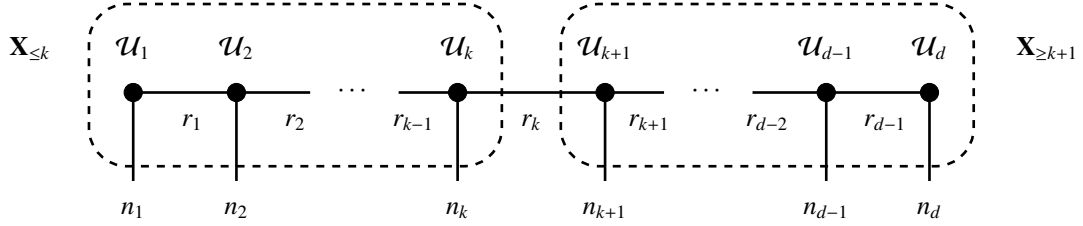
\begin{figure}[htbp]
  \centering
    \begingroup
        \let\Large\normalsize\let\LARGE\normalsize\let\huge\normalsize
    \tikzset{every node/.style={font=\normalsize}}
        \begin{circuitikz}[x=0.78cm,y=0.78cm]
      \tikzset{every node/.style={font=\normalsize}}
      \draw [line width=1pt, short] (10,5.75) -- (11.75,5.75);
      \draw [line width=1pt, short] (-1.75,5.75) -- (-1.75,4.25);
      \node [font=\LARGE] at (8.75,5.25) {};
      \draw [ fill={rgb,255:red,0; green,0; blue,0} ] (6.25,5.75) circle[radius=0.12cm];
      \draw [ fill={rgb,255:red,0; green,0; blue,0} ] (3.75,5.75) circle[radius=0.12cm];
      \draw [ fill={rgb,255:red,0; green,0; blue,0} ] (11.75,5.75) circle[radius=0.12cm];
      \draw [line width=1pt, short] (11.75,4.25) -- (11.75,5.75);
      \node [font=\Large] at (-1.75,6.5) {$\mathcal{U}_1$};
      \node [font=\LARGE] at (11.5,5.25) {};
      \draw [line width=1pt, short] (-1.75,5.75) -- (0,5.75);
      \draw [ fill={rgb,255:red,0; green,0; blue,0} ] (-1.75,5.75) circle[radius=0.12cm];
      \node [font=\LARGE] at (-1.5,4.25) {};
      \draw [line width=1pt, short] (0,5.75) -- (0,4.25);
      \draw [line width=1pt, short] (10,5.75) -- (10,4.25);
      \node [font=\Large] at (0,6.5) {$\mathcal{U}_2$};
      \node [font=\Large] at (10,6.5) {$\mathcal{U}_{d-1}$};
      \node [font=\Large] at (11.75,6.5) {$\mathcal{U}_d$};
      \node [font=\Large] at (-1.75,3.75) {$n_1$};
      \node [font=\Large] at (0,3.75) {$n_2$};
      \node [font=\Large] at (10,3.75) {$n_{d-1}$};
      \node [font=\Large] at (11.75,3.75) {$n_d$};
      \node [font=\Large] at (-0.75,5.25) {$r_1$};
      \node [font=\Large] at (1,5.25) {$r_2$};
      \node [font=\Large] at (7,5.25) {$r_{k+1}$};
      \node [font=\Large] at (11,5.25) {$r_{d-1}$};
      \node [font=\Large] at (2,5.75) {$\cdots$};
      \draw [ line width=1pt , rounded corners = 14.4, dashed] (-2.5,7.25) rectangle  (4.5,4.5);
      \draw [ line width=1pt , rounded corners = 14.4, dashed] (5.5,7.25) rectangle  (12.5,4.5);
      \node [font=\Large] at (-3.5,6.5) {$\mathbf X_{\leq k}$};
      \node [font=\Large] at (13.75,6.5) {$\mathbf X_{\geq k+1}$};
      \draw [line width=1pt, short] (3.75,5.75) -- (6.25,5.75);
      \node [font=\LARGE] at (5.75,5.25) {};
      \draw [line width=1pt, short] (2.75,5.75) -- (3.75,5.75);
      \draw [line width=1pt, short] (0,5.75) -- (1.25,5.75);
      \draw [ fill={rgb,255:red,0; green,0; blue,0} ] (0,5.75) circle[radius=0.12cm];
      \node [font=\LARGE] at (0.25,4.25) {};
      \draw [line width=1pt, short] (3.75,5.75) -- (3.75,4.25);
      \node [font=\Large] at (3.75,3.75) {$n_k$};
      \draw [line width=1pt, short] (6.25,5.75) -- (7.25,5.75);
      \draw [line width=1pt, short] (8.75,5.75) -- (10,5.75);
      \draw [ fill={rgb,255:red,0; green,0; blue,0} ] (10,5.75) circle[radius=0.12cm];
      \node [font=\Large] at (8,5.75) {$\cdots$};
      \draw [line width=1pt, short] (6.25,5.75) -- (6.25,4.25);
      \node [font=\Large] at (6.25,3.75) {$n_{k+1}$};
      \node [font=\Large] at (3.75,6.5) {$\mathcal{U}_k$};
      \node [font=\Large] at (6.25,6.5) {$\mathcal{U}_{k+1}$};
      \node [font=\Large] at (9,5.25) {$r_{d-2}$};
      \node [font=\Large] at (3,5.25) {$r_{k-1}$};
      \node [font=\Large] at (5,5.25) {$r_{k}$};
      \end{circuitikz}
    \endgroup
  \caption{TT decomposition with the left and right interface matrices obtained by grouping modes $1,\ldots,k$ and $k+1,\ldots,d$, respectively.}
  \label{fig:graph-tt}
\end{figure}

\section{Geometry of low-rank tensor parametrization}
\label{sec:desing}

This section develops the geometric computations required by the algorithms.
We consider the smooth parametrization of the bounded-rank TT tensor space. 
Then the geometry of the parametrized space is developed for Riemannian optimization, including tangent space, projection, retraction, and vector transport.

Specifically, the high-dimensional extreme eigenvalue
problem~\eqref{eq:final_EEP} is an instance of the following optimization
problem with a smooth objective
$f:\mathbb R^{n_1\times\cdots\times n_d}\to\mathbb R$:
\begin{equation}\label{eq:objective}
  \begin{aligned}
    \min f(\mathcal{X}), \quad \text{s.\,t.} \quad
  \mathcal{X}\in\mathbb{R}^{n_1\times \cdots\times n_d}_{\leq\mathbf r^{\TT}},
  \end{aligned}
\end{equation}
where $\mathbf r^{\TT}=(r_0,r_1,\ldots,r_d)$ is the rank parameter.
When no ambiguity arises, we omit the superscript $\TT$ and write the rank
parameter simply as $\mathbf r$.
Since the set $\mathbb{R}_{\leq\mathbf r^{\TT}}^{n_{1}\times\cdots\times{n_d}}$ is not a smooth manifold, Riemannian optimization on manifolds is not applicable. 
The proposed parametrization treats this nonsmooth set as the image of a smooth manifold. Related constructions were studied in~\cite{Valentin2018,Quentin2024,Yang2025,Duan2026}, including the parametrization of bounded-rank Tucker varieties in~\cite{gao2024},
and other low-rank tensor formulations~\cite{Pippan2010,UschmajewVandereycken2013,Dong2022CP,Gao2024TR,gaoLowrankOptimizationTucker2025}.

Through the lens of a smooth parametrization, problem~\eqref{eq:objective} on a nonsmooth set is transformed into a problem on a smooth manifold, enabling Riemannian optimization methods. Optimization on manifolds requires the geometric information from the smooth manifold; see an illustration in~\Cref{fig:intro}. First, the explicit representation of points on the parametrization manifold and vectors in the tangent space is derived, which is the basis for the geometric computations. Moreover, the projection from ambient space onto tangent space gives the descent direction of the objective function using its Euclidean gradient, and allows the computation of vector transport employed in the conjugate gradient method. Finally, in order to search on the manifold, the retraction map is necessary to map the tangent vector back to the manifold.

\subsection{Parametrization manifold}\label{sec:parametrization}

To address the nonsmoothness of the bounded-rank tensor space, slack
variables are introduced to lift the original variables to the
higher-dimensional parametrization manifold. With additional slack
variables, singular points with non-maximal TT ranks in the original space
are represented by regular points on the parametrization manifold~\cite{Valentin2018,Quentin2024}.
The bounded-rank tensor space is then recovered by projecting this manifold
into its tensor component. This construction is subsequently extended to
low-rank Tucker tensor spaces~\cite{gao2024}, motivating us to adopt a similar
approach for the TT format.

Specifically, we use the orthogonal projectors
into the complements of the right-interface ranges as slack variables. These
projectors retain the information lost at singular points. Accordingly, we
define the parametrization manifold by
\begin{equation*}
  \mathcal{M}_{\mathbf r}^{\TT}:=\left\{ \left(\mathcal{X},\mathbf P_1,\dots,\mathbf P_{d-1}\right):\
  \begin{aligned}
        &\mathcal{X}=\mathcal{U}_1 \cdots \mathcal{U}_d, \quad\left(\mathcal{U}_1,\dots,\mathcal{U}_d\right) \in \mathcal{S}_{\mathbf r}^{\TT}, \\
        &  \mathbf P_k=\mathbf I_{n_{k+1}\cdots n_d}
  -\mathbf X_{\geq k+1}
  (\mathbf X_{\geq k+1}^\top\mathbf X^{}_{\geq k+1})^{-1}
  \mathbf X_{\geq k+1}^\top,\quad k=1,\dots,d-1\,
  \end{aligned}
  \right\},
\end{equation*}
where
\begin{equation*}
  \mathcal{S}_{\mathbf r}^{\TT}:=
  \left\{ \left(\mathcal{U}_1,\dots,\mathcal{U}_d\right):
  \begin{aligned}
    &\, \mathcal{U}_k\in\mathbb{R}^{r_{k-1}\times{n_k}\times{r_k}}, & k=1,\dots,d\, \\
    &\rank\bigl(\R(\mathcal U_k)\bigr)=r_{k-1},& k=2,\dots,d\,
  \end{aligned}
  \right\}
\end{equation*}
is the TT-cores space. The full-row-rank conditions on the right
matricizations imply that all right interface matrices
$\mathbf X_{\geq k+1}$ have full-column-rank; hence the projectors
$\mathbf P_k$ are well defined. For convenience, let
\begin{equation*}
      {\widehat{\mathcal{M}}}_{\mathbf r}^{\TT}:=\mathbb{R}^{n_{1}\times\cdots\times n_d} \times \mathbb S^{n_2\cdots n_d} \times \cdots \times \mathbb S^{n_d},\quad   {\widehat{\mathcal{S}}}_{\mathbf r}^{\TT}:=\mathbb{R}^{{r_0}\times{n_1}\times{r_1}} \times \cdots \times \mathbb{R}^{{r_{d-1}}\times{n_d}\times{r_d}}
\end{equation*}
be the ambient spaces of the parametrization manifold
$\mathcal{M}_{\mathbf r}^{\TT}$ and the TT-cores space
$\mathcal S_{\mathbf r}^{\TT}$, respectively.
The parametrization map $\phi$ is defined as the projection to the first component
\begin{equation*}
  \phi
  \colon{\widehat{\mathcal{M}}}_{\mathbf r}^{\TT}\rightarrow
  \mathbb{R}^{n_{1}\times\cdots\times{n_d}}_{\leq\mathbf r^{\TT}}
  ,\qquad
  \left(\mathcal{X},\mathbf P_1,\dots,\mathbf P_{d-1}\right)\mapsto\mathcal{X}.
\end{equation*}
Moreover, the map $\psi$ sends the TT cores
$\left(\mathcal{U}_1,\dots,\mathcal{U}_d\right)$ to the tuple
$(\mathcal X,\mathbf P_1,\dots,\mathbf P_{d-1})$ and is smooth.
The resulting relationship among the TT-cores space, the parametrization
manifold, and the ambient spaces is illustrated in~\Cref{fig:spaces}.

\begin{figure}[htbp]
  \centering
  \begin{circuitikz}[x=0.72cm,y=0.72cm]
    \tikzset{every node/.style={font=\normalsize}}

        \node (low-rank-space) at (7.6,4.4)
      {$\mathbb{R}_{\leq\mathbf r^{\TT}}^{n_1\times\cdots\times n_d}$};
    \node at (8.8,4.4) {$\subseteq$};
    \node [anchor=west] at (9.4,4.4)
      {$\mathbb{R}^{n_1\times\cdots\times n_d}$};

        \node (parametrization-manifold) at (7.6,2.2)
      {$\mathcal M_{\mathbf r}^{\TT}$};
    \node at (8.8,2.2) {$\subseteq$};
    \node [anchor=west] at (9.4,1.85)
      {$\begin{aligned}
        \widehat{\mathcal M}_{\mathbf r}^{\TT}
        &={}\mathbb R^{n_1\times\cdots\times n_d}
        \times\mathbb S^{n_2\cdots n_d}\times\cdots\times\mathbb S^{n_d}\\[-0.1em]
        &\quad\quad\quad (\mathcal X,\mathbf P_1,\ldots,\mathbf P_{d-1})
      \end{aligned}$};

        \node (cores-space) at (7.6,0.0) {$\mathcal S_{\mathbf r}^{\TT}$};
    \node at (8.8,0.0) {$\subseteq$};
    \node [anchor=west] at (9.4,-0.35)
      {$\begin{aligned}
        \widehat{\mathcal S}_{\mathbf r}^{\TT}
        &={}\mathbb R^{r_0\times n_1\times r_1}\times\cdots\times
        \mathbb R^{r_{d-1}\times n_d\times r_d}\\[-0.1em]
        &\quad\quad\quad\ \  (\mathcal U_1,\ldots,\mathcal U_d)
      \end{aligned}$};

        \draw [line width=0.8pt, ->, >=Stealth]
      ([yshift=0.22cm]parametrization-manifold.north) --
      node [midway, left=0.08cm] {$\phi$}
      ([yshift=-0.22cm]low-rank-space.south);
    \draw [line width=0.8pt, dashed, ->, >=Stealth]
      ([yshift=0.22cm]cores-space.north) --
      node [midway, left=0.08cm] {$\psi$}
      ([yshift=-0.22cm]parametrization-manifold.south);
  \end{circuitikz}
  \caption{Relationship among the low-rank tensor space, the parametrization
  manifold, and the TT-cores space.}
  \label{fig:spaces}
\end{figure}

To ensure that the parametrization covers the entire bounded-rank TT space, the map $\phi$ must be surjective. We establish this property by constructing elements on the parametrization
manifold for any tensor in the bounded-rank TT space through enlarging TT cores.
\begin{proposition}\label{thm:surjectivity}
  Suppose that the rank parameter
  $\mathbf r^{\TT}=(r_0,r_1,\ldots,r_d)$, $r_0=r_d=1$, satisfies admissibility bounds~\eqref{eq:tt-rank-admissibility}. It holds that 
  \[
    \phi(\mathcal M_{\mathbf r}^{\TT})
    =\mathbb R^{n_1\times\cdots\times n_d}_{\leq\mathbf r^{\TT}}.
  \]
\end{proposition}
\begin{proof}
  Every tensor represented by cores of sizes
  $r_{k-1}\times n_k\times r_k$ has TT rank bounded by $\mathbf r^{\TT}$,
  which proves the inclusion from left to right.
  Conversely, take
  $\mathcal Y\in
  \mathbb R^{n_1\times\cdots\times n_d}_{\leq\mathbf r^{\TT}}$, and choose a right-orthogonal TT decomposition
  $\mathcal Y=\mathcal V_1\cdots\mathcal V_d$ with the TT rank
  $\widehat{\mathbf r}^{\TT}\leq\mathbf r^{\TT}$. Enlarge each core to
  $\widetilde{\mathcal V}_k\in
  \mathbb R^{r_{k-1}\times n_k\times r_k}$ as follows. For the first core, we
  retain $\mathcal V_1$ in the first $\widehat r_1$ third-mode slices and
  set the remaining slices to zero. For $k=2,\ldots,d$, first extend
  $\mathcal V_k$ with zeros along its third mode to size
  $\widehat r_{k-1}\times n_k\times r_k$. Retain this padded core in the
  first $\widehat r_{k-1}$ first-mode slices, and fill the remaining
  $r_{k-1}-\widehat r_{k-1}$ slices so that the right-matricization rows
  form an orthonormal set orthogonal to the retained rows. For the last
  core, only the first mode is enlarged since $r_d=\widehat r_d=1$.
  Direct calculation gives
  \[
    \rank\bigl(\R(\widetilde{\mathcal V}_k)\bigr)=r_{k-1},
    \qquad k=2,\ldots,d,
    \qquad
    \widetilde{\mathcal V}_1\cdots\widetilde{\mathcal V}_d
    =\mathcal V_1\cdots\mathcal V_d=\mathcal Y.
  \]
  Hence the new cores belong to $\mathcal S_{\mathbf r}^{\TT}$ and represent
  the same tensor $\mathcal Y$. Thus the claimed surjectivity holds.
\end{proof}
Then the quotient manifold theorem~\cite{AbsilMahonySepulchre+2008} is used to show that the parametrization manifold is a smooth manifold.
A TT tensor is unchanged under
invertible transformations of its internal cores, so different core tuples
can determine the same tensor and right-interface projectors. To remove this
nonuniqueness, consider the gauge group
$\mathcal G:=\GL(r_1)\times\cdots\times\GL(r_{d-1})$ and the equivalence
relation $\sim$ induced by its action on $\mathcal S_{\mathbf r}^{\TT}$.
As proved in~\cite{gao2024}, this action is smooth, free, and proper.
The quotient-manifold theorem therefore applies and makes
$\mathcal S_{\mathbf r}^{\TT}/\!\sim$ a smooth manifold. Consequently,
$\mathcal M_{\mathbf r}^{\TT}$ is an embedded submanifold of
$\widehat{\mathcal M}_{\mathbf r}^{\TT}$, with dimension
\[
  \dim\mathcal M_{\mathbf r}^{\TT}
  =\dim\mathcal S_{\mathbf r}^{\TT}-\dim\mathcal G
  =\sum_{k=1}^d r_{k-1}n_kr_k-\sum_{k=1}^{d-1}r_k^2.
\]

\subsection{Tangent space}

We obtain the tangent space of $\mathcal{M}_{\mathbf r}^{\TT}$ by
differentiating smooth curves in the TT-cores space. The resulting core
representations are nonunique because different TT cores can represent the same
tensor and projector. We therefore impose gauge conditions that remove the freedom of cores and yield a unique representation of each tangent vector.

Let $z=(\mathcal{X},\mathbf P_1,\dots,\mathbf P_{d-1})\in
\mathcal{M}_{\mathbf r}^{\TT}$ be represented by
$(\mathcal{U}_1,\dots,\mathcal{U}_d)\in\mathcal{S}_{\mathbf r}^{\TT}$, and let
$(\delta\mathcal U_1,\dots,\delta\mathcal U_d)\in
\widehat{\mathcal{S}}_{\mathbf r}^{\TT}$ be an arbitrary tuple of tensors. 
Consider the smooth curve
$c:\mathbb{R}\to\widehat{\mathcal{S}}_{\mathbf r}^{\TT}$ given by
\begin{equation*}
c\left(t\right):=\left(\mathcal{U}_1,\dots,\mathcal{U}_d\right)+t\left(\delta\mathcal U_1,\dots,\delta\mathcal U_d\right).
\end{equation*}
By the openness of the full-rank matrix set discussed in~\Cref{sec:matrix}, there exists $T>0$ such that, for every $|t|<T$ and $k=2,\dots,d$, the matrix $\R(\mathcal{U}_k)+t\R(\delta\mathcal U_k)$ remains full-row-rank.
Consequently, $c(t)\in\mathcal{S}_{\mathbf r}^{\TT}$ and $\psi\circ c$ is a smooth curve in $\mathcal{M}_{\mathbf r}^{\TT}$.
We use the prefix $\delta$ to denote derivatives along this curve
at $t=0$. In particular, $\delta\mathcal U_k$,
$\delta\mathcal X$, $\delta\mathbf X_{\geq k}$, and
$\delta\mathbf P_k$ denote the derivatives of the TT cores,
the represented tensor, the right-interface matrices, and the
projectors, respectively. The resulting tangent vector is
$\xi=(\delta\mathcal X,\delta\mathbf P_1,\ldots,\delta\mathbf P_{d-1})$.
When several tangent directions are involved, we add a subscript,
as in $\delta_\xi\mathcal U_k$ and
$\delta_\xi\mathbf X_{\geq k}$, to specify the corresponding direction.
At $t=0$, the derivatives of the TT cores are $\delta\mathcal U_1,\ldots,\delta\mathcal U_d$, and the derivative of $\psi\circ c$ gives the corresponding tangent vector.
The core representation is not unique due to the selection freedom of vector $(\delta\mathcal U_1,\ldots,\delta\mathcal U_d)$.
We select a unique representative and simplify the derivatives of the interface matrices by imposing the
\emph{gauge conditions}:
\begin{equation}\label{eq:gauging}
  \R{\left(\mathcal{U}_k \right)}\R{{\left(\mathcal{U}_k \right)}^\top}=\mathbf I_{r_{k-1}},\quad \R{\left(\mathcal{U}_k \right)}\R{{\left(\delta\mathcal U_k \right)}^\top}=\mathbf 0,\qquad k=2,\dots,d,
\end{equation}
or equivalently, we have the following interface form
\begin{equation}\label{eq:gauging-interface}
    {\mathbf X_{\geq k}^\top}{\mathbf X^{}_{\geq k}}=\mathbf I_{r_{k-1}}, \quad
    {\mathbf X_{\geq k}^\top}\delta{\mathbf X^{}_{\geq k}}=\mathbf 0,\qquad k=2,\dots,d.
\end{equation}
Note that it requires the TT cores to be right-orthogonal and the tensors $\delta\mathcal U_k$ to be orthogonal to the corresponding gauge directions.
Similar gauge conditions are standard for tangent representations on
fixed-rank TT manifolds; see~\cite[Theorem~4.2 and Remark~4.3]{holtz2012} and
\cite[Section~3.1]{steinlechner2016}.

The second interface identity in~\eqref{eq:gauging-interface} implies that
$\delta\mathbf X_{\geq k}$ lies in the orthogonal complement of the
range of $\mathbf X_{\geq k}$. Therefore,
${\left({\mathbf I-{\mathbf X_{\geq k}}{\mathbf X_{\geq k}^\dagger}}\right)}
\delta{\mathbf X_{\geq k}}=\delta{\mathbf X_{\geq k}}$ for
$k=2,\dots,d$. Thus, the following proposition is obtained.
\begin{proposition}[tangent space] \label{thm:tangent-space}
  Let $z=(\mathcal X,\mathbf P_1,\ldots,\mathbf P_{d-1})\in
  \mathcal M_{\mathbf r}^{\TT}$ be represented by right-orthogonal cores
  $(\mathcal U_1,\ldots,\mathcal U_d)\in\mathcal S_{\mathbf r}^{\TT}$.
  Then every tangent vector at $z$ admits a core representation
  $(\delta\mathcal U_1,\ldots,\delta\mathcal U_d)$ satisfying the gauge
  conditions~\eqref{eq:gauging}, which has the form
  $\left(\delta\mathcal{X},\delta\mathbf P_1,\dots,
  \delta\mathbf P_{d-1}\right)$, where
  \begin{align*}
    \delta\mathcal{X}
    &=\sum_{k=1}^d \mathcal{U}_1 \cdots \mathcal{U}_{k-1}
      \delta\mathcal U_k \mathcal{U}_{k+1} \cdots \mathcal{U}_d, \\
    \delta\mathbf P_k
    &=-2\sym \left(\delta \mathbf X^{}_{\geq k+1}
      \mathbf X_{\geq k+1}^\top\right),
      \qquad k=1,\dots,d-1,
  \end{align*}
  with $\delta \mathbf X_{\geq d} =\R \left(\delta\mathcal U_d \right)^\top$ and
$    \delta \mathbf X_{\geq k}
    =\left(\mathbf X_{\geq k+1}\otimes \mathbf I_{n_k}\right)
      \R \left(\delta\mathcal U_k\right)^\top+\left(\delta \mathbf X_{\geq k+1}\otimes \mathbf I_{n_k}\right)
      \R \left(\mathcal{U}_k\right)^\top$, $k=d-1,\dots,2$.
\end{proposition}
\begin{proof}
  Differentiating the multilinear TT product $\mathcal{X}=\mathcal{U}_1 \cdots \mathcal{U}_d$ gives the formula for
  $\delta\mathcal X$. For a full-column-rank matrix $\mathbf Y$, the
  derivative of $\mathbf I-\mathbf Y\mathbf Y^\dagger$ in the direction
  $\delta\mathbf Y$ is
  $-2\sym((\mathbf I-\mathbf Y\mathbf Y^\dagger)
  \delta\mathbf Y\mathbf Y^\dagger)$. Applying this identity with
  $\mathbf Y=\mathbf X_{\geq k+1}$ and using the interface form of the
  gauge conditions gives the stated formula for $\delta\mathbf P_k$.
  Differentiating the right-interface recursion gives the remaining formulas,
  including the terminal value
  $\delta\mathbf X_{\geq d}=\R(\delta\mathcal U_d)^\top$.
  These derivatives belong to $\mathrm{T}_z\mathcal M_{\mathbf r}^{\TT}$, so the
  proposed formulas define a linear map from the tuples satisfying
  \eqref{eq:gauging} to the tangent space. This map is injective by
  \Cref{lemma:horizontal-injective}, which is proved below.

  Then we show that the map is surjective by computing the dimensions of its domain.
  Note that
  $\delta\mathcal U_1$ is unconstrained and contributes $n_1r_1$ degrees of
  freedom. For each $k=2,\ldots,d$, the tensor $\delta\mathcal U_k$ has
  $r_{k-1}n_kr_k$ entries, while the gauge condition imposes
  $r_{k-1}^2$ independent linear constraints. Hence the domain has dimension
  $n_1r_1+\sum_{k=2}^d(r_{k-1}n_kr_k-r_{k-1}^2)
  =\sum_{k=1}^d r_{k-1}n_kr_k-\sum_{k=1}^{d-1}r_k^2$.
  This equals $\dim (\mathrm{T}_z\mathcal M_{\mathbf r}^{\TT})$ by the dimension formula
  established in~\Cref{sec:parametrization}. Therefore, the linear map is
  bijective, and every tangent vector has the displayed representation.
\end{proof}

The recursive formula for the interface derivative also has a direct
interpretation as the sum of TT products in which one core is replaced
by its derivative.
\begin{lemma}\label{lemma:resp-delta}
  Let $\delta{\mathbf X_{\geq k}}$ be the derivative of the interface
  matrix generated by $(\delta\mathcal U_k,\ldots,\delta\mathcal U_d)$. Then
  \begin{equation*}
    \delta{\mathbf X_{\geq k}}=\R{\left({\sum_{i=k}^d {\mathcal{U}_k \cdots\delta\mathcal U_i \cdots\mathcal{U}_d }}\right)}^\top,\qquad k=2,\ldots,d.
  \end{equation*}
  In other words, $\delta\mathbf X_{\geq k}$ collects the contributions from
  $\delta\mathcal U_k,\ldots,\delta\mathcal U_d$.
\end{lemma}
\begin{proof}
  For $k=d$, $\delta\mathbf X_{\geq d}=\R(\delta\mathcal U_d)^\top$.
  Assume inductively that the identity holds for $k=i$. Substituting this
  identity into the right-interface recursion at $k=i-1$ gives
  $\delta\mathbf X_{\geq k}
  =\R(\sum_{i=k}^d\mathcal U_k\cdots\delta\mathcal U_i\cdots
  \mathcal U_d)^\top$.
\end{proof}
Note that the identity remains valid when the interface derivative
$\delta\mathbf X_{\geq k}$ is extended to $k=1$.
The gauge conditions select a horizontal complement to the TT gauge
directions. To show that the resulting representation is unique, it remains to
verify that no nonzero tuple satisfying the gauge conditions represents the
zero tangent vector.
\begin{lemma}\label{lemma:horizontal-injective}
If $(\delta\mathcal U_1,\ldots,\delta\mathcal U_d)$ satisfies the gauge
conditions and represents the zero tangent vector, then
$\delta\mathcal U_k=0$ for every $k=1,\ldots,d$.
\end{lemma}
\begin{proof}
Suppose $(\delta\mathcal X,\delta\mathbf P_1,\ldots,
\delta\mathbf P_{d-1})=0$. For the last projector,
$\delta\mathbf P_{d-1}=-2\sym(\R(\delta\mathcal U_d)^\top
\R(\mathcal U_d))=0$. Together with the gauge condition
$\R(\mathcal U_d)\R(\delta\mathcal U_d)^\top=0$ and the orthonormality of
the rows of $\R(\mathcal U_d)$, this gives
$\R(\delta\mathcal U_d)=0$.

Assume recursively that $\delta\mathcal U_{k+1}=\cdots=
\delta\mathcal U_d=0$. Then
$\delta\mathbf X_{\geq k}=(\mathbf X_{\geq k+1}\otimes
\mathbf I_{n_k})\R(\delta\mathcal U_k)^\top$. The equality
$\delta\mathbf P_{k-1}=0$, the gauge condition, and the full-column-rank of
$\mathbf X_{\geq k+1}$ imply $\R(\delta\mathcal U_k)=0$. Backward
induction gives $\delta\mathcal U_k=0$ for $k=2,\ldots,d$. Finally,
$0=\delta\mathbf X^{\langle1\rangle}=\matL(\delta\mathcal U_1)
\mathbf X_{\geq2}^\top$. Since $\mathbf X_{\geq2}$ has full
column rank, $\matL(\delta\mathcal U_1)=0$.
\end{proof}

Riemannian optimization methods on $\mathcal{M}^{\TT}_{\mathbf{r}}$ also requires Riemannian metric on the tangent space.
We choose the Euclidean metric which can be evaluated without computing the full
tensors or projector derivatives by using a backward recursion over the TT
cores.
\begin{proposition}[Riemannian metric] \label{thm:inner}
  Let $z=(\mathcal X,\mathbf P_1,\dots,\mathbf P_{d-1})\in
  \mathcal M_{\mathbf r}^{\TT}$ be represented by
  $(\mathcal U_1,\dots,\mathcal U_d)$. Let tangent vectors
$\xi=(\delta_\xi\mathcal X,\delta_\xi\mathbf P_1,\dots,\delta_\xi\mathbf P_{d-1})$ and $
    \eta=(\delta_\eta\mathcal X,\delta_\eta\mathbf P_1,\dots,
    \delta_\eta\mathbf P_{d-1})$ be represented by the tuples
  $(\delta_\xi\mathcal U_1,\dots,\delta_\xi\mathcal U_d)$ and
  $(\delta_\eta\mathcal U_1,\dots,\delta_\eta\mathcal U_d)$ that satisfy the
  gauge conditions~\eqref{eq:gauging}, respectively. Then
  \begin{equation*}
    \langle\xi,\eta\rangle=\tr(\mathbf S_1)+2\sum_{k=2}^{d}\tr(\mathbf S_k),
  \end{equation*}
  where $\mathbf S_k$, $k=1,\dots,d$, are square matrices of size $r_{k-1}$,
  defined inductively by $\mathbf S_d:=\R(\delta_\xi\mathcal U_d){\R(\delta_\eta\mathcal U_d)}^\top$, and
$\mathbf S_k
    :=\R(\delta_\xi\mathcal U_k){\R(\delta_\eta\mathcal U_k)}^\top
    +\R(\mathcal{U}_k)
      \left(\mathbf S_{k+1} \otimes \mathbf I_{n_k}\right)
      {\R(\mathcal{U}_k)}^\top$, $k=d-1,\dots,1$.
\end{proposition}
\begin{proof}
  By the product metric on the ambient space and the gauge conditions,
  \begin{align*}
    \langle\xi,\eta\rangle
    &= \left\langle \delta_\xi\mathcal{X},\delta_\eta\mathcal{X}\right\rangle
    + \sum_{k=1}^{d-1}
    \left\langle \delta_\xi\mathbf P_k,\delta_\eta\mathbf P_k \right\rangle \\
    &= \left\langle \delta_\xi\mathbf X_{\geq 1},
    \delta_\eta\mathbf X_{\geq 1} \right\rangle
    + 2 \sum_{k=1}^{d-1}
    \left\langle \delta_\xi\mathbf X_{\geq k+1},
    \delta_\eta\mathbf X_{\geq k+1}\right\rangle.
  \end{align*}
  Substituting the interface recursions into each inner product and expanding
  shows that the two mixed terms vanish by the gauge conditions. The remaining
  terms satisfy $\langle\delta_\xi\mathbf X_{\geq k},
  \delta_\eta\mathbf X_{\geq k}\rangle=\tr(\mathbf S_k)$ by the definition of
  $\mathbf S_k$.
  Substitution into the product metric proves the stated recursion.
\end{proof}
\paragraph{Remark}
  For the manifold of tensors with fixed TT rank~\cite{steinlechner2016}, the tangent-vector inner
  product reduces to the term $\tr(\mathbf S_1)$ in~\Cref{thm:inner}. The additional
  contribution $2\sum_{k=2}^{d}\tr(\mathbf S_k)$ accounts for the derivatives of the
  right-interface projectors in the parametrization manifold.

\subsection{Projection}
The Riemannian gradient is obtained by projecting the ambient Euclidean
gradient onto the tangent space $\mathrm{T}_z\mathcal{M}_{\mathbf r}^{\TT}$. The following construction
expresses the projection through TT-core and interface recursions without
forming an explicit basis of the tangent space.

First, we define the following forward matrix sequences.
  Let $z=(\mathcal X,\mathbf P_1,\dots,\mathbf P_{d-1})\in
  \mathcal M_{\mathbf r}^{\TT}$ be represented by gauged cores
  $(\mathcal U_1,\dots,\mathcal U_d)\in\mathcal S_{\mathbf r}^{\TT}$. For an
  ambient vector $(\mathcal Y,\mathbf Q_1,\dots,\mathbf Q_{d-1})\in
  \widehat{\mathcal M}_{\mathbf r}^{\TT}$, define the matrix sequences
  $(\mathbf L_k)_{k=0}^{d-1}$ and $(\mathbf R_k)_{k=0}^{d-1}$ by
  $\mathbf L_0:=\mathbf I_{r_0}$ and
  $\mathbf R_0:=\vecop(\mathcal Y)^\top$, together with
  \begin{align*}
    {\mathbf L_k}&:={\ptr_{n_k}}{\left({\R{{\left(\mathcal{U}_k \right)}^\top}{\mathbf L_{k-1}}\R{\left(\mathcal{U}_k \right)}}\right)}+2{\mathbf I_{r_k}},\qquad k=1,\dots,d-1, \\
    {\mathbf R_k}&:={\ptr_{n_k}}{\left({\R{{\left(\mathcal{U}_k \right)}^\top}{\mathbf R_{k-1}}}\right)}-2{\mathbf X_{\geq k+1}^\top}{\mathbf Q_k},\qquad k=1,\dots,d-1.
  \end{align*}

Note that every matrix $\mathbf L_k$, $k=0,\dots,d-1$, is positive definite and hence invertible.
  The initial matrix $\mathbf L_0=\mathbf I_{r_0}$ is positive definite. If
  $\mathbf L_k$ is positive definite, then
  $\R(\mathcal U_{k+1})^\top\mathbf L_k\R(\mathcal U_{k+1})$ is positive
  semidefinite. Its partial trace is also positive semidefinite by the result
  in~\Cref{sec:matrix}; adding $2\mathbf I_{r_{k+1}}$ therefore makes
  $\mathbf L_{k+1}$ positive definite. The claim follows by induction.

The positive definiteness of $\mathbf L_k$ makes each core update well defined. The
orthogonal projection is then obtained using the gauge conditions in~\eqref{eq:gauging} to simplify the computation.
\begin{proposition}[projection] \label{thm:projection}
  Let $z=(\mathcal X,\mathbf P_1,\dots,\mathbf P_{d-1})\in
  \mathcal M_{\mathbf r}^{\TT}$. The
  orthogonal projection $\Proj_{z}(\mathcal Y,\mathbf Q_1,\dots,\mathbf Q_{d-1})$ of
  $(\mathcal Y,\mathbf Q_1,\dots,\mathbf Q_{d-1})\in
  \widehat{\mathcal M}_{\mathbf r}^{\TT}$ into 
  $\mathrm{T}_z\mathcal M_{\mathbf r}^{\TT}$
  has core representation
          $(\delta\mathcal U_1,\dots,\delta\mathcal U_d)$ that satisfies
  \begin{align*}
    \R{\left({\delta\mathcal U_{1}}\right)}&={\mathbf L_0^{-1}}{\mathbf R_0}{\left({{\mathbf X_{\geq 2}}\otimes{\mathbf I_{n_{1}}}}\right)}, \\
    \R{\left({\delta\mathcal U_{k+1}}\right)}
    &={\mathbf L_k^{-1}}{\mathbf R_k}{\left({{\mathbf X_{\geq k+2}}\otimes{\mathbf I_{n_{k+1}}}}\right)}
    \left(\mathbf{I}_{r_{k+1}n_{k+1}}-\R\left(\mathcal{U}_{k+1}\right)^\top
    \R\left(\mathcal{U}_{k+1}\right)\right),
    \qquad k=1,\dots,d-2, \\
    \R{\left(\delta\mathcal U_d \right)}&={\mathbf L_{d-1}^{-1}}{\mathbf R_{d-1}}  \left(\mathbf I_{r_d n_d}-\R\left(\mathcal{U}_{d}\right)^\top \R\left(\mathcal{U}_{d}\right)\right).
  \end{align*}
\end{proposition}
\begin{proof}
  To distinguish the projected vector from a test tangent vector in this
  proof, we denote $\delta\mathcal U_k$ by $\delta_\xi\mathcal U_k$ and use
  the same subscript for the corresponding interface derivatives and
  components, so $\xi=(\delta_\xi\mathcal X,\delta_\xi\mathbf P_1,\ldots,
  \delta_\xi\mathbf P_{d-1})$. Let an arbitrary tangent vector
  $\eta\in \mathrm{T}_z\mathcal M_{\mathbf r}^{\TT}$ be represented by a tuple
  $(\delta_\eta\mathcal U_1,\dots,\delta_\eta\mathcal U_d)$ satisfying the
  gauge conditions, where $\eta=(\delta_\eta\mathcal X,
  \delta_\eta\mathbf P_1,\dots,\delta_\eta\mathbf P_{d-1})$.
  Since $\eta$ is arbitrary, it suffices to prove
  \begin{equation*}
    \left\langle
    (\mathcal Y,\mathbf Q_1,\dots,\mathbf Q_{d-1})
    -(\delta_\xi\mathcal X,\delta_\xi\mathbf P_1,\dots,\delta_\xi\mathbf P_{d-1}),
    (\delta_\eta\mathcal X,\delta_\eta\mathbf P_1,\dots,
    \delta_\eta\mathbf P_{d-1})\right\rangle=0.
  \end{equation*}
  Expanding this condition gives
  $\langle\mathcal Y-\delta_\xi\mathcal X,\delta_\eta\mathcal X\rangle
  +\sum_{k=1}^{d-1}\langle\mathbf Q_k-\delta_\xi\mathbf P_k,
  \delta_\eta\mathbf P_k\rangle=0$.

Using the tangent representation in~\Cref{thm:tangent-space}, the symmetry of $\mathbf Q_k$ and $\delta_\xi\mathbf P_k$, and the interface representation in~\Cref{lemma:resp-delta}, the orthogonality condition becomes
    \begin{equation}\label{eq:proj-origin}
      \left\langle \vecop (\mathcal{Y})-
      \delta_\xi{\mathbf X_{\geq1}},
      \delta_\eta{{\mathbf X}_{\geq1}}\right\rangle
      -2\sum_{k=1}^{d-1} \left(
      \left\langle {{\mathbf Q_k},{\mathbf X_{\geq k+1}}
      \delta_\eta{{\mathbf X}_{\geq k+1}^\top}}\right\rangle
      +\left\langle {\delta_\xi{\mathbf X_{\geq k+1}},
      \delta_\eta{{\mathbf X}_{\geq k+1}}}\right\rangle
      \right)=0.
  \end{equation}
  For the first term,
  \begin{align*}
    \left\langle {\vecop(\mathcal{Y})-\delta_\xi{\mathbf X_{\geq1}},
      \delta_\eta{{\mathbf X}_{\geq1}}}\right\rangle
    &=\left\langle {\vecop(\mathcal{Y}),
      \delta_\eta{{\mathbf X}_{\geq1}}}\right\rangle
      -\left\langle {\delta_\xi{\mathbf X_{\geq1}},
      \delta_\eta{{\mathbf X}_{\geq1}}}\right\rangle \\
    &=\left\langle {{\mathbf R_0^\top},
      \delta_\eta{{\mathbf X}_{\geq1}}}\right\rangle
      -\left\langle {{\mathbf L_0^\top},
      \delta_\xi{\mathbf X_{\geq1}^\top}
      \delta_\eta{{\mathbf X}_{\geq1}}}\right\rangle.
  \end{align*}
  Using the above equation, we show by induction that, after substituting the first $s$
projected tensors $\delta_\xi\mathcal U_1,\ldots,\delta_\xi\mathcal U_s$, \eqref{eq:proj-origin} reduces to
  \begin{equation}\label{eq:proj-iter}
      \left\langle {{\mathbf R_s^\top},
      \delta_\eta{{\mathbf X}_{\geq s+1}}}\right\rangle
      -\left\langle {{\mathbf L_s^\top},
      \delta_\xi{\mathbf X_{\geq s+1}^\top}
      \delta_\eta{{\mathbf X}^{}_{\geq s+1}}}\right\rangle
      -2\sum_{k=s+1}^{d-1} \left(
      {\left\langle {{\mathbf Q_k},{\mathbf X_{\geq k+1}}
      \delta_\eta{{\mathbf X}_{\geq k+1}^\top}}\right\rangle}
      +{\left\langle {\delta_\xi{\mathbf X_{\geq k+1}},
      \delta_\eta{{\mathbf X}_{\geq k+1}}}\right\rangle}
      \right)=0.
  \end{equation}
  The case $s=0$ follows from $\mathbf L_0=\mathbf I_{r_0}$ and
  $\mathbf R_0=\vecop(\mathcal Y)^\top$. Assume~\eqref{eq:proj-iter} holds for some
  $s<d-1$, and denote its first two terms by $A$ and $B$:
  \begin{equation*}
    A:=\left\langle {{\mathbf R_s^\top},
    \delta_\eta{{\mathbf X}_{\geq s+1}}}\right\rangle,
    \qquad
    B:=\left\langle {{\mathbf L_s^\top},
    \delta_\xi{\mathbf X_{\geq s+1}^\top}
    \delta_\eta{{\mathbf X}^{}_{\geq s+1}}}\right\rangle.
  \end{equation*}
  By the recursions for $\delta_\xi\mathbf X_{\geq s+1}$ and
  $\delta_\eta\mathbf X_{\geq s+1}$, it follows that
  \begin{align*}
    A
    &=\left\langle {{\mathbf R_s^\top}},
      {\left({{\mathbf X_{\geq s+2}}\otimes{\mathbf I_{n_{s+1}}}}\right)}
      \R{{\left({\delta_\eta\mathcal U_{s+1}}\right)}^\top}
      +{\left({\delta_\eta{{\mathbf X}_{\geq s+2}}
      \otimes{\mathbf I_{n_{s+1}}}}\right)}
      \R{{\left({\mathcal{U}_{s+1}}\right)}^\top}
      \right\rangle \\
    &=\left\langle
      {{\mathbf R_s}{\left({{\mathbf X_{\geq s+2}}\otimes{\mathbf I_{n_{s+1}}}}\right)}},
      {\R{\left({\delta_\eta\mathcal U_{s+1}}\right)}}
      \right\rangle +\left\langle {{\mathbf R_s^\top}},
      {\left({\delta_\eta{{\mathbf X}_{\geq s+2}}
      \otimes{\mathbf I_{n_{s+1}}}}\right)}
      \R{{\left({\mathcal{U}_{s+1}}\right)}^\top}
      \right\rangle,
  \end{align*}
  and
  \begin{align*}
    B
    &=\left\langle
      \left({{\left({{\mathbf X_{\geq s+2}}\otimes{\mathbf I_{n_{s+1}}}}\right)}
      \R{{\left({\delta_\xi\mathcal U_{s+1}}\right)}^\top}
      +{\left({\delta_\xi{\mathbf X_{\geq s+2}}
      \otimes{\mathbf I_{n_{s+1}}}}\right)}
      \R{{\left({\mathcal{U}_{s+1}}\right)}^\top}}\right){\mathbf L_s},
      \right. \\
    &\qquad\left.
      {\left({{\mathbf X_{\geq s+2}}\otimes{\mathbf I_{n_{s+1}}}}\right)}
      \R{{\left({\delta_\eta\mathcal U_{s+1}}\right)}^\top}
      +{\left({\delta_\eta{{\mathbf X}_{\geq s+2}}
      \otimes{\mathbf I_{n_{s+1}}}}\right)}
      \R{{\left({\mathcal{U}_{s+1}}\right)}^\top}
      \right\rangle \\
    &=\left\langle {{\mathbf L_s}\R{\left({\delta_\xi\mathcal U_{s+1}}\right)}},
      {\R{\left({\delta_\eta\mathcal U_{s+1}}\right)}}\right\rangle+\left\langle
      {\R{{\left({\mathcal{U}_{s+1}}\right)}^\top}{\mathbf L_s}},
      {\left({\delta_\xi{\mathbf X_{\geq s+2}^\top}
      \delta_\eta{{\mathbf X}^{}_{\geq s+2}}\otimes{\mathbf I_{n_{s+1}}}}\right)}
      \R{{\left({\mathcal{U}_{s+1}}\right)}^\top}
      \right\rangle.
  \end{align*}
  Substituting the expression for $\delta_\xi\mathcal U_{s+1}$ gives
  \begin{align*}
    A-B
    &=\left\langle {{\mathbf R_s^\top}},
      {\left({\delta_\eta{{\mathbf X}_{\geq s+2}}
      \otimes{\mathbf I_{n_{s+1}}}}\right)}
      \R{{\left({\mathcal{U}_{s+1}}\right)}^\top}
      \right\rangle-\left\langle
      {\R{{\left({\mathcal{U}_{s+1}}\right)}^\top}{\mathbf L_s}},
      {\left({\delta_\xi{\mathbf X_{\geq s+2}^\top}
      \delta_\eta{{\mathbf X}^{}_{\geq s+2}}
      \otimes{\mathbf I_{n_{s+1}}}}\right)}
      \R{{\left({\mathcal{U}_{s+1}}\right)}^\top}
      \right\rangle \\
    &=\left\langle
      {{\mathbf R_s^\top}\R{\left({\mathcal{U}_{s+1}}\right)}},
      {\delta_\eta{{\mathbf X}_{\geq s+2}}\otimes{\mathbf I_{n_{s+1}}}}
      \right\rangle -\left\langle
      {\R{{\left({\mathcal{U}_{s+1}}\right)}^\top}{\mathbf L_s}
      \R{\left({\mathcal{U}_{s+1}}\right)}},
      {\delta_\xi{\mathbf X_{\geq s+2}^\top}
      \delta_\eta{{\mathbf X}^{}_{\geq s+2}}\otimes{\mathbf I_{n_{s+1}}}}
      \right\rangle \\
    &=\left\langle
      {{\ptr_{n_{s+1}}}{\left({{\mathbf R_s^\top}
      \R{\left({\mathcal{U}_{s+1}}\right)}}\right)}},
      {\delta_\eta{{\mathbf X}_{\geq s+2}}}
      \right\rangle-\left\langle
      {{\ptr_{n_{s+1}}}{\left({
      \R{{\left({\mathcal{U}_{s+1}}\right)}^\top}{\mathbf L_s}
      \R{\left({\mathcal{U}_{s+1}}\right)}}\right)}},
      {\delta_\xi{\mathbf X_{\geq s+2}^\top}
      \delta_\eta{{\mathbf X}^{}_{\geq s+2}}}
      \right\rangle.
  \end{align*}
  After substituting $A-B$, \eqref{eq:proj-origin} becomes
  \begin{align*}
    A-B-2{\left\langle {{\mathbf Q_{s+1}},{\mathbf X_{\geq s+2}}
      \delta_\eta{{\mathbf X}_{\geq s+2}^\top}}\right\rangle}
      -2{\left\langle {\delta_\xi{\mathbf X_{\geq s+2}},
      \delta_\eta{{\mathbf X}_{\geq s+2}}}\right\rangle}
    -2\sum_{k=s+2}^{d-1} {\left(
      {\left\langle {{\mathbf Q_k},{\mathbf X_{\geq k+1}}
      \delta_\eta{{\mathbf X}_{\geq k+1}^\top}}\right\rangle}
      +{\left\langle {\delta_\xi{\mathbf X_{\geq k+1}},
      \delta_\eta{{\mathbf X}_{\geq k+1}}}\right\rangle}
      \right)}=0.
  \end{align*}
  The first four terms can be rewritten as
  \begin{align*}
    &A-B-2{\left\langle {{\mathbf Q_{s+1}},{\mathbf X_{\geq s+2}}
      \delta_\eta{{\mathbf X}_{\geq s+2}^\top}}\right\rangle}
      -2{\left\langle {\delta_\xi{\mathbf X_{\geq s+2}},
      \delta_\eta{{\mathbf X}_{\geq s+2}}}\right\rangle} \\
    &\quad=\left\langle
      {{\ptr_{n_{s+1}}}{\left({{\mathbf R_s^\top}
      \R{\left({\mathcal{U}_{s+1}}\right)}}\right)}
      -2{\mathbf Q_{s+1}}{\mathbf X_{\geq s+2}}},
      {\delta_\eta{{\mathbf X}_{\geq s+2}}}
      \right\rangle -{\left\langle
      {{\ptr_{n_{s+1}}}{\left({
      \R{{\left({\mathcal{U}_{s+1}}\right)}^\top}{\mathbf L_s}
      \R{\left({\mathcal{U}_{s+1}}\right)}}\right)}+2{\mathbf I_{r_{s+1}}}},
      {\delta_\xi{\mathbf X_{\geq s+2}^\top}
      \delta_\eta{{\mathbf X}^{}_{\geq s+2}}}
      \right\rangle} \\
    &\quad=\left\langle {{\mathbf R_{s+1}^\top},
      \delta_\eta{{\mathbf X}_{\geq s+2}}}\right\rangle
      -\left\langle {{\mathbf L_{s+1}^\top},
      \delta_\xi{\mathbf X_{\geq s+2}^\top}
      \delta_\eta{{\mathbf X}^{}_{\geq s+2}}}\right\rangle.
  \end{align*}
  This is precisely~\eqref{eq:proj-iter} with $s+1$, which
  completes the induction. At $s=d-1$, the summation in~\eqref{eq:proj-iter} is
  empty, and the remaining expression is
  \begin{align*}
    \left\langle {{\mathbf R_{d-1}^\top},
      \delta_\eta{{\mathbf X}_{\geq d}}}\right\rangle
      -\left\langle {{\mathbf L_{d-1}^\top},
      \delta_\xi{\mathbf X_{\geq d}^\top}
      \delta_\eta{{\mathbf X}^{}_{\geq d}}}\right\rangle =\left\langle
      {{\mathbf R_{d-1}^\top}-\R{{\left(\delta_\xi\mathcal U_d \right)}^\top}
      {\mathbf L_{d-1}^\top}},
      {\R{{\left(\delta_\eta\mathcal U_d \right)}^\top}}
      \right\rangle,
  \end{align*}
  substituting the formula for $\delta_\xi\mathcal U_d$ makes this expression zero
  for every tuple $(\delta_\eta\mathcal U_1,\dots,
  \delta_\eta\mathcal U_d)$ satisfying the gauge conditions. Hence the residual is orthogonal to the tangent
  space, proving that the tangent vector is the unique orthogonal projection.
\end{proof}

For a smooth function $F=f\circ\phi$, only the tensor component contributes to the ambient
Euclidean gradient, which is
$(\nabla f(\mathcal X),\mathbf 0,\ldots,\mathbf 0)$. Projecting this vector by
\Cref{thm:projection} gives the Riemannian gradient.
\begin{corollary}[Riemannian gradient] \label{thm:riemann-grad}
  Let $f:\mathbb R^{n_1\times\cdots\times n_d}\to\mathbb R$ be smooth and
  set $F=f\circ\phi$ on $\mathcal M_{\mathbf r}^{\TT}$. At a point
  $z=(\mathcal X,\mathbf P_1,\dots,\mathbf P_{d-1})$ represented by gauged
  cores $(\mathcal U_1,\dots,\mathcal U_d)\in\mathcal S_{\mathbf r}^{\TT}$,
  the Riemannian gradient $\grad F(z)$ has core representation
  $(\delta\mathcal U_1,\dots,\delta\mathcal U_d)$ given by
  \begin{align*}
    \R{\left({\delta\mathcal U_{1}}\right)}&={\mathbf L_0^{-1}}{\mathbf R_0}{\left({{\mathbf X_{\geq 2}}\otimes{\mathbf I_{n_{1}}}}\right)}, \\
    \R{\left({\delta\mathcal U_{k+1}}\right)}&={\mathbf L_k^{-1}}{\mathbf R_k}{\left({{\mathbf X_{\geq k+2}}\otimes{\mathbf I_{n_{k+1}}}}\right)}  \left(\mathbf I_{r_{k+1}n_{k+1}}-\R\left(\mathcal{U}_{k+1}\right)^\top \R\left(\mathcal{U}_{k+1}\right)\right),\qquad k=1,\dots,d-2, \\
    \R{\left(\delta\mathcal U_d \right)}&={\mathbf L_{d-1}^{-1}}{\mathbf R_{d-1}}  \left(\mathbf I_{r_{d}n_{d}}-\R\left(\mathcal{U}_{d}\right)^\top \R\left(\mathcal{U}_{d}\right)\right),
  \end{align*}
  where $\mathbf L_{0}=\mathbf I_{r_0}$,
  $\mathbf R_{0}=\vecop\left(\nabla f(\mathcal X)\right)^\top$, and the matrix sequences $(\mathbf L_k)_{k=1}^{d-1}$ and $(\mathbf R_k)_{k=1}^{d-1}$ are simplified to
  \begin{align*}
    \mathbf L_k&={\ptr_{n_k}}{\left({\R{{\left(\mathcal{U}_k \right)}^\top}{\mathbf L_{k-1}}\R{\left(\mathcal{U}_k \right)}}\right)}+2{\mathbf I_{r_k}},\qquad k=1,\dots,d-1, \\
    \mathbf R_k&={\ptr_{n_k}}{\left({\R{{\left(\mathcal{U}_k \right)}^\top}{\mathbf R_{k-1}}}\right)},\qquad k=1,\dots,d-1.
  \end{align*}
\end{corollary}

A Riemannian conjugate-gradient method employs tangent vectors based at
different points. We use projection-based vector transport to map tangent vectors between tangent spaces at different points on the manifold, allowing vectors from different iterations to be combined.
\begin{proposition}[vector transport]\label{thm:vector-transport}
Let $z,\widetilde z\in\mathcal M_{\mathbf r}^{\TT}$ have right-orthogonal core representations
$(\mathcal U_1,\ldots,\mathcal U_d)$ and
$(\widetilde{\mathcal U}_1,\ldots,\widetilde{\mathcal U}_d)$,
respectively, and let
$\xi_z=(\delta\mathcal X,\delta\mathbf P_1,\ldots,
\delta\mathbf P_{d-1})\in \mathrm{T}_z\mathcal M_{\mathbf r}^{\TT}$. Define
\begin{equation*}
  \mathcal T_{z\to\widetilde z}(\xi_z)
  :=\Proj_{\widetilde z}(\delta\mathcal X,\delta\mathbf P_1,\ldots,
  \delta\mathbf P_{d-1})
  \in \mathrm{T}_{\widetilde z}\mathcal M_{\mathbf r}^{\TT}.
\end{equation*}
Then $\mathcal T_{z\to\widetilde z}$ is a vector transport.
If $(\delta\widetilde{\mathcal U}_1,\ldots,
\delta\widetilde{\mathcal U}_d)$ is the core representation of the
transported vector, then
\begin{align*}
  \R(\delta\widetilde{\mathcal U}_1)&=\mathbf N_1,\\
  \R(\delta\widetilde{\mathcal U}_k)
  &=\widetilde{\mathbf L}_{k-1}^{-1}\mathbf N_k
  \left(\mathbf I_{n_kr_k}
  -\R(\widetilde{\mathcal U}_k)^\top
  \R(\widetilde{\mathcal U}_k)\right),
  \qquad k=2,\ldots,d.
\end{align*}
The matrix sequences used in the computation are defined below.
First, the matrix sequence $(\widetilde{\mathbf L}_k)_{k=0}^{d-1}$ at $\widetilde z$ is defined by $\widetilde{\mathbf L}_0:=\mathbf I_{r_0}$ and
\begin{align*}
  \widetilde{\mathbf L}_k
  :=\ptr_{n_k}\!\left(
  \R(\widetilde{\mathcal U}_k)^\top
  \widetilde{\mathbf L}_{k-1}
  \R(\widetilde{\mathcal U}_k)\right)+2\mathbf I_{r_k},
  \qquad k=1,\ldots,d-1.
\end{align*}
Then the matrices $(\mathbf N_k)_{k=1}^{d}$ are defined as
\begin{align*}
  \mathbf N_k:={}&\left(\mathbf A_{k-1}\R(\mathcal U_k)
  +\mathbf B_{k-1}\R(\delta\mathcal U_k)\right)
  (\mathbf W_{k+1}^\top\otimes\mathbf I_{n_k})+\mathbf B_{k-1}\R(\mathcal U_k)
  (\delta\mathbf W_{k+1}^\top\otimes\mathbf I_{n_k}),\qquad k=1,\ldots,d.
\end{align*}
where the matrix sequences $(\mathbf W_k)_{k=1}^{d+1}$ and $(\delta\mathbf W_k)_{k=1}^{d+1}$ are defined by $\mathbf W_{d+1}:=1$, $\delta\mathbf W_{d+1}:=0$, and
\begin{align*}
  \mathbf W_k
  &:=\R(\widetilde{\mathcal U}_k)
  (\mathbf W_{k+1}\otimes\mathbf I_{n_k})\R(\mathcal U_k)^\top,\qquad k=d,\ldots,1,\\
  \delta\mathbf W_k
  &:=\R(\widetilde{\mathcal U}_k)
  (\mathbf W_{k+1}\otimes\mathbf I_{n_k})
  \R(\delta\mathcal U_k)^\top
  +\R(\widetilde{\mathcal U}_k)
  (\delta\mathbf W_{k+1}\otimes\mathbf I_{n_k})
  \R(\mathcal U_k)^\top,
  \qquad k=d,\ldots,1,
\end{align*}
moreover, starting from $\mathbf A_0:=\mathbf 0$ and
$\mathbf B_0:=\mathbf I_{r_0}$, the matrix sequences $(\mathbf A_k)_{k=1}^{d-1}$ and $(\mathbf B_k)_{k=1}^{d-1}$ are defined as
\begin{align*}
  \mathbf A_k
  &:=\ptr_{n_k}\!\left(
  \R(\widetilde{\mathcal U}_k)^\top
  \left(\mathbf A_{k-1}\R(\mathcal U_k)
  +\mathbf B_{k-1}\R(\delta\mathcal U_k)\right)\right)
  +2\delta\mathbf W_{k+1},\qquad k=1,\ldots,d-1,\\
  \mathbf B_k
  &:=\ptr_{n_k}\!\left(
  \R(\widetilde{\mathcal U}_k)^\top
  \mathbf B_{k-1}\R(\mathcal U_k)\right)+2\mathbf W_{k+1},
  \qquad k=1,\ldots,d-1.
\end{align*}
\end{proposition}
\begin{proof}
The two backward recursions compute the mixed products of interface matrices
\begin{equation*}
  \mathbf W_k=\widetilde{\mathbf X}_{\geq k}^\top
  \mathbf X^{}_{\geq k},
  \qquad
  \delta\mathbf W_k=\widetilde{\mathbf X}_{\geq k}^\top
  \delta\mathbf X^{}_{\geq k},
  \qquad k=1,\ldots,d+1.
\end{equation*}
Here the terminal interfaces are $\mathbf X_{\geq d+1}
=\widetilde{\mathbf X}_{\geq d+1}:=1$ and
$\delta\mathbf X_{\geq d+1}:=0$, so the identities hold at $k=d+1$. Assuming them at
$k+1$ and substituting the right-interface recursions and the derivatives
into the left-hand sides gives precisely the recursions defining
$\mathbf W_k$ and $\delta\mathbf W_k$. Backward induction proves the two
identities. The recursion for
$\widetilde{\mathbf L}_k$ is the metric recursion in~\Cref{thm:projection}
at the new base point $\widetilde z$.

Let $\widetilde{\mathbf R}_k$ be the matrices in~\Cref{thm:projection} for
the ambient vector $(\delta\mathcal X,\delta\mathbf P_1,\ldots,
\delta\mathbf P_{d-1})$ at $\widetilde z$. We show by forward induction that
\begin{equation}
  \widetilde{\mathbf R}_k
  =\mathbf A_k\mathbf X_{\geq k+1}^\top
  +\mathbf B_k\delta\mathbf X_{\geq k+1}^\top.
  \label{eq:transport-rk-factorization}
\end{equation}
For $k=0$, this follows from $\mathbf A_0=\mathbf 0$,
$\mathbf B_0=\mathbf I_{r_0}$, and
$\widetilde{\mathbf R}_0=\vecop(\delta\mathcal X)^\top$. The tangent
representation also gives
$\delta\mathbf P_k=-\delta\mathbf X_{\geq k+1}
\mathbf X_{\geq k+1}^\top-\mathbf X_{\geq k+1}
\delta\mathbf X_{\geq k+1}^\top$.
Substituting this identity, the formulas for the mixed interface products above, and the
induction hypothesis at $k-1$ into the recursion for
$\widetilde{\mathbf R}_k$ gives
\begin{align*}
  \widetilde{\mathbf R}_k
  ={}&\left[\ptr_{n_k}\!\left(
  \R(\widetilde{\mathcal U}_k)^\top
  (\mathbf A_{k-1}\R(\mathcal U_k)
  +\mathbf B_{k-1}\R(\delta\mathcal U_k))\right)
  +2\delta\mathbf W_{k+1}\right]
  \mathbf X_{\geq k+1}^\top\\
  &+\left[\ptr_{n_k}\!\left(
  \R(\widetilde{\mathcal U}_k)^\top
  \mathbf B_{k-1}\R(\mathcal U_k)\right)
  +2\mathbf W_{k+1}\right]
  \delta\mathbf X_{\geq k+1}^\top\\
  ={}&\mathbf A_k\mathbf X_{\geq k+1}^\top
  +\mathbf B_k\delta\mathbf X_{\geq k+1}^\top.
\end{align*}
Thus~\eqref{eq:transport-rk-factorization} holds for every $k$.
Using this factorization at $k-1$ gives
\begin{align*}
  \widetilde{\mathbf R}_{k-1}
  (\widetilde{\mathbf X}_{\geq k+1}\otimes\mathbf I_{n_k})=\left(\mathbf A_{k-1}\R(\mathcal U_k)
  +\mathbf B_{k-1}\R(\delta\mathcal U_k)\right)
  (\mathbf W_{k+1}^\top\otimes\mathbf I_{n_k})
  +\mathbf B_{k-1}\R(\mathcal U_k)
  (\delta\mathbf W_{k+1}^\top\otimes\mathbf I_{n_k})
  =\mathbf N_k.
\end{align*}
Substitution into~\Cref{thm:projection}, together with the metric sequence
$\widetilde{\mathbf L}_k$, yields the displayed core formulas. Finally,
orthogonal projection is linear and acts as the identity on its target tangent
space. Hence $\mathcal T_{z\to\widetilde z}$ is linear in $\xi_z$ and
$\mathcal T_{z\to z}$ is the identity on $\mathrm{T}_z\mathcal M_{\mathbf r}^{\TT}$,
which proves the vector-transport claim.
\end{proof}

\subsection{Retraction}
We use the standard definition of a retraction
in~\cite[Definition~4.1.1]{AbsilMahonySepulchre+2008}; see
also~\cite[Definition~3.47]{boumal2020introduction}.
Let $\mathcal M$ be a smooth manifold with tangent bundle $\mathrm{T}\mathcal M$.
A \emph{retraction} on $\mathcal M$ is a smooth map
$\Retr:\mathrm{T}\mathcal M\to\mathcal M$ onto $\mathcal M$ such that, for every
$z\in\mathcal M$, its restriction $\Retr_z:\mathrm{T}_z\mathcal M\to\mathcal M$
satisfies
\begin{equation}\label{eq:retraction}
  \Retr_z(0_z)=z,
  \qquad
  \mathrm D\Retr_z(0_z)=\mathrm{id}_{\mathrm{T}_z\mathcal M},
\end{equation}
where $0_z$ denotes the zero element of $\mathrm{T}_z\mathcal M$, and the second
identity uses the canonical identification
$\mathrm{T}_{0_z}(\mathrm{T}_z\mathcal M)\simeq \mathrm{T}_z\mathcal M$.
Equivalently, for every $\xi_z\in \mathrm{T}_z\mathcal M$, the curve
$t\mapsto\Retr_z(t\xi_z)$ satisfies $\Retr_z(0_z)=z$ and
$\frac{\mathrm d}{\mathrm dt}\Retr_z(t\xi_z)|_{t=0}=\xi_z$.

For the parametrization manifold, the following construction preserves the
right orthogonality of the TT cores.
\begin{proposition}[retraction]\label{thm:retraction}
Let $z=(\mathcal X,\mathbf P_1,\ldots,\mathbf P_{d-1})$ be
represented by right-orthogonal cores
$(\mathcal U_1,\ldots,\mathcal U_d)$, and let tangent vector $\xi_z$ be represented by
$(\delta\mathcal U_1,\ldots,\delta\mathcal U_d)$ satisfying the gauge
conditions.
For $t$ in a neighborhood of zero, define
\begin{align}
  \overline{\mathcal U}_1(t)&:=\mathcal U_1+t\delta\mathcal U_1,
  \nonumber\\
  \R(\overline{\mathcal U}_k(t))^\top
  &:=\left(\R(\mathcal U_k)^\top+t\R(\delta\mathcal U_k)^\top\right)
  \left(\mathbf I+t^2\R(\delta\mathcal U_k)
  \R(\delta\mathcal U_k)^\top\right)^{-1/2},
  \qquad k=2,\ldots,d.
  \label{eq:retraction-cores}
\end{align}
Let $\Retr_z(t\xi_z)$ be the point in $\mathcal M_{\mathbf r}^{\TT}$ generated by these cores. Then $\Retr$ is a retraction.
\end{proposition}
\begin{proof}
Write $\mathbf U_k:=\R(\mathcal U_k)$ and $\delta\mathbf U_k:=\R(\delta\mathcal U_k)$. The right-orthogonality and gauge conditions give
\[
  (\mathbf U_k+t\delta\mathbf U_k)
  (\mathbf U_k+t\delta\mathbf U_k)^\top
  =\mathbf I+t^2\delta\mathbf U_k\delta\mathbf U_k^\top.
\]
Consequently,~\eqref{eq:retraction-cores} satisfies
\[
  \R(\overline{\mathcal U}_k(t))
  \R(\overline{\mathcal U}_k(t))^\top=\mathbf I,
  \qquad k=2,\ldots,d.
\]
Thus the updated cores belong to $\mathcal S_{\mathbf r}^{\TT}$ and generate a point of $\mathcal M_{\mathbf r}^{\TT}$. Clearly $\Retr_z(0_z)=z$.
It remains to verify the differential at zero. The derivative at $t=0$ of
$(\mathbf I+t^2\delta\mathbf U_k\delta\mathbf U_k^\top)^{-1/2}$ is zero. Hence
\[
  \left.\frac{\mathrm d}{\mathrm dt}\overline{\mathcal U}_k(t)\right|_{t=0}
  =\delta\mathcal U_k,
  \qquad k=1,\ldots,d.
\]
Applying~\Cref{thm:tangent-space} to the derivative of the generated point gives
$\frac{\mathrm d}{\mathrm dt}\Retr_z(t\xi_z)|_{t=0}=\xi_z$. Hence both
conditions~\eqref{eq:retraction} hold.
\end{proof}

\section{Optimization for the Rayleigh--Ritz problem}\label{sec:rgd}
The geometric computations developed in the previous section allow the
bounded-rank Rayleigh--Ritz problem to be treated by smooth optimization on
the parametrization manifold. We first derive the Riemannian gradient of the
Rayleigh--Ritz problem and then exploit a Kronecker-product representation of the
operator to evaluate the required products directly from TT cores.

The Rayleigh--Ritz problem~\eqref{eq:final_EEP} is formulated on the parametrization manifold as
\begin{equation}\label{eq:desing-obj}
  \min\,\rho(\phi(z)),\quad\text{s.\,t.}\quad z\in\mathcal M_{\mathbf r}^{\TT}.
\end{equation}
Let
$\mathcal X=\phi(z)$. Since $\mathscr H$ is self-adjoint, differentiating
the Rayleigh quotient gives the ambient Euclidean gradient
\begin{equation}\label{eq:grad}
  \nabla\rho(\mathcal X)
  =\frac{2}{\|\mathcal X\|_{\mathrm{F}}^2}
  \left(\mathscr H\mathcal X-\rho(\mathcal X)\mathcal X\right).
\end{equation}
The objective $F:=\rho\circ\phi$ depends on $z$ only through its tensor
component. Consequently, its ambient gradient on the product space is
$(\nabla\rho(\mathcal X),\mathbf 0,\ldots,\mathbf 0)$, and
\Cref{thm:riemann-grad} yields
\begin{equation}\label{eq:pullback-rgrad}
  \grad F(z)
  =\Proj_z\bigl(\nabla\rho(\mathcal X),\mathbf 0,\ldots,\mathbf 0\bigr).
\end{equation}
Thus, computing the Riemannian gradient reduces to applying $\mathscr H$ to a
TT tensor and projecting the resulting tensor component into
$\mathrm{T}_z\mathcal M_{\mathbf r}^{\TT}$.

\subsection{Efficient computations of gradient}

In many practical applications~\cite{HeinEisertBriegel2004,CockburnXia2022,GongLiuWang2024,CancesKemlinLevitt2024}, the matrix representation of \(\mathscr H\) naturally exhibits a Kronecker-product structure. Exploiting this structure avoids forming \(\mathscr H\mathcal X\) in ambient space and thereby preserves the low-rank structure.
 Specifically, we assume the following
reverse-order Kronecker representation:
\begin{equation}\label{eq:kron-operator}
  \mathbf H=\sum_{\ell=1}^m
  \mathbf H_d^{(\ell)}\otimes\cdots\otimes\mathbf H_1^{(\ell)},
\end{equation}
where $m$ is the number of Kronecker terms, and
$\mathbf H_k^{(\ell)}\in\mathbb R^{n_k\times n_k}$ is the local factor for
mode $k$ in term $\ell$. The reverse order follows from the vectorization
formula in~\Cref{sec:preliminary} and gives the tensor action
$\mathscr H\mathcal X
  =\sum_{\ell=1}^m
  \mathcal X\times_1\mathbf H_1^{(\ell)}\times_2\cdots\times_d\mathbf H_d^{(\ell)}$.
The representation in~\eqref{eq:kron-operator} therefore allows each
Kronecker term to be evaluated separately from its factors without
computing the full matrix
$\mathbf H$ or the full tensor $\mathscr H\mathcal X$.

\begin{lemma}\label{thm:inner-kronecker}
Let $\mathcal X=\mathcal U_1\cdots\mathcal U_d$ and
$\mathcal Y=\mathcal V_1\cdots\mathcal V_d$ be two TT tensors with the same
mode sizes. For each term $\ell=1,\ldots,m$, define the sequence of matrices $(\mathbf C_k^{(\ell)})_{k=1}^{d+1}$ with $\mathbf C_{d+1}^{(\ell)}:=1$ and
$\mathbf C_k^{(\ell)}
  :=\R(\mathcal U_k)
  \left(\mathbf C_{k+1}^{(\ell)}\otimes\mathbf H_k^{(\ell)}\right)
  \R(\mathcal V_k)^\top$,$k=d,\ldots,1$.
Then $\mathbf C_1^{(\ell)}$ is a scalar and
\[
  \langle\mathcal X,\mathscr H\mathcal Y\rangle_{\mathrm{F}}
  =\sum_{\ell=1}^m\mathbf C_1^{(\ell)}.
\]
Since $\mathscr H$ is self-adjoint, the same sum also equals
$\langle\mathcal Y,\mathscr H\mathcal X\rangle_{\mathrm{F}}$.
\end{lemma}
\begin{proof}
For a fixed $\ell$, multiplying $\R(\mathcal U_d)$,
$\mathbf H_d^{(\ell)}$, and $\R(\mathcal V_d)^\top$ gives
$\mathbf C_d^{(\ell)}$. Repeating the matrix products in the recursion
for $k=d-1,\ldots,1$ yields $\mathbf C_1^{(\ell)}
=\langle\mathcal X,
(\mathcal Y\times_1\mathbf H_1^{(\ell)}\times_2\cdots
\times_d\mathbf H_d^{(\ell)})\rangle_{\mathrm{F}}$.
Summing over $\ell$ proves the first identity, and self-adjointness gives the
second.
\end{proof}

The above result computes the inner product of the Rayleigh quotient used
in~\eqref{eq:grad} without forming
$\mathscr H\mathcal X$ in ambient coordinates. Computing~\eqref{eq:pullback-rgrad}
also requires the projection of each operator term. The following result
concerns the Kronecker product in the projection formula, thereby
preserving the TT representation throughout the computation.
\begin{lemma}\label{thm:project-kron-tt}
Let $z\in\mathcal M_{\mathbf r}^{\TT}$ be represented by right-orthogonal cores
$(\mathcal U_1,\ldots,\mathcal U_d)$, and let
$\widetilde{\mathcal X}=\widetilde{\mathcal U}_1\cdots
\widetilde{\mathcal U}_d$ be another TT tensor with the same mode sizes.
In this lemma, $\mathscr H$ denotes a single Kronecker-product operator
with arbitrary matrices $\mathbf H_k\in\mathbb R^{n_k\times n_k}$,
$k=1,\ldots,d$. Its action has the same form $\mathscr H\widetilde{\mathcal X}=
  \widetilde{\mathcal X}\times_1\mathbf H_1
  \times_2\cdots\times_d\mathbf H_d$ as each term
in~\eqref{eq:kron-operator}.
Then $\Proj_z(\mathscr H\widetilde{\mathcal X},\mathbf 0,\ldots,\mathbf 0)$
has the core representation
$(\delta\mathcal U_1,\ldots,\delta\mathcal U_d)$ given by
\begin{align}
  \R(\delta\mathcal U_1)
  &=\mathbf N_1,\nonumber\\
  \R(\delta\mathcal U_k)
  &=\mathbf L_{k-1}^{-1}\mathbf N_k
  \left(\mathbf I_{n_kr_k}
  -\R(\mathcal U_k)^\top\R(\mathcal U_k)\right),
  \qquad k=2,\ldots,d,
  \label{eq:kron-tt-projection}
\end{align}
where $(\mathbf L_k)_{k=0}^{d-1}$ is the matrix sequence
in~\Cref{thm:projection}. The remaining matrix sequences are defined below.
First, starting from $\mathbf W_{d+1}^{\mathscr H}:=1$, define
\begin{equation}
  \mathbf W_k^{\mathscr H}
  :=\R(\mathcal U_k)
  (\mathbf W_{k+1}^{\mathscr H}\otimes\mathbf H_k)
  \R(\widetilde{\mathcal U}_k)^\top,
  \qquad k=d,\ldots,1.
  \label{eq:kron-mixed-interface}
\end{equation}
Next, starting from $\mathbf B_0:=\mathbf I_{r_0}$, let
\begin{equation*}
  \mathbf B_k=\left\langle
  \R(\mathcal U_k)^\top\mathbf B_{k-1}\R(\widetilde{\mathcal U}_k),
  \mathbf H_k\right\rangle^{n_k},
  \qquad k=1,\ldots,d-1.
\end{equation*}
Here $\langle\cdot,\cdot\rangle^{n_k}$ is the partial inner product
defined in~\Cref{sec:matrix}.
Finally, let
\begin{equation*}
  \mathbf N_k=\mathbf B_{k-1}\R(\widetilde{\mathcal U}_k)
  \left((\mathbf W_{k+1}^{\mathscr H})^\top
  \otimes\mathbf H_k^\top\right),
  \qquad k=1,\ldots,d.
\end{equation*}
\end{lemma}
\begin{proof}
Set $\mathbf H_{\geq k}:=\mathbf H_d\otimes\cdots\otimes\mathbf H_k$
and $\mathbf H_{\geq d+1}:=1$, so that
$\mathbf H_{\geq1}$ is the matrix representation of $\mathscr H$.
The backward recursion in~\eqref{eq:kron-mixed-interface} gives
\[
  \mathbf W_k^{\mathscr H}
  =\mathbf X_{\geq k}^\top\mathbf H_{\geq k}
  \widetilde{\mathbf X}_{\geq k},\qquad k=1,\ldots,d.
\]
Indeed, this identity holds at $k=d$, and substituting the right-interface
recursions into the right-hand side proves it by backward induction.
Thus these products are evaluated without forming $\mathscr H\widetilde{\mathcal X}$.

Let $\mathbf R_k$ be the matrices in~\Cref{thm:projection} for the ambient vector
$(\mathscr H\widetilde{\mathcal X},\mathbf 0,\ldots,\mathbf 0)$. We show that
$\mathbf R_k=\mathbf B_k\widetilde{\mathbf X}_{\geq k+1}^\top
\mathbf H_{\geq k+1}^\top$ for $k=0,\ldots,d-1$.
For $k=0$, this follows from $\mathbf B_0=\mathbf I_{r_0}$ and
$\mathbf R_0=\vecop(\mathscr H\widetilde{\mathcal X})^\top
=\widetilde{\mathbf X}_{\geq1}^\top
\mathbf H_{\geq1}^\top$. If the factorization holds at $k-1$,
substituting the right-interface recursion and
$\mathbf H_{\geq k}=\mathbf H_{\geq k+1}\otimes\mathbf H_k$
into the partial-trace recursion for $\mathbf R_k$ and applying the
partial trace identity in~\Cref{sec:matrix} gives
\[
  \mathbf R_k
  =\left\langle
  \R(\mathcal U_k)^\top\mathbf B_{k-1}\R(\widetilde{\mathcal U}_k),
  \mathbf H_k\right\rangle^{n_k}
  \widetilde{\mathbf X}_{\geq k+1}^\top
  \mathbf H_{\geq k+1}^\top
  =\mathbf B_k\widetilde{\mathbf X}_{\geq k+1}^\top
  \mathbf H_{\geq k+1}^\top.
\]
The claimed factorization therefore follows
by induction. Multiplying its instance at $k-1$ by
$\mathbf X_{\geq k+1}\otimes\mathbf I_{n_k}$ and using the definition
of $\mathbf W_{k+1}^{\mathscr H}$ gives $\mathbf N_k$. Substitution into
\Cref{thm:projection} proves
\eqref{eq:kron-tt-projection}.
\end{proof}

\begin{proposition}\label{thm:riemann-grad-kronecker}
Let $z\in\mathcal M_{\mathbf r}^{\TT}$ and
$\mathcal X=\phi(z)$. For the operator in~\eqref{eq:kron-operator},
the Riemannian gradient of $F:=\rho\circ\phi$ is
\begin{equation}
  \grad F(z)=\frac{2}{\|\mathcal X\|_{\mathrm{F}}^2}
  \left(\sum_{\ell=1}^m\xi^{(\ell)}
  -\rho(\mathcal X)\xi^{(0)}\right),
  \label{eq:kron-rgrad-projections}
\end{equation}
where $\xi^{(0)}:=\Proj_z(\mathcal X,\mathbf 0,\ldots,\mathbf 0)$ and
$\xi^{(\ell)}:=\Proj_z(\mathcal X\times_1\mathbf H_1^{(\ell)}
\times_2\cdots\times_d\mathbf H_d^{(\ell)},\mathbf 0,\ldots,\mathbf 0)$
for $\ell=1,\ldots,m$.
\end{proposition}
\begin{proof}
In view of~\eqref{eq:grad} and~\eqref{eq:pullback-rgrad}, the Riemannian gradient is the
projection of the ambient vector with tensor component
$2(\mathscr H\mathcal X-\rho(\mathcal X)\mathcal X)
/\|\mathcal X\|_{\mathrm{F}}^2$ and zero slack components. Linearity of
$\Proj_z$ and the sum in~\eqref{eq:kron-operator} give
\eqref{eq:kron-rgrad-projections}.
Each $\xi^{(\ell)}$ is computed by~\Cref{thm:project-kron-tt} with
$\widetilde{\mathcal X}=\mathcal X$ and
$\mathbf H_k=\mathbf H_k^{(\ell)}$. Taking $\mathbf H_k=\mathbf I_{n_k}$
for every $k$ gives $\xi^{(0)}$.
\end{proof}

Equation~\eqref{eq:kron-rgrad-projections} projects the Kronecker terms and
$\mathcal X$ separately and combines the tangent-core representations only
after the projections have been computed. It therefore avoids constructing
the potentially higher-rank TT representation of $\mathscr H\mathcal X$.

\subsection{Riemannian algorithms}
The Riemannian gradient, retraction, and vector transport developed above
provide the operations needed for the proposed low-rank tensor train Riemannian gradient descent~(LTT-RGD) and low-rank tensor train
Riemannian conjugate gradient~(LTT-RCG).

LTT-RGD takes the negative Riemannian gradient as its search direction.
Armijo backtracking selects a step size that gives sufficient decrease in
the Rayleigh quotient, and the retraction maps the tangent step back to the parametrization
manifold. These steps are summarized in~\Cref{alg:rgd}.
LTT-RCG also uses the previous search direction. Since the search
directions belong to different tangent spaces, the previous direction and
gradient are first transported to the current tangent space using
\Cref{thm:vector-transport}. The new direction combines the negative current
gradient with the transported direction.
\Cref{alg:rcg} gives the resulting iteration, which uses the same line search
and retraction as LTT-RGD.
Both algorithms stop when the Riemannian gradient norm reaches its prescribed
tolerance or the iteration limit is reached.
Under the regularity and backtracking assumptions
in~\cite[Corollary~4.13]{boumal2020introduction}, the Riemannian gradient
norm in LTT-RGD tends to zero. For LTT-RCG, the general line-search
theory ensures that every accumulation point is stationary, if the search directions satisfy the gradient-related
condition~\cite[Theorem~4.3.1]{AbsilMahonySepulchre+2008}.

\begin{algorithm}[htbp]
  \SetAlgoLined
  \caption{Low-rank tensor train Riemannian gradient descent method (LTT-RGD)}
  \label{alg:rgd}
  \KwIn{Initial point $z^{(0)}\in\mathcal M_{\mathbf r}^{\TT}$.}
  \For{$t=0,1,\ldots$ until termination}{
    Compute $g^{(t)}=\grad F(z^{(t)})$ using~\Cref{thm:riemann-grad-kronecker}\;
    \If{the gradient tolerance is satisfied or the iteration limit is reached}{
      Stop\;
    }
    Set $\xi^{(t)}=-g^{(t)}$\;
    Select a step size $\alpha_t$ by Armijo backtracking\;
    Set $z^{(t+1)}=\Retr_{z^{(t)}}(\alpha_t\xi^{(t)})$\;
  }
  \KwOut{Final iterate $z^{(t)}$ and approximate eigenpair
  $(\rho(\phi(z^{(t)})),\phi(z^{(t)}))$.}
\end{algorithm}

\begin{algorithm}[htbp]
  \SetAlgoLined
  \caption{Low-rank tensor train Riemannian
conjugate gradient method (LTT-RCG)}
  \label{alg:rcg}
  \KwIn{Initial point $z^{(0)}\in\mathcal M_{\mathbf r}^{\TT}$.}
  Compute $g^{(0)}=\grad F(z^{(0)})$ and set $\xi^{(0)}=-g^{(0)}$\;
  \For{$t=0,1,\ldots$ until termination}{
    \If{the gradient tolerance is satisfied or the iteration limit is reached}{
      Stop\;
    }
    Select $\alpha_t$ by Armijo backtracking and set
    $z^{(t+1)}=\Retr_{z^{(t)}}(\alpha_t\xi^{(t)})$\;
    Compute $g^{(t+1)}=\grad F(z^{(t+1)})$\;
    Transport $\widetilde g^{(t)}=\mathcal T_{z^{(t)}\to z^{(t+1)}}(g^{(t)})$ and
    $\widetilde\xi^{(t)}=\mathcal T_{z^{(t)}\to z^{(t+1)}}(\xi^{(t)})$\;
    Compute $\beta_t$ using the chosen conjugate-gradient formula\;
    Set $\xi^{(t+1)}=-g^{(t+1)}+\beta_t\widetilde\xi^{(t)}$\;
    \If{$\xi^{(t+1)}$ is not a descent direction}{
      Set $\xi^{(t+1)}=-g^{(t+1)}$\;
    }
  }
  \KwOut{Final iterate $z^{(t)}$ and approximate eigenpair
  $(\rho(\phi(z^{(t)})),\phi(z^{(t)}))$.}
\end{algorithm}

In practice, the implementation reduces the cost of these iterations by carrying out
the required operations with TT cores and small matrices. The backward
products $\mathbf W_k^{\mathscr H}$ are computed once for each Kronecker
term and reused during the forward evaluation of the projected gradient.
The partial trace and partial inner product are evaluated by reshaping
the matrix arguments and using dense matrix multiplication, avoiding
separate operations on individual blocks.

The positive-definite matrices $\mathbf L_k$ in~\Cref{thm:projection}
are handled through matrix factorizations. In the gradient computation,
an SVD of the reshaped product of a factor of $\mathbf L_{k-1}$ and
$\R(\mathcal U_k)$ gives a factor of $\mathbf L_k$. The corresponding
inverse is obtained using the inverse of singular values $2+\sigma_j^2$, avoiding a
separate matrix inversion.
The retraction is also evaluated using an SVD. Set
$\delta\mathbf U_k:=\R(\delta\mathcal U_k)$, and let
$\delta\mathbf U_k=\mathbf Q_k\mathbf\Sigma_k\mathbf V_k^\top$ be an SVD
with $\mathbf Q_k\in\mathbb R^{r_{k-1}\times r_{k-1}}$ orthogonal and
$\mathbf\Sigma_k\in\mathbb R^{r_{k-1}\times r_{k-1}}$ diagonal,
including any zero singular values. The inverse square root in~\eqref{eq:retraction-cores} can then be
evaluated as
\begin{equation}\label{eq:retraction-svd}
  (\mathbf I+t^2\delta\mathbf U_k\delta\mathbf U_k^\top)^{-1/2}
  =\mathbf Q_k(\mathbf I+t^2\mathbf\Sigma_k^2)^{-1/2}\mathbf Q_k^\top.
\end{equation}
Equation~\eqref{eq:retraction-svd} reduces the inverse square root to diagonal
operations on the singular values. During backtracking, the search direction
is fixed, so its SVD can be reused for all trial step sizes; only the diagonal
factors change.

\subsection{Computational complexity}

We estimate the cost of the main operations and one iteration of LTT-RGD
and LTT-RCG. Let $n_1=\cdots=n_d=n$ and $r_1=\cdots=r_{d-1}=r$.
The matrix representation of $\mathscr H$ in~\eqref{eq:kron-operator}
contains $m$ Kronecker-product terms. \Cref{tab:complexity} summarizes
the computational complexity, with $s$ denoting the number of Armijo
trial steps, including the accepted step.
\begin{table}[H]
  \caption{Computational complexity of the main operations and one iteration
  of LTT-RGD and LTT-RCG, with mode size $n$, rank parameter $r$,
  $m$ Kronecker-product terms, and $s$ Armijo trial steps.}\label{tab:complexity}
  \centering
  \begin{tabular}{@{}lllc@{}}
    \toprule
    Operation & Notation & Complexity & Reference \\
    \midrule
    Rayleigh quotient & $\rho(\mathcal X)$ & $\mathcal O(dmn^2r^3)$ & \Cref{thm:inner-kronecker} \\
    Riemannian gradient & $\grad F(z)$ & $\mathcal O(dmn^2r^3)$ & \Cref{thm:riemann-grad-kronecker} \\
    Vector transport & $\mathcal T_{z\to\widetilde z}$ & $\mathcal O(dn^2r^3)$ & \Cref{thm:vector-transport} \\
    Riemannian metric & $\langle\xi,\eta\rangle$ & $\mathcal O(dnr^3)$ & \Cref{thm:inner} \\
    Retraction & $\Retr_z$ & $\mathcal O(dnr^3)$ & \Cref{thm:retraction} \\
    \midrule
    LTT-RGD & --- & $\mathcal O((s+1)dmn^2r^3)$ & \Cref{alg:rgd} \\
    LTT-RCG & --- & $\mathcal O((s+1)dmn^2r^3)$ & \Cref{alg:rcg} \\
    \bottomrule
  \end{tabular}
\end{table}

For $\mathbf A\in\mathbb R^{r\times nr}$ and
$\mathbf B\in\mathbb R^{r\times r}$, the product
$\mathbf A(\mathbf B\otimes\mathbf I_n)$ costs $\mathcal O(nr^3)$
when the computation is evaluated by reshaping $\mathbf A$
and multiplying smaller matrices.
The matrix products in the Rayleigh-quotient and gradient formulas multiply
$r\times nr$ matrices by $nr\times nr$ matrices, costing
$\mathcal O(n^2r^3)$ per mode and Kronecker term. Summing over $d$ modes
and $m$ terms gives $\mathcal O(dmn^2r^3)$.
Vector transport uses projection matrices of the same size without
a sum over Kronecker terms, giving $\mathcal O(dn^2r^3)$.
The matrix recursions for the Riemannian metric and retraction use core products and SVDs of $r\times nr$ matrices, costing $\mathcal O(dnr^3)$.

Each Armijo trial evaluates a retraction and a Rayleigh quotient,
and each iteration computes one Riemannian gradient.
LTT-RCG additionally evaluates vector transport, Riemannian
inner product, and direction update, which do not change the
overall complexity of $\mathcal O((s+1)dmn^2r^3)$ per iteration.
The TT cores and Kronecker factors require $\mathcal O(dnr^2)$ and
$\mathcal O(dmn^2)$ storage, respectively.
For fixed $m,n,r$, and $s$, both the iteration cost and representation
storage grow linearly with $d$, without forming tensors of size $n^d$.

\section{Numerical experiments}\label{sec:experiment}

In this section, we compare the proposed methods with the existing solvers for different problems: 1) the harmonic oscillator first
tests LTT-RGD and LTT-RCG against a known ground state; 
2) the Laplace problem then compares LTT-RCG with full-space eigensolvers;
3) the Schr\"odinger problem extends the comparison to larger tensor orders;
4) A layered cluster problem then tests nonseparable ground states with known TT ranks;
5) finally, a Bose--Einstein model tests LTT-RCG as the inner solver of a nonlinear iteration.
The experiments are run in
MATLAB R2019b under Ubuntu 22.04.3 on a workstation with two Intel Xeon Gold
6330 processors (28 cores at 2.00 GHz per processor, 42 MB cache) and 512 GB
of RAM.
The code that produced the results is available at \url{https://github.com/Pengfei-Hao/LTT-EVPs}.

For the linear eigenvalue comparisons, we use LOBPCG
(locally optimal block preconditioned conjugate gradient)~\cite{Knyazev2001}
from BLOPEX\footnote{Available from~\url{https://github.com/lobpcg/blopex}.} and
PRIMME (preconditioned iterative multimethod eigensolver)~\cite{StathopoulosMcCombs2010},
using its MATLAB interface\footnote{Available from~\url{https://github.com/primme/primme}.}.
These methods store full vectors, whereas LTT-RGD and LTT-RCG use TT cores.
In each test, all methods use the same discrete operator, reference
eigenpair, and initial tensor, represented as a full vector for LOBPCG and
PRIMME. The comparison runs use a common residual tolerance and
iteration limit. Running time is the main measure of efficiency,
since an iteration involves different operations in the different methods.
The reported TT ranks are numerical ranks of the represented tensors;
the prescribed core dimensions remain fixed. Eigenvectors are normalized
and signs aligned with the reference before comparing slices.
The problem sizes and rank parameters are specified below.
Throughout the experiments, a scalar rank parameter $r$ denotes the rank vector 
$\mathbf r=(1,r,\ldots,r,1)$
whenever the rank satisfies the admissible condition; otherwise, the components are reduced to satisfy
the admissibility bounds in~\eqref{eq:tt-rank-admissibility}.

The energy of an iterate $\mathcal X^{(t)}$ is its Rayleigh quotient
$\lambda_t:=\rho(\mathcal X^{(t)})$. The energy error
$|\lambda_t-\lambda_{\mathrm{ref}}|$ is exactly the absolute error in the
estimated smallest eigenvalue, where $\lambda_{\mathrm{ref}}$ is the
reference value. Unless otherwise stated, we report the relative energy error
\[
  \frac{|\lambda_t-\lambda_{\mathrm{ref}}|}{|\lambda_{\mathrm{ref}}|}.
\]
For a fixed nonzero reference value, the relative and absolute errors
differ only by the constant factor $1/|\lambda_{\mathrm{ref}}|$.
The cluster problem has $\lambda_{\mathrm{ref}}=\lambda_*=0$, so we use
the absolute energy error $|\rho(\mathcal X^{(t)})|$ instead. To assess how well a tensor
satisfies the eigenvalue equation, we also use the normalized full-space residual
\begin{equation}\label{eq:eigen-residual}
  \eta_{\mathrm{res}}(\mathcal X)
  :=\frac{\|\mathscr H\mathcal X-\rho(\mathcal X)\mathcal X\|_{\mathrm{F}}}
  {\|\mathcal X\|_{\mathrm{F}}}.
\end{equation}
This quantity measures eigenpair accuracy in the ambient space. A small
Riemannian gradient only indicates approximate stationarity of the optimization problem on the parametrization manifold, whereas a small residual indicates that the
computed tensor approximately satisfies the full eigenvalue equation.

\subsection{Harmonic oscillator}\label{sec:exp-harmonic}
The quantum harmonic oscillator provides a model problem for
the numerical approximation of eigenvalues~\cite{CockburnXia2022}.
Its separable ground state provides a reference for comparing LTT-RGD
and LTT-RCG, to exhibit numerical rank reduction, and to measure computational cost.
In the truncated product eigenbasis, the $d$-dimensional operator $\mathscr H$ has the matrix representation
\begin{equation*}
\mathbf H=\sum_{i=1}^{d}
\mathbf I_{n_d}\otimes\cdots\otimes\mathbf I_{n_{i+1}}\otimes
\mathbf H_i\otimes
\mathbf I_{n_{i-1}}\otimes\cdots\otimes\mathbf I_{n_1},
\end{equation*}
where $\mathbf H_i=\diag(1/2,\ldots,n_i-1/2)$ and
$\mathbf I_{n_i}$ is the identity of size $n_i$. The smallest eigenvalue is
$\lambda_*=d/2$, and the corresponding eigenvector is the tensor product of the  ground-state eigenvector of $\mathbf H_i$. The TT rank of the eigenvector is therefore $(1,1,\ldots,1)$.

The convergence test uses $d=3$, $n_k=20$, and TT-rank parameter
$(1,4,4,1)$, which exceeds the rank of the exact eigenvector.
\Cref{fig:harmonic-oscillator} shows the energy error, residual, and
numerical TT rank of LTT-RCG against iteration count. LTT-RCG reaches comparable
final accuracy with substantially fewer iterations than LTT-RGD and
reduces the numerical TT ranks to $(1,1,1,1)$, as expected for the
separable ground state. 

\begin{figure}[htbp]

\centering
\includegraphics[width=0.9\linewidth]{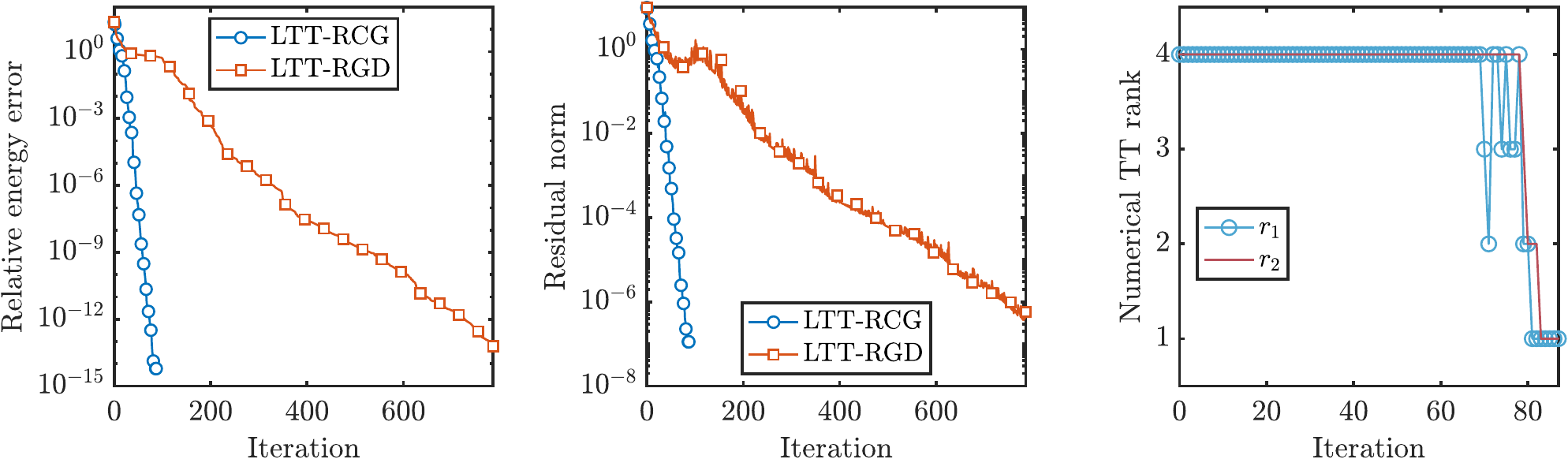}
\caption{Convergence and numerical TT-rank histories for the harmonic oscillator
with $d=3$, $n_k=20$, and TT-rank parameter $(1,4,4,1)$.}
\label{fig:harmonic-oscillator}
\end{figure}

We next measure time per iteration while varying one parameter at a time:
$d=3,\ldots,128$ with $n_k=10$ and rank parameter one;
$n_k=10,\ldots,80$ with $d=4$ and rank parameter one; and
$r=1,\ldots,16$ with $d=4$ and $n_k=10$.
The dimension sweep uses an iteration limit of 50, while the rank sweep
is run to convergence. \Cref{fig:harmonic-scaling} shows the effects
of tensor order, mode size, and rank parameter on the cost of an iteration.
LTT-RCG has slightly higher computation cost per iteration than LTT-RGD, but both methods scale polynomially with tensor order $d$, mode size $n$, and  rank parameter $r$, as expected from the complexity analysis in~\Cref{tab:complexity},  avoiding the curse of dimensionality.

\begin{figure}[htbp]
\centering
\includegraphics[width=0.9\linewidth]{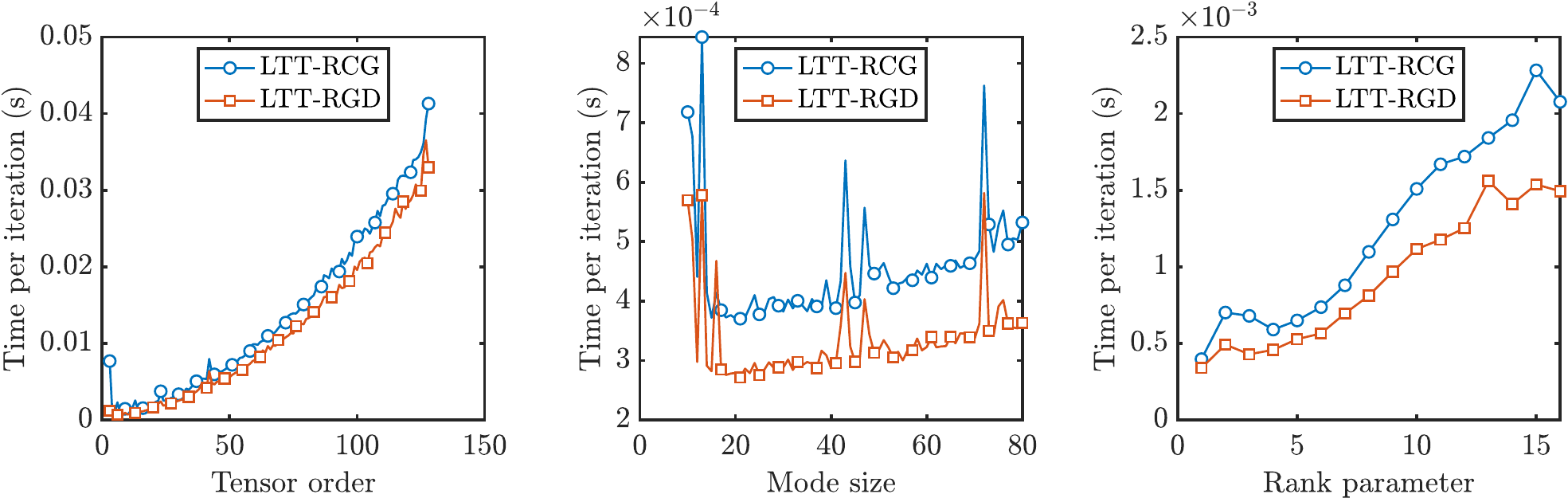}
\caption{Time per iteration for the harmonic oscillator as the tensor order,
mode size, and rank parameter vary. The sweeps use $d=3,\ldots,128$
with $n_k=10$ and $r=1$; $n_k=10,\ldots,80$ with $d=4$ and $r=1$;
and $r=1,\ldots,16$ with $d=4$ and $n_k=10$, respectively.}
\label{fig:harmonic-scaling}
\end{figure}

\subsection{Laplace eigenvalue problem}\label{sec:exp-laplace}
The Dirichlet Laplacian determines the vibration modes of a membrane
with a fixed boundary~\cite{GongLiuWang2024}. We use the separable
$d$-dimensional form to compare LTT-RCG, LOBPCG, and PRIMME against
a known discrete eigenpair. Consider
\begin{equation*}
\begin{cases}
-\Delta\phi=\lambda\phi,&\text{in}\ \Omega,\\
\phi=0,&\text{on}\ \partial\Omega,
\end{cases}
\end{equation*}
where $\Omega={\left[0,\pi\right]}^{d}$. On a uniform grid, with stepsize $h=\pi/(n_k+1)$ in each mode, we use the
finite-difference matrix where the common factor $h^{-2}$ is omitted:
\begin{equation*}
\mathbf H=\sum_{i=1}^{d}
\mathbf I_{n_d}\otimes\cdots\otimes\mathbf I_{n_{i+1}}\otimes
\mathbf H_i\otimes
\mathbf I_{n_{i-1}}\otimes\cdots\otimes\mathbf I_{n_1},
\qquad
\mathbf H_i=\operatorname{tridiag}(-1,2,-1).
\end{equation*}
For the one-dimensional tridiagonal matrix $\mathbf H _i$ of size $n$, the smallest eigenvalue is $2-2\cos(\pi/(n+1))$; hence the reference eigenvalue of the separable $d$-dimensional operator is the sum of the one-dimensional smallest eigenvalues. The corresponding eigenvector is similarly separable, with TT rank $(1,1,\ldots,1)$.

The comparison uses $d=4$, $n_k=30$, and TT-rank parameter $(1,2,2,2,1)$.
The reference eigenvalue belongs to the discrete matrix, so the reported
error excludes discretization error. \Cref{fig:laplace-diagnostics}
compares energy errors and residuals against time and shows the
numerical TT ranks of LTT-RCG.
All three methods reach comparable final energy errors and residuals,
while LTT-RCG requires less computation time. Its numerical TT ranks
decrease to $(1,1,1,1,1)$, as expected for the exact rank of the  eigenvector.
\begin{figure}[htbp]
\centering
\includegraphics[width=0.90\linewidth]{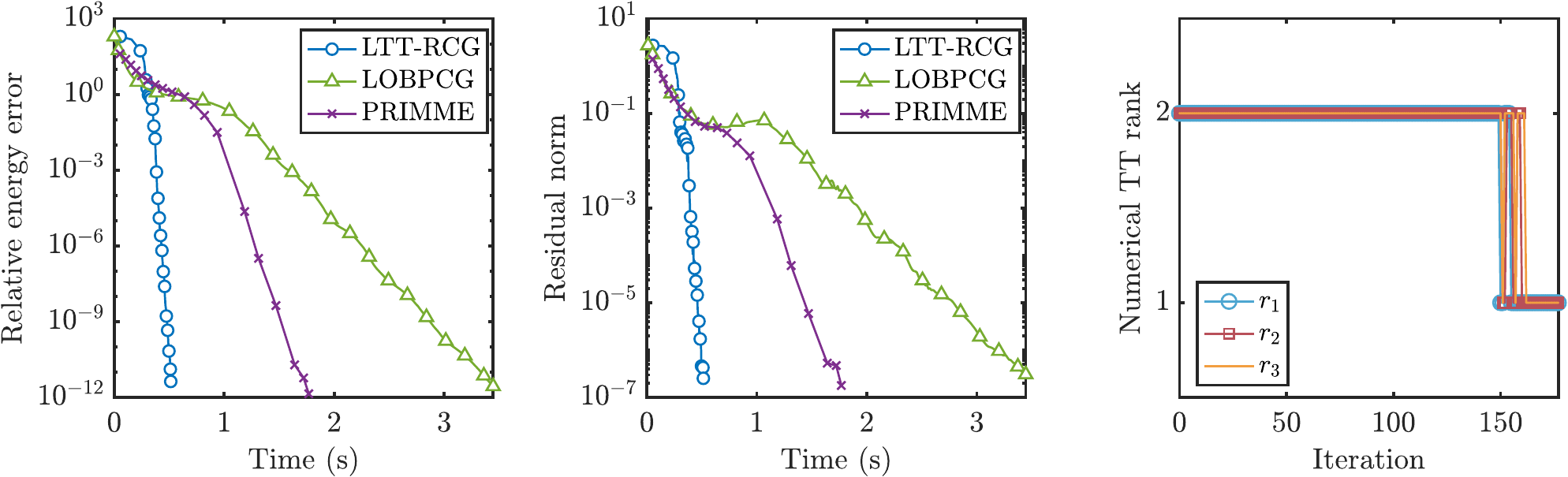}

\caption{
Convergence of LTT-RCG, LOBPCG, and PRIMME and numerical
TT ranks of LTT-RCG for the Laplace problem with $d=4$, $n_k=30$,
and TT-rank parameter $(1,2,2,2,1)$.
}
\label{fig:laplace-diagnostics}
\end{figure}

\Cref{fig:laplace-eigenvector-error} compares the mode-1 marginal
density and central-slice errors of the computed eigenvectors.
The marginal densities agree with the reference solution, while the
slice errors show the entrywise accuracy.

\begin{figure}[htbp]
\centering
\includegraphics[width=\linewidth]{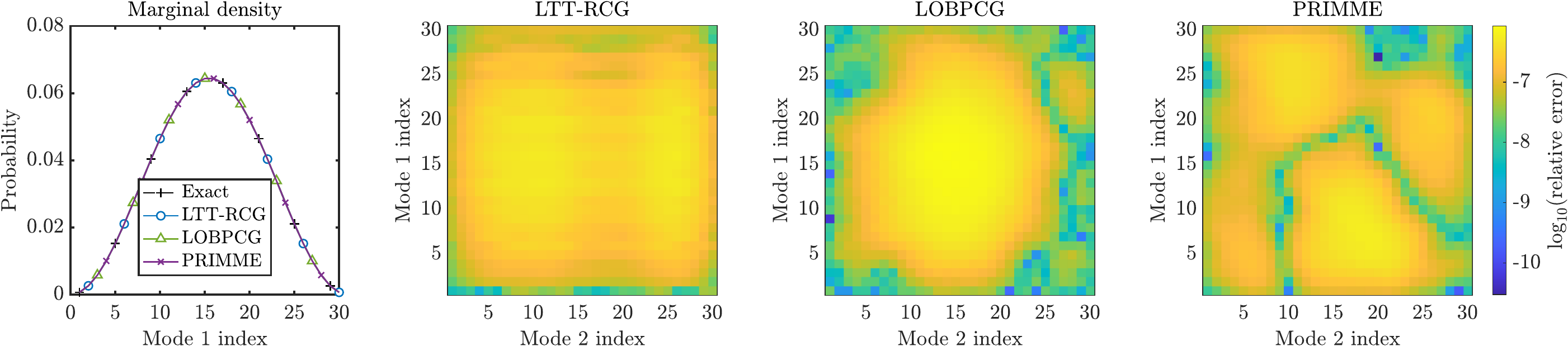}
\caption{Mode-1 marginal density and entrywise absolute errors of central
Laplace eigenvector slices.}
\label{fig:laplace-eigenvector-error}
\end{figure}

\subsection{Schrödinger eigenvalue problem}\label{sec:exp-schrodinger}
Periodic Schr\"odinger eigenvalue problems arise in electronic-structure
calculations for crystalline materials~\cite{CancesKemlinLevitt2024}.
 We use a separable cosine potential
to compare the eigensolvers and test computational cost at larger
tensor orders. The eigenvalue problem is
\begin{equation*}
\begin{aligned}
-\Delta\phi+V\phi=E\phi,&\quad\text{in}\ \Omega,
\end{aligned}
\end{equation*}
with periodic boundary conditions on
$\Omega=[0,2\pi]^d$ and potential
$V(x):=-\sum_{k=1}^d c_k\cos(x_k)$. We use a Fourier spectral
discretization with frequencies $\{-N,\ldots,N\}$ in each mode, so that
$n_k=n=2N+1$. The resulting Kronecker-sum matrix is
\begin{equation*}
\mathbf H=\sum_{i=1}^{d}
\mathbf I_{n_d}\otimes\cdots\otimes\mathbf I_{n_{i+1}}\otimes
\mathbf H_i\otimes
\mathbf I_{n_{i-1}}\otimes\cdots\otimes\mathbf I_{n_1},
\end{equation*}
where the tridiagonal matrices $\mathbf H_k\in\mathbb R^{n\times n}$ have entries
\begin{equation*}
  (\mathbf H_k)_{ab}:=
  \begin{cases}
    (a-N-1)^2, & a=b,\\
    -c_k/2, & |a-b|=1,\\
    0, & \text{otherwise},
  \end{cases}
  \qquad a,b=1,\ldots,n.
\end{equation*}
Throughout this experiment, $N=7$ and
$c_k:=0.2+0.6(k-1)/(d-1)$ for $k=1,\ldots,d$.

The fixed-size comparison uses $d=5$, $n_k=15$, and TT-rank parameter
$(1,2,2,2,2,1)$.
Because this Hamiltonian is a
Kronecker sum, its reference ground-state energy and eigenvector are assembled
from high-accuracy eigensolutions of the matrices $\mathbf{H}_k$. LTT-RCG is
compared with full-vector LOBPCG and PRIMME using the same initial guesses.
\Cref{fig:schrodinger} reports energy errors and residuals against
time, together with numerical TT ranks against iteration count.
All three solvers reach comparable final energy errors and residuals,
while LTT-RCG requires less computation time.
\begin{figure}[htbp]
\centering
\includegraphics[width=0.90\linewidth]{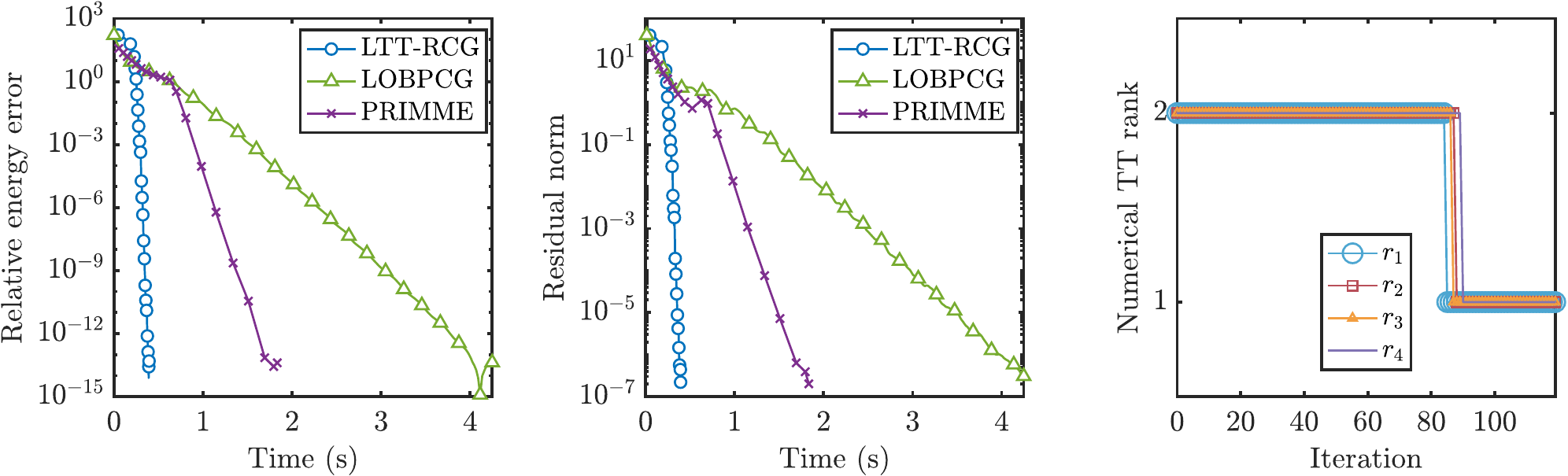}
\caption{
Convergence of LTT-RCG, LOBPCG, and PRIMME and numerical
TT ranks for the 
Schrödinger problem with $n_k=15$ and TT-rank parameter $(1,2,2,2,2,1)$.}
\label{fig:schrodinger}
\end{figure}

\Cref{fig:schrodinger-eigenvector-error} compares the mode-1 marginal
Fourier-space density with the reference and reports central-slice errors.
The Fourier-space marginal density computed by LTT-RCG reproduces the concentration near zero frequency in the reference solution, while its slice errors indicate accurate recovery of the ground-state eigenvector.

\begin{figure}[htbp]
\centering
\includegraphics[width=\linewidth]{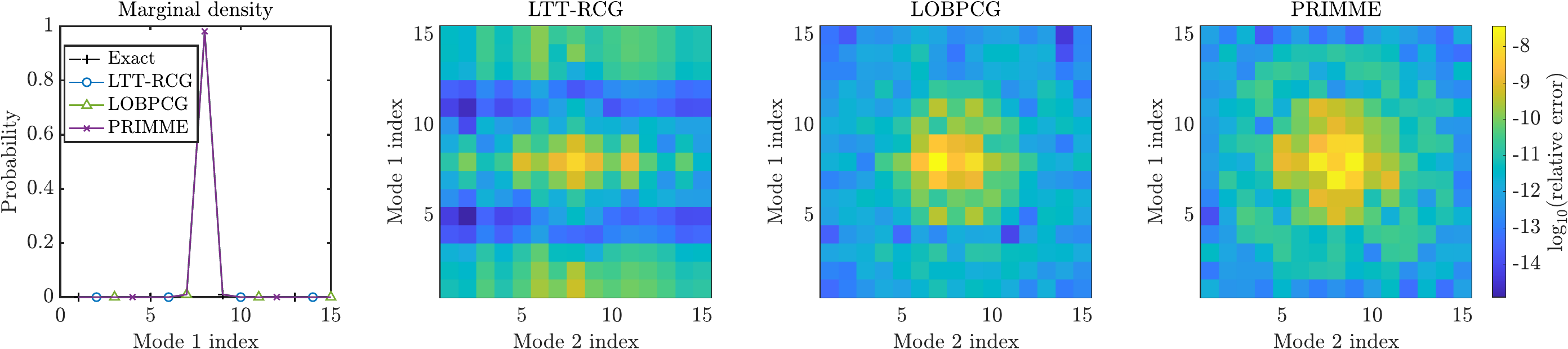}
\caption{Mode-1 marginal Fourier-space density and entrywise absolute
errors of central Schrödinger eigenvector slices.}
\label{fig:schrodinger-eigenvector-error}
\end{figure}

Finally, we increase the tensor order while keeping $n_k=15$.
\Cref{tab:schrodinger-cases} reports mean times and final residuals over
three repetitions, with the rank parameters listed separately for each test.
The memory column gives the size of one full vector, not the total memory
used by a solver. Full-space comparisons are omitted above the
problem-size limit.
The full-space solvers are faster at some small tensor orders, whereas
LTT-RCG becomes faster as the full-vector size increases and remains
applicable beyond the problem-size limit. The final
residual increases at the largest orders. Omitted baseline entries indicate the size limit,
not an algorithmic failure.
\begin{table}[htbp]
\caption{Mean timing and final normalized residual over three Schrödinger
repetitions with TT-rank parameter $(1,r,\ldots,r,1)$.
Times are in seconds, and memory is the
storage required by one full double-precision vector in decimal units. A dash
denotes that the full-space method exceeded the size limit.}
\label{tab:schrodinger-cases}
\centering
\setlength{\tabcolsep}{4pt}
\begin{tabular}{@{}rrlrrrrrrr@{}}
\toprule
\multirow{2}{*}{$d$} & \multirow{2}{*}{$r$} & \multicolumn{1}{c}{\multirow{2}{*}{$n^d$}} & \multirow{2}{*}{Memory} & \multicolumn{2}{c}{LTT-RCG} & \multicolumn{2}{c}{LOBPCG} & \multicolumn{2}{c}{PRIMME} \\
\cmidrule(lr){5-6}\cmidrule(lr){7-8}\cmidrule(lr){9-10}
 &  &  &  & Time (s) & Residual & Time (s) & Residual & Time (s) & Residual \\
\midrule
2   & 2 & $2.25\times10^{2}$   & 1.8 KB  & $0.168$ & $2.56\times10^{-7}$ & $0.079$ & $4.49\times10^{-7}$ & $0.006$ & $2.25\times10^{-7}$ \\
3   & 2 & $3.38\times10^{3}$   & 27 KB   & $0.115$ & $2.48\times10^{-7}$ & $0.114$ & $3.24\times10^{-7}$ & $0.007$ & $2.10\times10^{-7}$ \\
4   & 2 & $5.06\times10^{4}$   & 405 KB  & $0.126$ & $2.51\times10^{-7}$ & $1.067$ & $3.25\times10^{-7}$ & $0.515$ & $2.34\times10^{-7}$ \\
5   & 2 & $7.59\times10^{5}$   & 6.08 MB & $0.155$ & $2.73\times10^{-7}$ & $4.193$ & $3.22\times10^{-7}$ & $1.832$ & $2.01\times10^{-7}$ \\
6   & 4 & $1.14\times10^{7}$   & 91.1 MB & $0.310$ & $2.42\times10^{-7}$ & $33.419$ & $3.08\times10^{-7}$ & $27.291$ & $2.25\times10^{-7}$ \\
8   & 4 & $2.56\times10^{9}$   & 20.5 GB & $0.289$ & $2.49\times10^{-7}$ & -- & -- & -- & -- \\
16  & 8 & $6.57\times10^{18}$  & 52.5 EB & $3.099$ & $3.79\times10^{-7}$ & -- & -- & -- & -- \\
32  & 8 & $4.31\times10^{37}$  & $>1$ YB & $15.276$ & $1.10\times10^{-6}$ & -- & -- & -- & -- \\
64  & 2 & $1.86\times10^{75}$  & $>1$ YB & $10.218$ & $1.34\times10^{-6}$ & -- & -- & -- & -- \\
128 & 1 & $3.46\times10^{150}$ & $>1$ YB & $33.530$ & $3.37\times10^{-6}$ & -- & -- & -- & -- \\
\bottomrule
\end{tabular}
\end{table}

\subsection{Layered cluster eigenvalue problem}\label{sec:exp-cluster}

The preceding linear examples have rank-one eigenvectors associated with
the smallest eigenvalues. We next consider a structured matrix motivated
by cluster models in quantum spin systems~\cite{HeinEisertBriegel2004}.
Its smallest eigenpair is explicitly known, and the exact TT ranks of the
eigenvector are controlled by the number of layers. This provides a test
of the methods for nonseparable solutions.

Let $d\geq2$ be the tensor order and $p\geq1$ the number of layers. Set
$q:=2^p$ and take $n_k=n\geq q$. For each mode, let matrix
$\mathbf B=[\mathbf b_0,\ldots,\mathbf b_{q-1}]\in\mathbb R^{n\times q}$
with orthonormal columns given by the first $q$ cosine modes:
\[
  (\mathbf b_a)_j:=
  \begin{cases}
    n^{-1/2}, & a=0,\\
    (2/n)^{1/2}\cos\bigl(\pi a(j-\tfrac12)/n\bigr), & 1\leq a<q,
  \end{cases}
  \qquad j=1,\ldots,n.
\]
Index these columns also by binary vectors $\mathbf s\in\{0,1\}^p$,
writing $\mathbf b_{\mathbf s}:=\mathbf b_{a(\mathbf s)}$ with
$a(\mathbf s):=\sum_{\ell=1}^p2^{\ell-1}s_\ell$. Define the local
$n\times n$ matrices
\[
  \mathbf P:=\mathbf B\mathbf B^\top,\qquad
  \mathbf X_\ell:=\sum_{\mathbf s\in\{0,1\}^p}
    \mathbf b_{\mathbf s\oplus\mathbf e_\ell}\mathbf b_{\mathbf s}^\top,
  \qquad
  \mathbf Z_\ell:=\sum_{\mathbf s\in\{0,1\}^p}
    (-1)^{s_\ell}\mathbf b_{\mathbf s}\mathbf b_{\mathbf s}^\top,
\]
where $\ell=1,\ldots,p$, $\mathbf e_\ell$ is the $\ell$-th coordinate
vector in $\mathbb R^p$, and $\oplus$ denotes componentwise addition modulo two, thus $\mathbf s\oplus\mathbf e_\ell$ flips the $\ell$-th bit. The matrix $\mathbf P$ is the orthogonal projector into
$\operatorname{range}(\mathbf B)$, while $\mathbf X_\ell$ permutes the
basis vectors and $\mathbf Z_\ell$ changes the signs. Let $\mathbf I_n$
denote the $n\times n$ identity and set
$\mathbf D:=p(\mathbf I_n-\mathbf P/2)$. The discrete eigenvalue problem is
$\mathbf H\vecop(\mathcal X)=\lambda\vecop(\mathcal X)$, with
$\mathbf H\in\mathbb R^{n^d\times n^d}$ given explicitly by
\begin{equation}\label{eq:cluster-eigenproblem}
  \begin{aligned}[b]
    \mathbf H
    ={}&\sum_{i=1}^{d}
      \mathbf I_n^{\otimes(d-i)}\otimes\mathbf D\otimes
      \mathbf I_n^{\otimes(i-1)}\\
    &-\frac12\sum_{\ell=1}^{p}\sum_{i=2}^{d-1}
      \mathbf I_n^{\otimes(d-i-1)}\otimes
      \mathbf Z_\ell\otimes\mathbf X_\ell\otimes\mathbf Z_\ell\otimes
      \mathbf I_n^{\otimes(i-2)}\\
    &-\frac12\sum_{\ell=1}^{p}\left(
      \mathbf I_n^{\otimes(d-2)}\otimes
      \mathbf Z_\ell\otimes\mathbf X_\ell
      +\mathbf X_\ell\otimes\mathbf Z_\ell\otimes
      \mathbf I_n^{\otimes(d-2)}\right).
  \end{aligned}
\end{equation}
The double sum couples three consecutive modes, and the final sum
contains the two boundary couplings. The factors follow the
reverse mode order in~\eqref{eq:kron-operator}. Thus
$\mathbf H$ is a sum of $d(p+1)$ Kronecker products and can be applied
directly to TT cores.

The smallest eigenvalue of~\eqref{eq:cluster-eigenproblem} is
$\lambda_*=0$. In physical terminology, $\mathcal X_*$ represents the ground state, and the corresponding normalized eigenvector is
\[
  \vecop(\mathcal X_*)=
  q^{-d/2}\!\sum_{\mathbf s_1,\ldots,\mathbf s_d\in\{0,1\}^p}
  (-1)^{\sum_{i=1}^{d-1}\mathbf s_i\cdot\mathbf s_{i+1}}
  \mathbf b_{\mathbf s_d}\otimes\cdots\otimes\mathbf b_{\mathbf s_1}.
\]
The restriction of $\mathbf H$ to
$\operatorname{range}(\mathbf B)^{\otimes d}$ is orthogonally similar to a
Kronecker sum of $pd$ copies of $\diag(0,1)$. The restriction to the
orthogonal complement of this subspace has all eigenvalues at least $p$.
Consequently, $\mathbf H$ is positive semidefinite, its zero eigenvalue is
simple, and its smallest positive eigenvalue is one. 
For every $k=1,\ldots,d-1$, the unfolding
$\mathbf X_*^{\langle k\rangle}$ has exactly $q$ nonzero singular values,
all equal to $q^{-1/2}$. These values are preserved by the orthonormal
matrix $\mathbf B$, hence
\[
  \rankTT(\mathcal X_*)=(1,q,\ldots,q,1).
\]
A single layer therefore gives TT rank $(1,2,\ldots,2,1)$, and
two layers give TT rank $(1,4,\ldots,4,1)$. With rank parameter
$(1,r,\ldots,r,1)$, the exact eigenvector lies outside the feasible set
when $r<q$ and is representable when $r\geq q$. The rank-parameter sweep
tests the solvers on both sides of this threshold.

The reported rank sweep uses $d=4$, $n_k=10$, and $p=2$, so the
exact TT ranks are $(1,4,4,4,1)$. We vary $r$ over $1,\ldots,8$
and compare LTT-RCG with LTT-RGD, starting both methods from the
same random TT tensor at each rank parameter.
 Since
$\lambda_*=0$, we report the absolute energy error
$|\rho(\mathcal X)-\lambda_*|$ together with the residual
in~\eqref{eq:eigen-residual}.
\Cref{fig:cluster-two-layers} shows the final eigenvalue errors and residuals.
The accuracy improves sharply at $r=4$, when the exact eigenvector
becomes representable. For $r=1,2,3$, both methods retain nonzero
eigenvalue errors and residuals of at least $0.70$. For $r\geq4$,
LTT-RCG attains eigenvalue errors below $5\times10^{-14}$ and satisfies
the residual tolerance. LTT-RGD reaches eigenvalue errors below
$4\times10^{-13}$ and satisfies the residual tolerance for most rank parameters.
For $r>4$, the exact eigenvector has a TT rank smaller than the prescribed rank parameter. The proposed parametrization enables such lower-rank tensors to be handled within a smooth Riemannian framework. The TT representation alone provides a compact tensor format but does not resolve the nonsmoothness of the bounded-rank tensor space. By representing tensors of different TT ranks on a single smooth manifold, the parametrization allows the method to proceed even when the numerical TT rank is lower than the prescribed rank parameter.
\begin{figure}[htbp]
\centering
\includegraphics[width=0.6\linewidth]{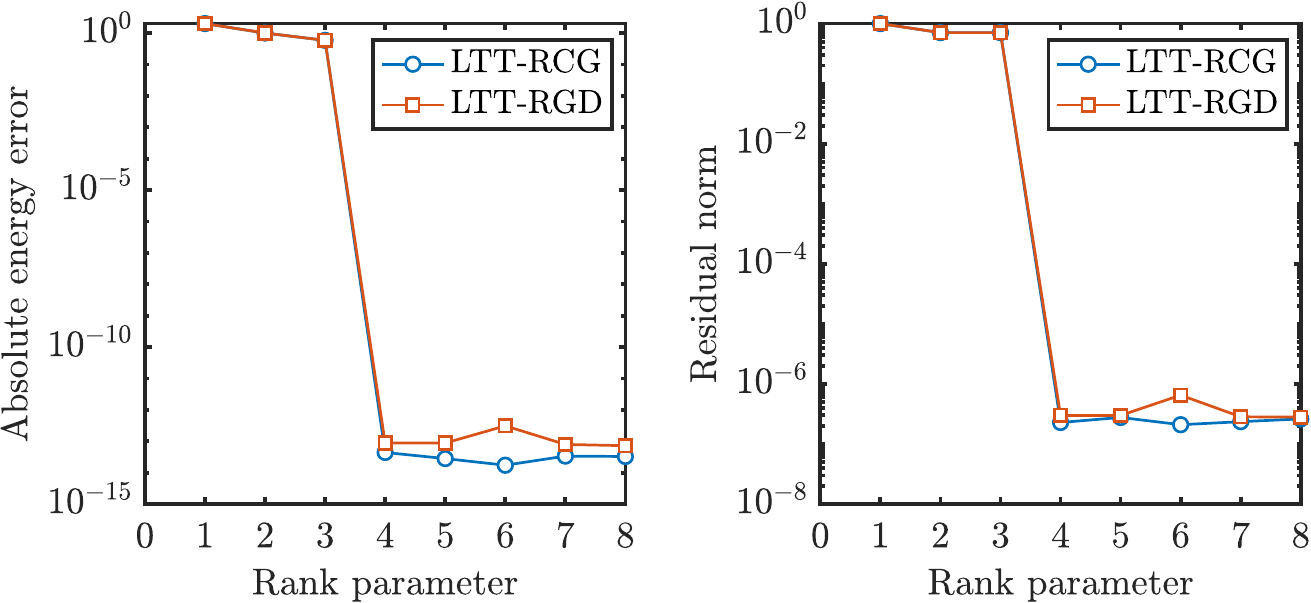}
\caption{Final absolute eigenvalue error and normalized residual for the
two-layer cluster problem with $d=4$, $n_k=10$, and exact TT ranks
$(1,4,4,4,1)$.}
\label{fig:cluster-two-layers}
\end{figure}

\subsection{Nonlinear eigenvalue problem}\label{sec:exp-nonlinear}
The Gross--Pitaevskii equation describes a dilute Bose--Einstein
condensate in the mean-field approximation~\cite{Henning2025}.
 Its density-dependent
potential leads to a nonlinear eigenvalue problem, for which we test
LTT-RCG as an inner eigensolver. Consider
\begin{equation*}
  \left(-\Delta+|x|^2+\epsilon|\phi|^2\right)\phi=\mu\phi,
  \qquad
  \int_{\mathbb R^d}|\phi(x)|^2\,\mathrm dx=1,
\end{equation*}
where $\epsilon\geq0$ is the interaction strength and $\mu$ is the chemical
potential. A frozen-potential fixed-point iteration is used. Given
$\phi^{(k)}$ and its density
$\varrho^{(k)}:=|\phi^{(k)}|^2$, the next iterate is the normalized ground
state of the linear Hamiltonian
\begin{equation*}
  \mathscr H^{(k)}:=-\Delta+|x|^2+\epsilon\varrho^{(k)},
  \qquad
  \phi^{(k+1)}\in
  \argmin_{\|\phi\|_{L^2}=1}
  \langle\phi,\mathscr H^{(k)}\phi\rangle_{L^2}.
\end{equation*}
The inner ground-state problem is solved by LTT-RCG, after which the density
is updated by $\varrho^{(k+1)}=|\phi^{(k+1)}|^2$.

The domain $[-3,3]^3$ is discretized using $50$ points per mode, with
$\epsilon=1$ and TT-rank parameter $(1,1,1,1)$.
We use uniform-grid quadrature with spacing $h$ and normalize each
iterate to satisfy $h^3\|\mathcal X\|_{\mathrm F}^2=1$.
The first inner iteration uses zero density and a
random TT initial guess obtained from independent standard normal
core entries, followed by right orthogonalization and normalization.
Subsequent solves use the preceding normalized solution both to
construct the frozen potential and as the initial guess.
The outer iteration stops when the relative density change
$\delta_k$, defined below, is less than $10^{-8}$ or after ten steps.
The nonlinear density and its pointwise action are evaluated as full
arrays, while the linear part retains its Kronecker-sum representation.
This test therefore assesses the eigensolver within the nonlinear
iteration, without addressing storage scalability at large tensor orders.

We test the outer iteration through the chemical potential, energy,
and relative density change.
The corresponding quantities are
\begin{align*}
  \mu_k
  &:=\langle\phi^{(k)},(-\Delta+|x|^2)\phi^{(k)}\rangle_{L^2}
    +\epsilon\int_{\mathbb R^d}|\phi^{(k)}|^4\,\mathrm dx,\\
  \mathcal E_k
  &:=\langle\phi^{(k)},(-\Delta+|x|^2)\phi^{(k)}\rangle_{L^2}
    +\frac{\epsilon}{2}\int_{\mathbb R^d}|\phi^{(k)}|^4\,\mathrm dx,\\
  N_k
  &:=\int_{\mathbb R^d}|\phi^{(k)}|^2\,\mathrm dx,\qquad
  \delta_k:=
  \frac{\|\varrho^{(k)}-\varrho^{(k-1)}\|_{L^2}}
       {\|\varrho^{(k-1)}\|_{L^2}}.
\end{align*}
Here $\mathcal E_k$ is the Gross--Pitaevskii energy, $N_k$ is the particle
number, and $\delta_k$ measures the relative density change.
In particular, $\mu_k-\mathcal E_k$ equals one half of the
interaction-energy contribution.
\Cref{fig:bose-einstein-inner} shows the Riemannian gradient norm histories of the inner iterations. The Riemannian gradient norm decreases substantially in the first three subproblems. In the last subproblem, the inner solver terminates after one iteration, yielding zero relative density change and thus satisfying the outer stopping criterion.
\begin{figure}[htbp]
  \centering
\includegraphics[width=\linewidth]{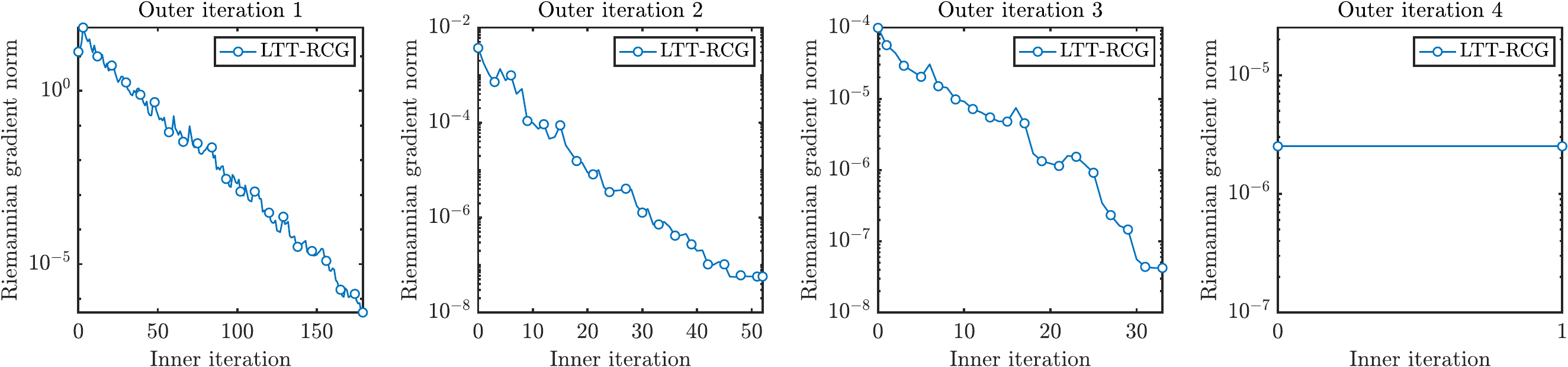}
\caption{Riemannian gradient norm histories of LTT-RCG for the four
successive frozen-potential subproblems in the Bose--Einstein experiment.}
\label{fig:bose-einstein-inner}
\end{figure}

\Cref{fig:bose-einstein} shows the evolution of the chemical potential,
energy, and relative density change over the outer iterations.  After four outer iterations, the density satisfies the stopping criterion, while the chemical potential and energy have also stabilized. These results demonstrate that LTT-RCG can be used as an inner eigensolver within a fixed-point iteration to solve nonlinear eigenvalue problems.
\begin{figure}[htbp]
  \centering
\includegraphics[width=0.90\linewidth]{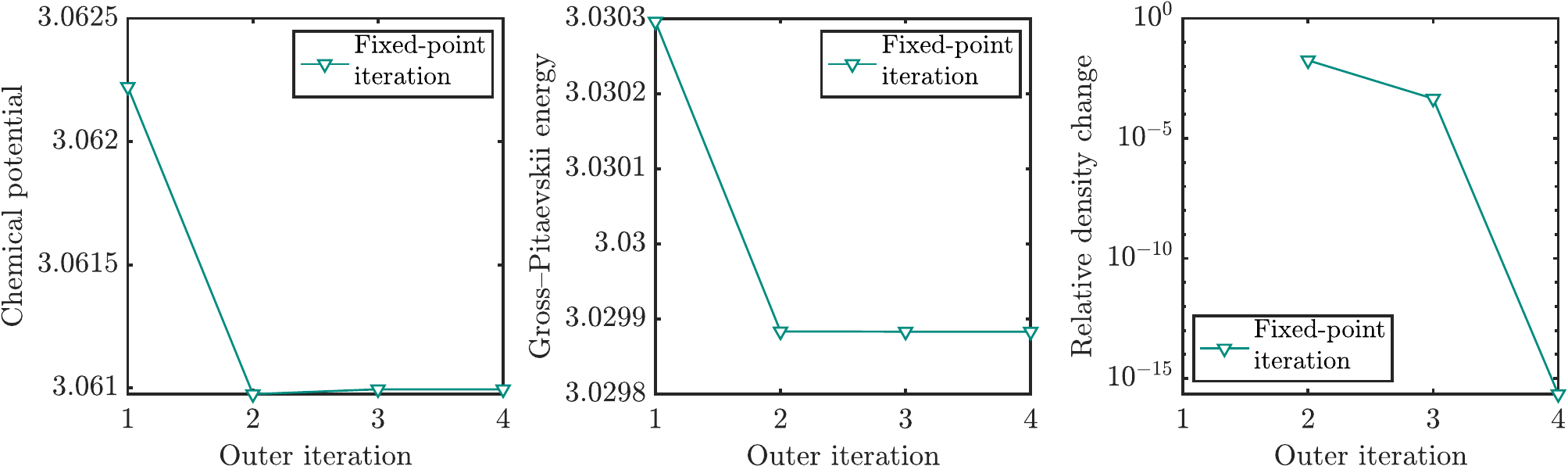}
  \caption{Frozen-potential fixed-point iteration for the three-dimensional
  Bose--Einstein problem with $n=50$, $\epsilon=1$, and TT-rank parameter
  $(1,1,1,1)$.}
  \label{fig:bose-einstein}
\end{figure}

\section{Conclusion}

We have developed LTT-RGD and LTT-RCG for high-dimensional extreme
eigenvalue problems by minimizing the Rayleigh quotient over the low-rank
tensor space with a prescribed TT-rank parameter.
A smooth parametrization manifold represents tensors throughout this space,
including tensors with ranks below the prescribed rank parameter. The tangent-space
representation, projection, vector transport, and retraction provide the
operations needed to optimize directly with TT cores. For operators given
as sums of Kronecker products, these operations avoid forming full tensors
and yield a polynomial computation cost per iteration.

The numerical results demonstrate the computational and storage advantages of the proposed methods. LTT-RCG achieves accuracy comparable to full-space eigensolvers LOBPCG and PRIMME with less computation time on larger problems, while the TT representation allows computations using tensor cores where storing full vectors is impractical. The observed reduction in numerical TT ranks and the accurate solutions obtained with overestimated rank parameters further demonstrate robustness to the choice of rank parameter.

Future work will focus on developing preconditioners to accelerate convergence and adaptive rank-selection strategies to balance approximation accuracy and computational cost. We will also investigate the approximation error introduced by restricting the original eigenvalue problem to a bounded-rank tensor space, aiming to establish how the accuracy of the resulting eigenvalues and eigenvectors depends on the rank parameter.

\section*{Acknowledgments}

This work was supported by the National Key R\&D Program of China
(grant No. 2023YFA1009300). BG was supported by the National Natural Science Foundation of
China (grant No. 12288201). YK was supported by the National Natural Science Foundation of China (grant No. 12201130) and the Natural Science Foundation of Guangdong Province of China (grant No. 2026A1515011703).
\bibliographystyle{elsarticle-num}
\bibliography{references}

\begin{thebibliography}{10}
\expandafter\ifx\csname url\endcsname\relax
  \def\url#1{\texttt{#1}}\fi
\expandafter\ifx\csname urlprefix\endcsname\relax\def\urlprefix{URL }\fi
\expandafter\ifx\csname href\endcsname\relax
  \def\href#1#2{#2} \def\path#1{#1}\fi

\bibitem{RakhubaOseledets2016}
M.~Rakhuba, I.~Oseledets, Calculating vibrational spectra of molecules using tensor train decomposition, The Journal of Chemical Physics 145~(12) (2016) 124101.
\newblock \href {https://doi.org/10.1063/1.4962420} {\path{doi:10.1063/1.4962420}}.

\bibitem{HackbuschKhoromskijSauterTyrtyshnikov2012}
W.~Hackbusch, B.~N. Khoromskij, S.~Sauter, E.~E. Tyrtyshnikov, Use of tensor formats in elliptic eigenvalue problems, Numerical Linear Algebra with Applications 19~(1) (2012) 133--151.
\newblock \href {https://doi.org/10.1002/nla.793} {\path{doi:10.1002/nla.793}}.

\bibitem{johnson1974beam}
C.~P. Johnson, K.~M. Will, Beam buckling by finite element procedure, Journal of the Structural Division 100~(3) (1974) 669--685.

\bibitem{KERNER19891}
W.~Kerner, \href{https://www.sciencedirect.com/science/article/pii/0021999189902003}{Large-scale complex eigenvalue problems}, Journal of Computational Physics 85~(1) (1989) 1--85.
\newblock \href {https://doi.org/10.1016/0021-9991(89)90200-3} {\path{doi:10.1016/0021-9991(89)90200-3}}.
\newline\urlprefix\url{https://www.sciencedirect.com/science/article/pii/0021999189902003}

\bibitem{Lehoucq2001}
R.~B. Lehoucq, A.~G. Salinger, \href{https://onlinelibrary.wiley.com/doi/abs/10.1002/fld.135}{Large-scale eigenvalue calculations for stability analysis of steady flows on massively parallel computers}, International Journal for Numerical Methods in Fluids 36~(3) (2001) 309--327.
\newblock \href {https://doi.org/10.1002/fld.135} {\path{doi:10.1002/fld.135}}.
\newline\urlprefix\url{https://onlinelibrary.wiley.com/doi/abs/10.1002/fld.135}

\bibitem{Oseledets2011}
I.~V. Oseledets, \href{https://doi.org/10.1137/090752286}{Tensor-train decomposition}, SIAM Journal on Scientific Computing 33~(5) (2011) 2295--2317.
\newblock \href {https://doi.org/10.1137/090752286} {\path{doi:10.1137/090752286}}.
\newline\urlprefix\url{https://doi.org/10.1137/090752286}

\bibitem{PhysRevB.73.094423}
F.~Verstraete, J.~I. Cirac, \href{https://link.aps.org/doi/10.1103/PhysRevB.73.094423}{Matrix product states represent ground states faithfully}, Phys. Rev. B 73 (2006) 094423.
\newblock \href {https://doi.org/10.1103/PhysRevB.73.094423} {\path{doi:10.1103/PhysRevB.73.094423}}.
\newline\urlprefix\url{https://link.aps.org/doi/10.1103/PhysRevB.73.094423}

\bibitem{steinlechner2016}
M.~Steinlechner, \href{http://epubs.siam.org/doi/10.1137/15M1010506}{Riemannian {{Optimization}} for {{High-Dimensional Tensor Completion}}}, SIAM Journal on Scientific Computing 38~(5) (2016) S461--S484.
\newblock \href {https://doi.org/10.1137/15M1010506} {\path{doi:10.1137/15M1010506}}.
\newline\urlprefix\url{http://epubs.siam.org/doi/10.1137/15M1010506}

\bibitem{RakhubaNovikovOseledets2019}
M.~Rakhuba, A.~Novikov, I.~Oseledets, Low-rank {Riemannian} eigensolver for high-dimensional {Hamiltonians}, Journal of Computational Physics 396 (2019) 718--737.
\newblock \href {https://doi.org/10.1016/j.jcp.2019.07.003} {\path{doi:10.1016/j.jcp.2019.07.003}}.

\bibitem{holtz2012}
S.~Holtz, T.~Rohwedder, R.~Schneider, \href{http://link.springer.com/10.1007/s00211-011-0419-7}{On manifolds of tensors of fixed {{TT-rank}}}, Numerische Mathematik 120~(4) (2012) 701--731.
\newblock \href {https://doi.org/10.1007/s00211-011-0419-7} {\path{doi:10.1007/s00211-011-0419-7}}.
\newline\urlprefix\url{http://link.springer.com/10.1007/s00211-011-0419-7}

\bibitem{Anderson2010}
C.~R. Anderson, A {Rayleigh--Chebyshev} procedure for finding the smallest eigenvalues and associated eigenvectors of large sparse {Hermitian} matrices, Journal of Computational Physics 229~(19) (2010) 7477--7487.
\newblock \href {https://doi.org/10.1016/j.jcp.2010.06.030} {\path{doi:10.1016/j.jcp.2010.06.030}}.

\bibitem{Knyazev2001}
A.~V. Knyazev, Toward the optimal preconditioned eigensolver: Locally optimal block preconditioned conjugate gradient method, SIAM Journal on Scientific Computing 23~(2) (2001) 517--541.
\newblock \href {https://doi.org/10.1137/S1064827500366124} {\path{doi:10.1137/S1064827500366124}}.

\bibitem{StathopoulosMcCombs2010}
A.~Stathopoulos, J.~R. McCombs, {PRIMME}: {PReconditioned} iterative {MultiMethod} eigensolver--methods and software description, ACM Transactions on Mathematical Software 37~(2) (2010) 21:1--21:30.
\newblock \href {https://doi.org/10.1145/1731022.1731031} {\path{doi:10.1145/1731022.1731031}}.

\bibitem{Schollwoeck2011}
U.~Schollw{\"o}ck, The density-matrix renormalization group in the age of matrix product states, Annals of Physics 326~(1) (2011) 96--192.
\newblock \href {https://doi.org/10.1016/j.aop.2010.09.012} {\path{doi:10.1016/j.aop.2010.09.012}}.

\bibitem{White1992}
S.~R. White, Density matrix formulation for quantum renormalization groups, Physical Review Letters 69~(19) (1992) 2863--2866.
\newblock \href {https://doi.org/10.1103/PhysRevLett.69.2863} {\path{doi:10.1103/PhysRevLett.69.2863}}.

\bibitem{HaegemanLubichOseledetsVandereyckenVerstraete2016}
J.~Haegeman, C.~Lubich, I.~Oseledets, B.~Vandereycken, F.~Verstraete, Unifying time evolution and optimization with matrix product states, Physical Review B 94~(16) (2016) 165116.
\newblock \href {https://doi.org/10.1103/PhysRevB.94.165116} {\path{doi:10.1103/PhysRevB.94.165116}}.

\bibitem{KressnerTobler2011}
D.~Kressner, C.~Tobler, Preconditioned low-rank methods for high-dimensional elliptic {PDE} eigenvalue problems, Computational Methods in Applied Mathematics 11~(3) (2011) 363--381.
\newblock \href {https://doi.org/10.2478/cmam-2011-0020} {\path{doi:10.2478/cmam-2011-0020}}.

\bibitem{Lebedeva2011}
O.~S. Lebedeva, Tensor conjugate-gradient-type method for {Rayleigh} quotient minimization in block {QTT}-format, Russian Journal of Numerical Analysis and Mathematical Modelling 26~(5) (2011).
\newblock \href {https://doi.org/10.1515/rjnamm.2011.026} {\path{doi:10.1515/rjnamm.2011.026}}.

\bibitem{HoltzRohwedderSchneider2012ALS}
S.~Holtz, T.~Rohwedder, R.~Schneider, The alternating linear scheme for tensor optimization in the tensor train format, SIAM Journal on Scientific Computing 34~(2) (2012) A683--A713.
\newblock \href {https://doi.org/10.1137/100818893} {\path{doi:10.1137/100818893}}.

\bibitem{DolgovKhoromskijOseledetsSavostyanov2014}
S.~V. Dolgov, B.~N. Khoromskij, I.~V. Oseledets, D.~V. Savostyanov, Computation of extreme eigenvalues in higher dimensions using block tensor train format, Computer Physics Communications 185~(4) (2014) 1207--1216.
\newblock \href {https://doi.org/10.1016/j.cpc.2013.12.017} {\path{doi:10.1016/j.cpc.2013.12.017}}.

\bibitem{LeeCichocki2015}
N.~Lee, A.~Cichocki, Estimating a few extreme singular values and vectors for large-scale matrices in tensor train format, SIAM Journal on Matrix Analysis and Applications 36~(3) (2015) 994--1014.
\newblock \href {https://doi.org/10.1137/140983410} {\path{doi:10.1137/140983410}}.

\bibitem{Oseledets2011DMRG}
I.~Oseledets, {DMRG} approach to fast linear algebra in the {TT}-format, Computational Methods in Applied Mathematics 11~(3) (2011) 382--393.
\newblock \href {https://doi.org/10.2478/cmam-2011-0021} {\path{doi:10.2478/cmam-2011-0021}}.

\bibitem{KressnerSteinlechnerUschmajew2014}
D.~Kressner, M.~Steinlechner, A.~Uschmajew, Low-rank tensor methods with subspace correction for symmetric eigenvalue problems, SIAM Journal on Scientific Computing 36~(5) (2014) A2346--A2368.
\newblock \href {https://doi.org/10.1137/130949919} {\path{doi:10.1137/130949919}}.

\bibitem{peng2025}
R.~Peng, C.~Zhu, B.~Gao, X.~Wang, Y.-x. Yuan, \href{http://arxiv.org/abs/2511.04369}{Normalized tensor train decomposition} (2025).
\newblock \href {http://arxiv.org/abs/2511.04369} {\path{arXiv:2511.04369}}.
\newline\urlprefix\url{http://arxiv.org/abs/2511.04369}

\bibitem{DEKTOR2026286}
A.~Dektor, P.~DelMastro, E.~Ye, R.~{Van Beeumen}, C.~Yang, \href{https://www.sciencedirect.com/science/article/pii/S0024379526001801}{Inexact subspace projection methods for low-rank tensor eigenvalue problems}, Linear Algebra and its Applications 743 (2026) 286--321.
\newblock \href {https://doi.org/10.1016/j.laa.2026.04.021} {\path{doi:10.1016/j.laa.2026.04.021}}.
\newline\urlprefix\url{https://www.sciencedirect.com/science/article/pii/S0024379526001801}

\bibitem{BachmayrKramerPfeffer2026}
M.~Bachmayr, S.~Kr{\"a}mer, M.~Pfeffer, \href{https://arxiv.org/abs/2604.16118}{Low-rank eigenvalue solvers for block-sparse matrix product states} (2026).
\newblock \href {http://arxiv.org/abs/2604.16118} {\path{arXiv:2604.16118}}.
\newline\urlprefix\url{https://arxiv.org/abs/2604.16118}

\bibitem{AbsilMahonySepulchre+2008}
P.-A. Absil, R.~Mahony, R.~Sepulchre, \href{https://doi.org/10.1515/9781400830244}{Optimization Algorithms on Matrix Manifolds}, Princeton University Press, Princeton, 2008.
\newblock \href {https://doi.org/10.1515/9781400830244} {\path{doi:10.1515/9781400830244}}.
\newline\urlprefix\url{https://doi.org/10.1515/9781400830244}

\bibitem{boumal2020introduction}
N.~Boumal, An Introduction to Optimization on Smooth Manifolds, Cambridge University Press, Cambridge, 2023.
\newblock \href {https://doi.org/10.1017/9781009166164} {\path{doi:10.1017/9781009166164}}.

\bibitem{KUTSCHAN2018370}
B.~Kutschan, \href{https://www.sciencedirect.com/science/article/pii/S0024379518300181}{Tangent cones to tensor train varieties}, Linear Algebra and its Applications 544 (2018) 370--390.
\newblock \href {https://doi.org/10.1016/j.laa.2018.01.012} {\path{doi:10.1016/j.laa.2018.01.012}}.
\newline\urlprefix\url{https://www.sciencedirect.com/science/article/pii/S0024379518300181}

\bibitem{gao2024}
B.~Gao, R.~Peng, Y.-x. Yuan, \href{https://arxiv.org/abs/2411.14093}{Desingularization of bounded-rank tensor sets} (2024).
\newblock \href {http://arxiv.org/abs/2411.14093} {\path{arXiv:2411.14093}}.
\newline\urlprefix\url{https://arxiv.org/abs/2411.14093}

\bibitem{Uschmajew2020}
A.~Uschmajew, B.~Vandereycken, \href{https://doi.org/10.1007/978-3-030-31351-7_9}{Geometric methods on low-rank matrix and tensor manifolds}, in: Handbook of Variational Methods for Nonlinear Geometric Data, Springer International Publishing, Cham, 2020, pp. 261--313.
\newblock \href {https://doi.org/10.1007/978-3-030-31351-7_9} {\path{doi:10.1007/978-3-030-31351-7_9}}.
\newline\urlprefix\url{https://doi.org/10.1007/978-3-030-31351-7_9}

\bibitem{Valentin2018}
V.~Khrulkov, I.~Oseledets, \href{https://doi.org/10.1137/16M1108194}{Desingularization of bounded-rank matrix sets}, SIAM Journal on Matrix Analysis and Applications 39~(1) (2018) 451--471.
\newblock \href {https://doi.org/10.1137/16M1108194} {\path{doi:10.1137/16M1108194}}.
\newline\urlprefix\url{https://doi.org/10.1137/16M1108194}

\bibitem{Quentin2024}
Q.~Rebjock, N.~Boumal, \href{https://arxiv.org/abs/2406.14211}{Optimization over bounded-rank matrices through a desingularization enables joint global and local guarantees}, to appear in SIAM Journal on Optimization (2024).
\newblock \href {http://arxiv.org/abs/2406.14211} {\path{arXiv:2406.14211}}.
\newline\urlprefix\url{https://arxiv.org/abs/2406.14211}

\bibitem{Yang2025}
Y.~Yang, B.~Gao, Y.-x. Yuan, \href{https://doi.org/10.1007/s10107-026-02331-7}{A space-decoupling framework for optimization on bounded-rank matrices with orthogonally invariant constraints}, Mathematical Programming (2026).
\newblock \href {https://doi.org/10.1007/s10107-026-02331-7} {\path{doi:10.1007/s10107-026-02331-7}}.
\newline\urlprefix\url{https://doi.org/10.1007/s10107-026-02331-7}

\bibitem{Duan2026}
H.~Duan, Z.~Chen, N.~Wong, \href{http://arxiv.org/abs/2603.12026}{Efficient {{Generative Modeling}} with {{Unitary Matrix Product States Using Riemannian Optimization}}} (2026).
\newblock \href {http://arxiv.org/abs/2603.12026} {\path{arXiv:2603.12026}}.
\newline\urlprefix\url{http://arxiv.org/abs/2603.12026}

\bibitem{Pippan2010}
P.~Pippan, S.~R. White, H.~G. Evertz, Efficient matrix-product state method for periodic boundary conditions, Physical Review B 81~(8) (2010) 081103.
\newblock \href {https://doi.org/10.1103/PhysRevB.81.081103} {\path{doi:10.1103/PhysRevB.81.081103}}.

\bibitem{UschmajewVandereycken2013}
A.~Uschmajew, B.~Vandereycken, The geometry of algorithms using hierarchical tensors, Linear Algebra and its Applications 439~(1) (2013) 133--166.
\newblock \href {https://doi.org/10.1016/j.laa.2013.03.016} {\path{doi:10.1016/j.laa.2013.03.016}}.

\bibitem{Dong2022CP}
S.~Dong, B.~Gao, Y.~Guan, F.~Glineur, New {Riemannian} preconditioned algorithms for tensor completion via polyadic decomposition, SIAM Journal on Matrix Analysis and Applications 43~(2) (2022) 840--866.
\newblock \href {https://doi.org/10.1137/21M1394734} {\path{doi:10.1137/21M1394734}}.

\bibitem{Gao2024TR}
B.~Gao, R.~Peng, Y.-x. Yuan, {Riemannian} preconditioned algorithms for tensor completion via tensor ring decomposition, Computational Optimization and Applications 88~(2) (2024) 443--468.
\newblock \href {https://doi.org/10.1007/s10589-024-00559-7} {\path{doi:10.1007/s10589-024-00559-7}}.

\bibitem{gaoLowrankOptimizationTucker2025}
B.~Gao, R.~Peng, Y.-x. Yuan, \href{https://doi.org/10.1007/s10107-024-02186-w}{Low-rank optimization on {{Tucker}} tensor varieties}, Mathematical Programming 214~(1--2) (2025) 357--407.
\newblock \href {https://doi.org/10.1007/s10107-024-02186-w} {\path{doi:10.1007/s10107-024-02186-w}}.
\newline\urlprefix\url{https://doi.org/10.1007/s10107-024-02186-w}

\bibitem{HeinEisertBriegel2004}
M.~Hein, J.~Eisert, H.~J. Briegel, Multiparty entanglement in graph states, Physical Review A 69 (2004) 062311.
\newblock \href {https://doi.org/10.1103/PhysRevA.69.062311} {\path{doi:10.1103/PhysRevA.69.062311}}.

\bibitem{CockburnXia2022}
B.~Cockburn, S.~Xia, An adjoint-based super-convergent {Galerkin} approximation of eigenvalues, Journal of Computational Physics 449 (2022) 110816.
\newblock \href {https://doi.org/10.1016/j.jcp.2021.110816} {\path{doi:10.1016/j.jcp.2021.110816}}.

\bibitem{GongLiuWang2024}
W.~Gong, X.~Liu, J.~Wang, Hearing the triangles: A numerical perspective, CSIAM Transactions on Applied Mathematics 5~(1) (2024) 58--72.
\newblock \href {https://doi.org/10.4208/csiam-am.SO-2023-0027} {\path{doi:10.4208/csiam-am.SO-2023-0027}}.

\bibitem{CancesKemlinLevitt2024}
E.~Canc{\`e}s, G.~Kemlin, A.~Levitt, A priori error analysis of linear and nonlinear periodic {Schr{\"o}dinger} equations with analytic potentials, Journal of Scientific Computing 98 (2024) 25.
\newblock \href {https://doi.org/10.1007/s10915-023-02421-0} {\path{doi:10.1007/s10915-023-02421-0}}.

\bibitem{Henning2025}
P.~Henning, E.~Jarlebring, The {Gross--Pitaevskii} equation and eigenvector nonlinearities: Numerical methods and algorithms, SIAM Review 67~(2) (2025) 256--317.
\newblock \href {https://doi.org/10.1137/22M1516324} {\path{doi:10.1137/22M1516324}}.

\end{thebibliography}

\end{document}